\documentclass[11pt,a4paper]{article}

\usepackage{epsf,epsfig,amsfonts,amsgen,amsmath,amstext,amsbsy,amsopn,amsthm,cases,listings,color
}
\usepackage{ebezier,eepic}
\usepackage{color}
\usepackage{multirow}
\usepackage{epstopdf}
\usepackage{graphicx}   
\usepackage{pgf,tikz}
\usepackage{mathrsfs}
\usepackage[marginal]{footmisc}
\usepackage{enumitem}
\usepackage[titletoc]{appendix}
\usepackage{booktabs}
\usepackage{url}
\usepackage{mathtools}

\usepackage{pgfplots}
\usepackage{authblk}
\usepackage{amssymb}

\usepackage{wasysym}

\usepackage{empheq}

\usepackage{dsfont}

\usepackage{tikz}
\usepackage{longtable}
\usepackage{subfigure}
\pgfplotsset{compat=1.18}
\usepackage{mathrsfs}
\usepackage{wasysym} 
\usetikzlibrary{arrows}
\usepackage{aligned-overset}
\usepackage{bm}%for vector
\usepackage{bbm}%for vector
\usepackage[T1]{fontenc}%for polish hook
\usepackage[backref=page]{hyperref}
\allowdisplaybreaks[2]

\definecolor{uuuuuu}{rgb}{0.27,0.27,0.27}
\definecolor{sqsqsq}{rgb}{0.1255,0.1255,0.1255}

\newtheorem{definition}{Definition} [section]
\newtheorem{theorem}[definition]{Theorem}
\newtheorem{lemma}[definition]{Lemma}
\newtheorem{proposition}[definition]{Proposition}

\newtheorem{corollary}[definition]{Corollary}
\newtheorem{conjecture}[definition]{Conjecture}
\newtheorem{claim}[definition]{Claim}
\newtheorem{problem}[definition]{Problem}

\newtheorem{fact}[definition]{Fact}

\newenvironment{pf}{\noindent{\bf Proof}}

\newcommand{\norm}[1]{\left\lVert#1\right\rVert}%for norm

\tikzset{unlabeled_vertex/.style={inner sep=1.7pt, outer sep=0pt, circle, fill}} 
\tikzset{labeled_vertex/.style={inner sep=2.2pt, outer sep=0pt, rectangle, fill=yellow, draw=black}} 
\tikzset{edge_color0/.style={color=black,line width=1.2pt,opacity=0.5}} 
\tikzset{edge_color1/.style={color=red,  line width=1.2pt,opacity=1}} 
\tikzset{edge_color2/.style={color=blue, line width=1.2pt,opacity=1}} 
\tikzset{edge_color3/.style={color=green,line width=1.2pt}} 
\tikzset{edge_color4/.style={color=red,  line width=1.2pt,dotted}} 
\tikzset{edge_color5/.style={color=blue, line width=1.2pt,dotted}} 
\tikzset{edge_color6/.style={color=green, line width=1.2pt,dotted}} 
\tikzset{edge_color7/.style={color=orange, line width=1.2pt}} 
\tikzset{edge_color8/.style={color=gray, line width=1.2pt}} 
\tikzset{edge_thin/.style={color=black}} 
\tikzset{edge_hidden/.style={color=black,dotted,opacity=0}} 
\tikzset{vertex_color1/.style={inner sep=1.7pt, outer sep=0pt, draw, circle, fill=red}} 
\tikzset{vertex_color2/.style={inner sep=1.7pt, outer sep=0pt, draw, circle, fill=blue}} 
\tikzset{vertex_color3/.style={inner sep=1.7pt, outer sep=0pt, draw, circle, fill=green}} 
\tikzset{labeled_vertex_color1/.style={inner sep=2.2pt, outer sep=0pt, draw, rectangle, fill=red}} 
\tikzset{labeled_vertex_color2/.style={inner sep=2.2pt, outer sep=0pt, draw, rectangle, fill=blue}} 
\tikzset{labeled_vertex_color3/.style={inner sep=2.2pt, outer sep=0pt, draw, rectangle, fill=green}}

\tikzset{
vtx/.style={inner sep=1.1pt, outer sep=0pt, circle, fill,draw}, 
vtxl/.style={inner sep=1.1pt, outer sep=0pt, rectangle, fill=yellow,draw=black}, 
hyperedge/.style={fill=pink,opacity=0.5,draw=black}, 
}

\begin{document}
%%%%%%%%%%%%%%%%%%%%%%%%%%%%%%%%%%%%%%%%%%%%%%%%%%%%%%%
\title{\bf\Large Nondegenerate Tur\'{a}n problems under $(t,p)$-norms}
\date{\today}
%%%%%%%%%%%%%%%%%%%%%%%%%%%%%%%%%%%%%%%%%%%%%%%%%%%%%%%%
\author[1,2]{Wanfang Chen\thanks{Research was supported by China Scholarship Council (CSC)~No.~202306140125. \\ Email: \texttt{52215500039@stu.ecnu.edu.cn}}}
\author[3]{Daniel I\v{l}kovi\v{c}\thanks{Research was supported by the MUNI Award in Science and Humanities (MUNI/I/1677/2018) of the Grant Agency of Masaryk University. Email: \texttt{493343@mail.muni.cz}}}
\author[1]{Jared Le\'{o}n\thanks{Research was supported by the Warwick Mathematics Institute Centre for Doctoral Training. \\ Email: \texttt{jared.leon.m@gmail.com}}}
\author[1]{Xizhi Liu\thanks{Research was supported by ERC Advanced Grant 101020255. Email: \texttt{xizhi.liu.ac@gmail.com}}}
\author[1]{Oleg Pikhurko\thanks{Research was supported by ERC Advanced Grant 101020255. Email: \texttt{o.pikhurko@warwick.ac.uk}}}
%%%%%%%%%%%%%%%%%%%%%%%%%%%%%%%%%%%%%%%%%%%%%%%%%%%%%
\affil[1]{Mathematics Institute and DIMAP,
            University of Warwick, 
            Coventry, CV4 7AL, UK}
\affil[2]{School of Mathematical Sciences and               Shanghai Key Laboratory of PMMP, 
            East China Normal University, 
            Shanghai, 200241, China}
\affil[3]{Faculty of Informatics, 
            Masaryk University, 
            Botanick\'a 68A, 602 00 Brno, 
            Czech Republic.}
%%%%%%%%%%%%%%%%%%%%%%%%%%%%%%%%%%%%%%%%%%%%%%%%%%%
\maketitle
%\footnote{footnote}
%%%%%%%%%%%%%%%%%%%%%%%%%%%%%%%%%%%%%%%%%%%%%%%%%
%%%%%%%%%%%%%%%%%%%%%%%%%%%
\begin{abstract}
Given integers $r > t \ge 1$ and a real number $p > 0$, the $(t,p)$-norm $\norm{\mathcal{H}}_{t,p}$ of an $r$-graph $\mathcal{H}$ is the sum of the $p$-th power of the degrees $d_{\mathcal{H}}(T)$ over all $t$-subsets $T \subset V(\mathcal{H})$. 
We conduct a systematic study of the Tur\'{a}n-type problem of determining $\mathrm{ex}_{t,p}(n,\mathcal{F})$, which is the maximum of $\norm{\mathcal{H}}_{t,p}$ over all $n$-vertex $\mathcal{F}$-free $r$-graphs $\mathcal{H}$. 

We establish several basic properties for the $(t,p)$-norm of $r$-graphs, enabling us to derive general theorems from the recently established framework in~\cite{CL24} that are useful for determining $\mathrm{ex}_{t,p}(n,\mathcal{F})$ and proving the corresponding stability. 
 
We determine the asymptotic value of $\mathrm{ex}_{t,p}(n,H_{F}^{r})$ for all feasible combinations of $(r,t,p)$ and for every graph $F$ with chromatic number greater than $r$, where $H_{F}^{r}$ represents the expansion of $F$. 
In the case where $F$ is edge-critical and $p \ge 1$, we establish strong stability and determine the exact value of $\mathrm{ex}_{t,p}(n,H_{F}^{r})$ for all sufficiently large $n$. 
These results extend the seminal theorems of Erd\H{o}s--Stone--Simonovits, Andr\'{a}sfai--Erd\H{o}s--S\'{o}s, Erd\H{o}s--Simonovits, and a classical theorem of Mubayi. 

For the $3$-uniform generalized triangle $F_5$, we determine the exact value of $\mathrm{ex}_{2,p}(n,F_5)$ for all $p \ge 1$ and its asymptotic value for all $p \in [1/2, 1]\cup \{k^{-1} \colon k \in 6\mathbb{N}^{+}+\{0,2\}\}$. 
This extends old theorems of Bollob\'{a}s, Frankl--F\"{u}redi, and a recent result of Balogh--Clemen--Lidick\'{y}. 
Our proofs utilize results on the graph inducibility problem, Steiner triple systems, and the feasible region problem introduced by Liu--Mubayi.

Our results reveal two interesting phenomena:
\begin{itemize}
    \item The extremal structure for $\mathrm{ex}_{t,p}(n,H_{F}^{r})$ remains consistent for all $p > 0$, whereas the number of extremal structures for $\mathrm{ex}_{t,p}(n,F_5)$ transitions from a single structure to infinitely many as $p$ approaches $0$.  
    \item Both problems exhibit degree-stability when $p \ge 1$, but not when $p <1$. 
\end{itemize}

\medskip

\textbf{Keywords:} hypergraphs, nondegenerate Tur\'{a}n problems, degree powers, stability, vertex-extendability. 

%\medskip

%\textbf{MSC2020:}
%https://mathscinet.ams.org/msc/msc2010.html
\end{abstract}
%%%%%%%%%%%%%%%%%%%%%%%%%%%%%%%%%%%%%%%%%%%%%%%%%%%%%%
%%%%%%%%%%%%%%%%%%%%%%%%%%%%%%%%%%%%%%%
\section{Introduction}\label{SEC:Intorduction}
Given an integer $r\ge 2$, an \textbf{$r$-uniform hypergraph} (henceforth \textbf{$r$-graph}) $\mathcal{H}$ is a collection of $r$-subsets of some finite set $V$.
We identify a hypergraph $\mathcal{H}$ with its edge set and use $V(\mathcal{H})$ to denote its vertex set. 
The size of $V(\mathcal{H})$ is denoted by $v(\mathcal{H})$. 
Given an integer $i \in [r-1]$, 
the \textbf{$i$-th shadow} of $\mathcal{H}$ is 
\begin{align*}
    \partial_{i}\mathcal{H}
    \coloneqq \left\{e\in \binom{V(\mathcal{H})}{r-i} \colon \exists E\in \mathcal{H} \text{ such that } e\subset E\right\}. 
\end{align*}
The \textbf{link} of an $i$-set $T \subset V(\mathcal{H})$ is   
\begin{align*}
    L_{\mathcal{H}}(T)
    \coloneqq \left\{e \in \binom{V(\mathcal{H})}{r-i} \colon T \cup e \in \mathcal{H}\right\}. 
\end{align*}
The \textbf{degree} of $T$ in $\mathcal{H}$ is $d_{\mathcal{H}}(T) \coloneqq |L_{\mathcal{H}}(T)|$. 
% We use $\delta_{i}(\mathcal{H})$, $\Delta_{i}(\mathcal{H})$, and $d_{i}(\mathcal{H})$ to denote the \textbf{minimum, maximum, and average $i$-degree} of $\mathcal{H}$, respectively.
% The subscript $i$ will be omitted in the case where $i = 1$. 
For integers $n \ge r \ge 2$, let $K_{n}^{r}$ denote the complete $r$-graph on $n$ vertices. 
We will omit the superscript $r$ when $r=2$. 

Given a family $\mathcal{F}$ of $r$-graphs, we say $\mathcal{H}$ is \textbf{$\mathcal{F}$-free}
if it does not contain any member of $\mathcal{F}$ as a subgraph.
Studying the extremal properties of $\mathcal{F}$-free $r$-graphs is a central topic in Extremal Combinatorics. 
For example, determining the maximum number of edges in an $n$-vertex $\mathcal{F}$-free $r$-graph is the well-known Tur\'{a}n problem (see~\cite{Fur91,Sid95,Keevash11} for its history and related results), starting with the seminal work of Tur\'{a}n~\cite{TU41}. 
Another example is determining the maximum number of copies of a fixed $r$-graph $Q$ in an $n$-vertex $\mathcal{F}$-free $r$-graph, known as the generalized Tur\'{a}n problem,  which began with the seminal work of Erd\H{o}s~\cite{Erdos62} and was popularized by Alon--Shikhelman~\cite{AS16}. 

In this work, we consider the following Tur\'{a}n-type problem, which is a common generalization of the classical Tur\'{a}n problem and the generalized Tur\'{a}n problem for counting stars (see Section~\ref{SEC:Remark} for related discussions). 
%(what we called the $(t,p)$-norm Tur\'{a}n problem)

Given integers $r > t \ge 1$ and a real number $p> 0$, the \textbf{$(t,p)$-norm}\footnote{We have slightly abused the use of notation here, as it is no longer a norm when $p < 1$.} of an $r$-graph $\mathcal{H}$ is
\begin{align*}
    \norm{\mathcal{H}}_{t,p}
    \coloneqq 
    \sum_{T\in \binom{V(\mathcal{H})}{t}} d_{\mathcal{H}}^{p}(T)
    = \sum_{T\in \partial_{r-t}\mathcal{H}} d_{\mathcal{H}}^{p}(T), 
\end{align*}
where $d_{\mathcal{H}}^{p}(T)\coloneqq \left(d_{\mathcal{H}}(T)\right)^{p}$. 
The \textbf{$(t,p)$-degree} of a vertex $v\in V(\mathcal{H})$ is 
\begin{align*}
    d_{\mathcal{H},t,p}(v)
    \coloneqq \norm{\mathcal{H}}_{t,p} - \norm{\mathcal{H}-v}_{t,p}, 
\end{align*}
where $\mathcal{H} - v$ denotes the $r$-graph obtained from $\mathcal{H}$ by removing $v$ and all edges containing $v$.
We use $\delta_{t,p}(\mathcal{H})$, $\Delta_{t,p}(\mathcal{H})$, and $d_{t,p}(\mathcal{H})$
to denote the \textbf{minimum, maximum}, and \textbf{average $(t,p)$-degree} of $\mathcal{H}$, respectively. 
Observe that $\norm{\mathcal{H}}_{t,1} = \binom{r}{t}\cdot |\mathcal{H}|$ and $\norm{\mathcal{H}}_{t,0} = |\partial_{r-t}\mathcal{H}|$. 

Given a family $\mathcal{F}$ of $r$-graphs, the \textbf{$(t,p)$-Tur\'{a}n number} of $\mathcal{F}$ is 
\begin{align*}
    \mathrm{ex}_{t,p}(n,\mathcal{F})
    \coloneqq \max\left\{\norm{\mathcal{H}}_{t,p} \colon \text{$v(\mathcal{H}) = n$ and $\mathcal{H}$ is $\mathcal{F}$-free}\right\},  
\end{align*}
%Similar to the ordinary Tur\'{a}n problem, 
and the \textbf{$(t,p)$-Tur\'{a}n density} (whose existence will be established in Proposition~\ref{PROP:limit-exist}) of $\mathcal{F}$ is   
\begin{align*}
    \pi_{t,p}(\mathcal{F})
    \coloneqq \lim_{n\to \infty}\frac{\mathrm{ex}_{t,p}(n,\mathcal{F})}{\binom{n}{t}(n-t)^{p(r-t)}}
    = \lim_{n\to \infty}\frac{t! \cdot \mathrm{ex}_{t,p}(n,\mathcal{F})}{n^{t+p(r-t)}}. 
\end{align*}
For convenience, we define the \textbf{$(t,p)$-extremal degree} of $\mathcal{F}$ as follows$\colon$ 
\begin{align*}
    \mathrm{exdeg}_{t,p}(n,\mathcal{F}) 
    \coloneqq 
    \frac{(t+p(r-t)) \cdot \mathrm{ex}_{t,p}(n,\mathcal{F})}{n}. 
\end{align*}
These notations are extensions of the well-known \textbf{Tur\'{a}n number} $\mathrm{ex}(n,\mathcal{F})$ and \textbf{Tur\'{a}n density} $\pi(\mathcal{F})$, since $\mathrm{ex}(n,\mathcal{F}) = \mathrm{ex}_{t,1}(n,\mathcal{F})/\binom{r}{t}$ and $\pi(\mathcal{F}) = \pi_{t,1}(\mathcal{F})/\binom{r}{t}$. 
We say that a family $\mathcal{F}$ is \textbf{$(t,p)$-nondegenerate} (resp. nondegenerate) if $\pi_{t,p}(\mathcal{F}) > 0$ (resp. $\pi(\mathcal{F}) > 0$).
Using a theorem of Erd\H{o}s~\cite{E64}, it is not hard to show that for every $r > t \ge 1$ and $p > 0$, a family $\mathcal{F}$ of $r$-graphs is $(t,p)$-nondegenerate iff it is nondegenerate. 

The study of $\mathrm{ex}_{1,p}(n,\mathcal{F})$ for graph families $\mathcal{F}$ was initiated by Caro--Yuster~\cite{CY00} who extended the seminal Tur\'{a}n Theorem~\cite{TU41} by determining $\mathrm{ex}_{1,p}(n,K_{\ell+1})$ for all $\ell \ge 2$ and $p > 1$. 
Numerous results concerning graphs were subsequently obtained by various researchers (see e.g.~\cite{BN00,PT05,Nik09,BN12,GLS15,Ger24}). 
The study of $\mathrm{ex}_{r-1,2}(n,\mathcal{F})$ for $r$-graph families $\mathcal{F}$ with $r \ge 3$ was initiated very recently by Balogh--Clemen--Lidick\'{y}~\cite{BCL22,BCL22b}. 
They determined the asymptotic values of $\mathrm{ex}_{2,2}(n,K_{4}^{3})$ and $\mathrm{ex}_{2,2}(n,K_{5}^{3})$ among many results, utilizing computer-assisted flag algebra computations~\cite{Raz07}. 
These results are particularly interesting given the notorious difficulty of determining $\mathrm{ex}(n,K_{\ell+1}^{3})$ for any $\ell \ge 3$, a problem raised by Tur\'{a}n~\cite{TU41}. 

In this work, we undertake a systematic study of  $\mathrm{ex}_{t,p}(n,\mathcal{F})$ for $r > t \ge 1$ and $p > 0$. 
We establish several basic properties concerning the $(t,p)$-norm of $r$-graphs (Section~\ref{SEC:General-property}), and adapt the recently established framework from~\cite{CL24} to $(t,p)$-norm Tur\'{a}n problems (Theorems~\ref{THM:Lp-general-a} and~\ref{THM:Lp-general-b}). 
We present applications of the general theorems to two classical examples: the generalized triangle $F_5$ (Section~\ref{SEC:intro-generalized-triangle}) and the expansion of graphs (Section~\ref{SEC:intro-expansion}). 
It is worth noting that there are several additional classes of hypergraphs, such as those considered in~\cite{LMR23unif}, to which our general theorems could potentially apply. 
However, these applications typically require extensive and nontrivial calculations (see Sections~\ref{SUBSEC:F5-large-STS} and~\ref{SUBSEC:proof-F5-inequalities} for example), which we did not undertake but leave for interested readers to explore.

%%%%%%%%%%%%%%%%%%%%%%%%%%
\subsection{The generalized triangle}\label{SEC:intro-generalized-triangle}
The first nondegenerate hypergraph for which we know the Tur\'{a}n density is (perhaps) the ($3$-uniform) \textbf{generalized triangle} $F_5$, which is the $5$-vertex $3$-graph with edge set 
\begin{align*}
    \left\{ \{1,2,3\}, \{1,2,4\}, \{3,4,5\} \right\}. 
\end{align*}
The $3$-graph $F_5$ is a classical example in hypergraph Tur\'{a}n problems and has a rich research history (see e.g.~\cite{FF83,Sido87,KM04,KLM14,BIJ17,LIU19,LM21feasible,Liu20a,LMR23unif,HLLYZ23}), beginning with the seminal work of Bollob\'{a}s~\cite{BO74}.    

Using the stability method, Balogh--Clemen--Lidick\'{y}, building on the argument of Bollob\'{a}s, determined $\mathrm{ex}_{2}(n,F_5)$ for large $n$ in~\cite{BCL22b}. 
In the following theorems, we determine $\mathrm{ex}_{2,p}(n,F_5)$ for all real numbers $p \ge 1$ when $n$ is large, and determine $\pi_{2,p}(F_5)$ for every $p \in [1/2,1] \cup \{k^{-1} \colon k \in 6\mathbb{N}^{+}+\{0,2\}\}$. 
Our proof for $p \ge 1$ is based on the recently established framework in~\cite{CL24}, which differs from the approach used by Balogh--Clemen--Lidick\'{y}. Additionally, our result for the case where $p =2$ strengthens the result of Balogh--Clemen--Lidick\'{y}. 
The case where $p \in (0,1)$ is particularly interesting and seems quite challenging to fully resolve. 
The extremal constructions in this case are closely related to Steiner triple systems, and the proof uses results from graph inducibility problems and the feasible region problem introduced in~\cite{LM21feasible}.

\begin{theorem}\label{THM:Lp-F5-p-large}
    For every real number $p \ge 1$ there exist $\varepsilon >0$ and $N_0>0$ such that the following statements hold for every integer $n \ge N_{0}$. 
    \begin{enumerate}[label=(\roman*)]
        \item Every $n$-vertex $F_5$-free $3$-graph with $\delta_{2,p}(\mathcal{H}) \ge (1-\varepsilon)\cdot \mathrm{exdeg}_{2,p}(n,F_5)$ is $3$-partite. 
        \item $\mathrm{ex}_{2,p}(n,F_5) = \norm{\mathcal{G}}_{2,p}$ for some $n$-vertex complete $3$-partite $3$-graph $\mathcal{G}$.  
    \end{enumerate} 
\end{theorem}

Note that Theorem~\ref{THM:Lp-F5-p-large} shows that the extremal construction for $p \ge 1$ is always $3$-partite, but it should be noted that the ratios of these three parts are not necessarily balanced for every $p\ge 1$. In fact, the ratio becomes more unbalanced as $p$ increases. 

For $p \in [1/2, 1]$, we show that the extremal construction is nearly balanced 3-partite and establish the corresponding edge-stability. 
An interesting phenomenon arises$\colon$ unlike the case when $p \ge 1$, the $(2,p)$-Tur\'{a}n problem for $F_5$ is no longer degree-stable when $p < 1$. 
A construction that illustrates this phenomenon is included in Section~\ref{APPENDIX:SEC:construction} of the Appendix for the case\footnote{Constructions for general
$p$ can be obtained easily in a similar manner.} $p = 1/2$. 
% That is, there exists an $F_5$-free $3$-graph $\mathcal{H}$ on $n$ vertices with $\delta_{2,p}(\mathcal{H}) = (1-o(1))\cdot \mathrm{degex}(n,F_5)$ that is not $3$-partite. 
% The construction is included in . 

\begin{theorem}\label{THM:Lp-F5-p-small-a}
    Suppose that $p$ is a real number in $[1/2, 1]$. 
    Then $\pi_{2,p}(F_5) = \frac{2}{3^{1+p}}$. 
    Moreover, for every $\delta>0$, there exist $\varepsilon >0$ and $N_0>0$ such that the following holds for $n \ge N_0$. 
    Every $n$-vertex $F_5$-free $3$-graph with $\norm{\mathcal{H}}_{2,p} \ge (1-\varepsilon)\cdot \mathrm{ex}_{2,p}(n,F_5)$ is $3$-partite after removing at most $\delta n^3$ edges.
\end{theorem}
As mentioned earlier, the case where $p \in (0,1/2)$ is closely connected to Steiner triple systems and seems difficult to fully resolve. 
Recall that a \textbf{Steiner triple system} (STS) is a $3$-graph in which every pair of vertices is contained in exactly one edge. 
It is a classic and well-known result~\cite{Bose39,Sko58} that a $k$-vertex STS exists iff $k \in 6\mathbb{N} + \{1,3\}$. 
For simplicity, we use $\mathrm{STS}(k)$ to denote the collection of all STSs on $k$ vertices. 

Given two $r$-graphs $\mathcal{H}$ and $\mathcal{G}$, a map $\phi \colon V(\mathcal{H}) \to V(\mathcal{G})$ is a \textbf{homomorphism} if $\phi(e) \in \mathcal{G}$ for all $e \in \mathcal{H}$. 
When such a homomorphism exists, we say $\mathcal{H}$ is \textbf{$\mathcal{G}$-colorable}. 

\begin{theorem}\label{THM:Lp-F5-p-small-b}
    Suppose that $p$ is a real number in $[0, 1/2]$. 
    Then $\pi_{2,p}(F_5) \le \frac{p^p}{(p+1)^{p+1}}$. 
    In addition, for every $k \in 6\mathbb{N}^{+} + \{0,2\}$, there exists $\alpha_k >0$ such that the following statements hold. 
    \begin{enumerate}[label=(\roman*)]
        \item\label{THM:Lp-F5-p-small-b-1} For every $p \in \left[k^{-1}, k^{-1} + \alpha_k\right]$, we have $\pi_{2,p}(F_5) = \frac{k}{(k+1)^{p+1}}$. 
        In particular, if $p = k^{-1}$, then $\pi_{2,p}(F_5) = \frac{p^p}{(p+1)^{p+1}}$. 
        % $\pi_{2,p}(F_5) = \frac{k}{(k+1)^{p+1}}$ for every $p \in \left[k^{-1}, k^{-1} + \alpha_k\right]$, and in particualr, $\pi_{2,p}(F_5) \le \frac{p^p}{(p+1)^{p+1}}$ if $p = k^{-1}$. 
        \item\label{THM:Lp-F5-p-small-b-2} For every $p\in \left[k^{-1}, k^{-1} + \alpha_k\right]$ and $\delta>0$, there exist $\varepsilon >0$ and $N_0>0$ such that every $F_5$-free $3$-graph on $n \ge N_0$ vertices with $\norm{\mathcal{H}}_{2,p} \ge (1-\varepsilon)\cdot \mathrm{ex}_{2,p}(n,F_5)$ is $\mathcal{S}$-colorable for some $\mathcal{S} \in \mathrm{STS}(k+1)$ after removing at most $\delta n^3$ edges. 
    \end{enumerate} 
\end{theorem}
\textbf{Remarks.}
Some straightforward but tedious calculations show that one can choose $\alpha_k = (k^3-k)^{-1}$; for details, we refer the reader to Lemma~\ref{LEMMA:F5-Lp-p-small-calculations} in the Appendix. 
For simplicity, we will only show the proof of Theorem~\ref{THM:Lp-F5-p-small-b} for the case $p = k^{-1}$. 
Extending the proof to the general case can be achieved easily by replacing Theorem~\ref{THM:LM-feasible-region-T3} with the piecewise linear bound as in the remark below~{\cite[Theorem~1.15]{LM21feasible}}. 

Proof of Theorem~\ref{THM:Lp-F5-p-large} is presented in Section~\ref{SEC:Proof-Lp-F5-large}. 
Proof of Theorems~\ref{THM:Lp-F5-p-small-a} and~\ref{THM:Lp-F5-p-small-b} is presented in Section~\ref{SEC:Proof-Lp-F5-small}. 
%%%%%%%%%%%%%%%%%%%%%%%%%%
\subsection{The expansion of graphs}\label{SEC:intro-expansion}
%
% In the second application, we investigate $\mathrm{ex}_{t,p}(n,\mathcal{F})$ for an important class of hypergraphs introduced by Mubayi~\cite{MU06}, known as expansions.
Given a graph $F$, the ($r$-uniform) \textbf{expansion} $H_{F}^{r}$ of $F$ is the $r$-graph obtained from $F$ by adding $r-2$ new vertices into each edge, ensuring that these $(r-2)$-sets are pairwise disjoint. 
For simplicity, let $H_{\ell+1}^{r} \coloneqq H_{K_{\ell+1}}^{r}$. 
Expansions are an important class of hypergraphs introduced by Mubayi~\cite{MU06} to provide the first explicitly defined examples that yield an infinite family of numbers realizable as Tur\'{a}n densities for hypergraphs.

We say a graph $F$ is \textbf{edge-critical} if there exists an edge $e\in F$ such that the chromatic number of $F\setminus\{e\}$ is strictly smaller than that of $F$. 
In the following theorem, we determine, for every edge-critical graph $F$ and for every real number $p > 1$, the exact value of $\mathrm{ex}_{t,p}(n,H_{F}^{r})$ and prove its corresponding degree-stability. 
This extends the classical theorems of Simonovits~\cite{SI68}, Andr\'{a}sfai--Erd\H{o}s--S\'{o}s~\cite{AES74}, Erd\H{o}s--Simonovits~\cite{ES73} on edge-critical graphs, as well as the classic theorems of Mubayi~\cite{MU06} and Pikhurko~\cite{PI13} on expansion of complete graphs.   

\begin{theorem}\label{THM:Lp-clique-expansion-p-large}
    Let $\ell \ge r > t \ge 1$ be integers, $p \ge 1$ be a real number, and $F$ be an edge-critical graph with $\chi(F) = \ell+1$. 
    There exist $\varepsilon >0$ and $N_0>0$ such that the following statements hold for every integer $n \ge N_{0}$. 
    \begin{enumerate}[label=(\roman*)]
        \item Every $n$-vertex $H_{F}^{r}$-free $r$-graph with $\delta_{t,p}(\mathcal{H}) \ge (1-\varepsilon)\cdot \mathrm{exdeg}_{t,p}(n,H_{F}^{r})$ is $\ell$-partite. 
        \item $\mathrm{ex}_{t,p}(n,H_{F}^{r}) = \norm{\mathcal{G}}_{t,p}$ for some $n$-vertex complete $\ell$-partite $r$-graph $\mathcal{G}$.  
    \end{enumerate} 
\end{theorem}

For the case $p \in (0,1)$, we utilize the theorems on the feasible region problems established in~\cite{LM21feasible,LM23KKstability} to determine the value of $\pi_{t,p}(H_{F}^{r})$ and prove its corresponding edge-stability. 
Similar to the case of $F_5$, the $(t,p)$-norm Tur\'{a}n problem for $H_{F}^{r}$ does not have degree-stability when $p < 1$.  
The constructions can be obtained using a similar approach to that of $F_5$, so we omit them here. 

\begin{theorem}\label{THM:Lp-clique-expansion-p-small}
    Let $\ell \ge r > t \ge 1$ be integers, $p \in (0,1)$ be a real number, and $F$ be a graph with $\chi(F) = \ell+1$. 
    Then 
    \begin{align*}
        \pi_{t,p}(H_{F}^{r})
        = t!\binom{\ell}{t}\binom{\ell-t}{r-t}^{p}\left(\frac{1}{\ell}\right)^{t+p(r-t)}. 
    \end{align*}
    Moreover, for every $\delta >0$ there exist $\varepsilon >0$ and $N_0>0$ such that every $H_{F}^{r}$-free $r$-graph $\mathcal{H}$ on $n \ge N_0$ vertices with $\norm{\mathcal{H}}_{t,p} \ge (1-\varepsilon)\cdot \mathrm{ex}_{t,p}(n,H_{F}^{r})$ is $\ell$-partite after removing at most $\delta n^r$ edges. %\left(\frac{\pi_{t,p}(H_{F}^{r})}{t!} -\varepsilon\right) n^{t+p(r-t)}
\end{theorem}

An immediate corollary (using Proposition~\ref{PROP:tp-density-blowup-lemma}) of Theorems~\ref{THM:Lp-clique-expansion-p-large} and~\ref{THM:Lp-clique-expansion-p-small} is the following extension of the seminal Erd\H{o}s--Stone--Simonovits Theorem~\cite{ES66}. 

\begin{corollary}\label{CORO:Lp-expansion-ESS}
    Suppose that $\ell \ge r > t \ge 1$ are integers and $F$ is a graph with $\chi(F) = \ell+1$. 
    Then 
    \begin{align*}
        \pi_{t,p}(H_{F}^{r})
        = 
        \begin{cases}
            t!\binom{\ell}{t}\binom{\ell-t}{r-t}^{p}\left(\frac{1}{\ell}\right)^{t+p(r-t)}, & \quad\text{if}\quad p \in (0,1), \\
            t! \cdot \lambda_{t,p}(K_{\ell}^{r}), & \quad\text{if}\quad p \ge 1. 
        \end{cases}
    \end{align*}
\end{corollary}
\textbf{Remark.} 
% The exact value of $\pi_{t,p}(H_{\ell+1}^{r})$ can be determined by using Theorem~\ref{THM:Lp-clique-expansion-p-small} in the case $p \in (0,1)$ and solving the optimization problem defined in Section
The value $\lambda_{t,p}(K_{\ell}^{r})$ will be defined in Section~\ref{SEC:Prelim} and can be determined by solving an optimization problem.

The proof of Theorem~\ref{THM:Lp-clique-expansion-p-small} is presented in Section~\ref{SEC:Proof-Lp-F5-small}. 
The proof of Theorem~\ref{THM:Lp-clique-expansion-p-large} follows a similar strategy to that of  Theorem~\ref{THM:Lp-F5-p-large} but is much simpler, as it does not require the complicated calculations presented in Sections~\ref{SUBSEC:F5-large-STS} and~\ref{SUBSEC:proof-F5-inequalities}. 
To prevent the paper from becoming excessively long, we omit it here and refer the interested reader to the proof of~{\cite[Theorem~2.5]{CL24}} for the necessary technical lemmas. 

In the next section, we introduce some definitions and present some preliminary results. 
In Section~\ref{SEC:General-property}, we establish general properties of the $(t,p)$-norm of hypergraphs and state the general theorems on $(t,p)$-norm Tur\'{a}n problems, which are consequences of more general theorems established in~\cite{CL24}.  
Section~\ref{SEC:Remark} contains additional remarks and some open problems. 

%%%%%%%%%%%%%%%%%%%%%%%%%%%%%%%%%%
\section{Preliminaries}\label{SEC:Prelim}
% %
A family $\mathfrak{F}$ of $r$-graphs is \textbf{hereditary} if, for every $F \in \mathfrak{F}$, all subgraphs of $F$ are also included in $\mathfrak{F}$. 
We use $\mathfrak{G}^{r}$ to denote the collection of all $r$-graphs. 
Recall that an $r$-graph $\mathcal{H}$ is \textbf{$\mathcal{G}$-colorable} if there exists a homomorphism from $\mathcal{H}$ to $\mathcal{G}$. 
Extending the definition of the chromatic number for graphs, the \textbf{chromatic number} $\chi(Q)$ of an $r$-graph $Q$ is the minimum integer $\ell$ such that $Q$ is $K_{\ell}^r$-colorable. 
For integers $\ell \ge r \ge 2$, we use $\mathfrak{K}_{\ell}^{r}$ to denote the collection of all $K_{\ell}^{r}$-colorable $r$-graphs. 
Note that $\mathfrak{K}_{\ell}^{r}$ is a hereditary family.

Given an $r$-graph $\mathcal{G}$ on $[m]$ and pairwise disjoint sets $V_1, \ldots, V_m$, we use $\mathcal{G}(V_1, \ldots, V_m)$ to denote the $r$-graph obtained from $\mathcal{G}$ by replacing every vertex $i\in [m]$ with the set $V_i$ and every edge $\{i_1, \ldots, i_m\} \in \mathcal{G}$ with the complete $r$-partite $r$-graph with parts $V_{i_1}, \ldots, V_{i_m}$. 
The $r$-graph $\mathcal{G}(V_1, \ldots, V_m)$ is called a \textbf{blowup} of $\mathcal{G}$. 
For every integer $k \ge 1$, we use $\mathcal{G}(k)$ to denote the blowup $\mathcal{G}(V_1, \ldots, V_m)$ in which each $V_i$ has size $k$. 
For every family $\mathcal{F}$, we use $\mathcal{F}(k)$ to denote the set $\{F(k) \colon F\in \mathcal{F}\}$. 
For convenience, we use $K^{r}(V_1, \ldots, V_{\ell})$ to denote the complete $\ell$-partite $r$-graph with parts $V_1, \ldots, V_{\ell}$. 
The generalized Tur\'{a}n graph $T^{r}(n,\ell)$ is the balanced complete $\ell$-partite $r$-graph on $n$ vertices. 

Given an $r$-graph $\mathcal{H}$ and a set $S\subset V(\mathcal{H})$, we use $\mathcal{H}[S]$ to denote the \textbf{induced subgraph} of $\mathcal{H}$ on $S$, and use $\mathcal{H} - S$ to denote the induced subgraph of $\mathcal{H}$ on $V(\mathcal{H})\setminus S$. 
For a vertex $v \in V(\mathcal{H})$ the \textbf{neighborhood} of $v$ in $\mathcal{H}$ is   
\begin{align*}
    N_{\mathcal{H}}(v)
    \coloneqq 
    \left\{u\in V(\mathcal{H}) \colon \exists e\in \mathcal{H} \text{ such that } \{u,v\} \subset e \right\}.
\end{align*}
Given a pair of vertices $u,v\in V(\mathcal{H})$,
we say $\{u,v\}$ is \textbf{uncovered} if there is no edge in $\mathcal{H}$ containing both $u$ and $v$. 
We say $u,v$ are \textbf{equivalent} if $L_{\mathcal{H}}(u) = L_{\mathcal{H}}(v)$ (in particular, this means that $\{u,v\}$ is uncovered).
We use $\mathcal{H}_{u\to v}$ to denote the $r$-graph obtained from $\mathcal{H}$ by \textbf{symmetrizing $u$ into $v$}, that is 
\begin{align*}
    \mathcal{H}_{u\to v} 
    \coloneqq \left(\mathcal{H} \setminus \{E\in \mathcal{H} \colon u\in E\} \right) \cup \left\{e \cup \{u\} \colon e\in L_{\mathcal{H}}(v) \right\}. 
\end{align*}
An $r$-graph $\mathcal{H}$ is \textbf{symmetrized} if every two uncovered vertices in $\mathcal{H}$ are equivalent.

We say a map $\Gamma \colon \mathfrak{G}^{r} \to \mathbb{R}$ is \textbf{symmetrization-increasing} if for every $\mathcal{H} \in \mathfrak{G}^{r}$ and for every pair of uncovered vertices $u,v \in V(\mathcal{H})$,  
    \begin{align}\label{equ:sym-increase}
        \text{either}\quad 
        \Gamma(\mathcal{H}) 
        < \max\left\{\Gamma(\mathcal{H}_{u\to v}),\ \Gamma(\mathcal{H}_{v\to u})\right\}
        \quad\text{or}\quad 
        \Gamma(\mathcal{H})
        = \Gamma(\mathcal{H}_{u\to v})
        = \Gamma(\mathcal{H}_{v\to u}).
    \end{align}   
Notice that~\eqref{equ:sym-increase} holds iff there exists $\alpha \in (0,1)$ such that 
\begin{align*}
    \Gamma(\mathcal{H})
    \le \alpha \cdot \Gamma(\mathcal{H}_{u\to v}) + (1-\alpha) \cdot \Gamma(\mathcal{H}_{v\to u}). 
\end{align*}
Note that we typically choose $\alpha = 1/2$ when proving that a specific map $\Gamma$ is symmetrization-increasing.

Given a family $\mathcal{F}$ of $r$-graphs, the \textbf{$\mathcal{F}$-freeness indicator map} $\mathbbm{1}_{\mathcal{F}} \colon \mathfrak{G}^{r} \to \{0,1\}$ is defined as follows$\colon$
\begin{align*}
    \mathbbm{1}_{\mathcal{F}}(\mathcal{H})
    \coloneqq 
    \begin{cases}
        1, & \quad\text{if $\mathcal{H}$ is $\mathcal{F}$-free}, \\
        0, & \quad \text{otherwise}. 
    \end{cases}
\end{align*}
Following the definition in~\cite{LMR23unif}, a family $\mathcal{F}$ of $r$-graphs is \textbf{blowup-invariant} if blowups of every $\mathcal{F}$-free $r$-graph remain $\mathcal{F}$-free.
It is easy to see that if $\mathcal{F}$ is blowup-invariant and $\mathcal{H}$ is an $\mathcal{F}$-free $r$-graph, then $\mathcal{H}_{u \to v}$ is $\mathcal{F}$-free for every uncovered pair $\{u,v\}$ in $\mathcal{H}$.  
Hence, the map $\mathbbm{1}_{\mathcal{F}}$ is symmetrization-increasing for every blowup-invariant $\mathcal{F}$. 

Now we are ready to state the first property of the $(t,p)$-norm. 

\begin{proposition}\label{PROP:star-polynomial-sym-increase}
    Let $r> t \ge 1$ be integers and $p \ge 1$ be a real number. The map $\Gamma \colon \mathfrak{G}^{r} \to \mathbb{R}$, defined by 
        \begin{align*}
            \Gamma(\mathcal{H})
            \coloneqq \norm{\mathcal{H}}_{t,p}
            \quad\text{for every } \mathcal{H} \in \mathfrak{G}^{r}, 
        \end{align*}
        is symmetrization-increasing.
        Consequently, for every blowup-invariant family $\mathcal{F}$, the map $\Gamma' \colon \mathfrak{G}^{r} \to \mathbb{R}$, defined by  
        \begin{align*}
            \Gamma'(\mathcal{H})
            \coloneqq \norm{\mathcal{H}}_{t,p} \cdot \mathbbm{1}_{\mathcal{F}}(\mathcal{H})
            \quad\text{for every } \mathcal{H} \in \mathfrak{G}^{r}, 
        \end{align*}
        is symmetrization-increasing.
\end{proposition}
\begin{proof}[Proof of Proposition~\ref{PROP:star-polynomial-sym-increase}]
    Let $\mathcal{H}$ be an $r$-graph and $u,v \in V(\mathcal{H})$ be a pair of uncovered vertices. 
    For convenience, let $V \coloneqq V(\mathcal{H})$, $\mathcal{T}_{0}\coloneqq \binom{V\setminus \{u,v\}}{t}$, 
    %$\mathcal{T}_{u}$ be the collection of elements in $\binom{V\setminus \{v\}}{t}$ containing $u$, and $\mathcal{T}_{v}$ be the collection of elements in $\binom{V\setminus \{u\}}{t}$ containing $v$. 
    \begin{align*}
       % \mathcal{T}_{0} 
       % \coloneqq \binom{V\setminus \{u,v\}}{t}, 
       % \quad 
       \mathcal{T}_{u} 
       \coloneqq \left\{T \in \binom{V\setminus \{v\}}{t} \colon u \in T\right\}, 
       \quad\text{and}\quad 
       \mathcal{T}_{v} 
       \coloneqq \left\{T \in \binom{V\setminus \{u\}}{t} \colon v \in T\right\}. 
    \end{align*}
    Let $\mathcal{H}_0 \coloneqq \mathcal{H}$, $\mathcal{H}_1 \coloneqq \mathcal{H}_{u\to v}$, and $\mathcal{H}_2 \coloneqq \mathcal{H}_{v\to u}$. 
    For each $i\in \{0,1,2\}$, since $\{u,v\}$ is uncovered in $\mathcal{H}_i$, we have $\partial_{r-t}\mathcal{H}_{i} \subset \mathcal{T}_{0} \cup \mathcal{T}_{u} \cup \mathcal{T}_{v}$, and hence, 
    \begin{align*}
        \Gamma(\mathcal{H}_i)
        = \sum_{T\in \mathcal{T}_u} d^{p}_{\mathcal{H}_i}(T) 
        + \sum_{T\in \mathcal{T}_v} d^{p}_{\mathcal{H}_i}(T) 
        + \sum_{T\in \mathcal{T}_0} d^{p}_{\mathcal{H}_i}(T). 
    \end{align*}
    It is easy to see from the definition of symmetrization that 
    \begin{align*}
        2 \sum_{T\in \mathcal{T}_u} d^{p}_{\mathcal{H}_0}(T)
        & = \sum_{T\in \mathcal{T}_u} d^{p}_{\mathcal{H}_2}(T) + \sum_{T\in \mathcal{T}_v} d^{p}_{\mathcal{H}_2}(T), 
        \quad\text{and}\quad  \\
        2 \sum_{T\in \mathcal{T}_v} d^{p}_{\mathcal{H}_0}(T)
        & = \sum_{T\in \mathcal{T}_u} d^{p}_{\mathcal{H}_1}(T) + \sum_{T\in \mathcal{T}_v} d^{p}_{\mathcal{H}_1}(T).
    \end{align*}
    So it suffices to show that 
    \begin{align}\label{equ:deg-T-p-power}
        2 \sum_{T\in \mathcal{T}_0} d^{p}_{\mathcal{H}_0}(T)
        \le \sum_{T\in \mathcal{T}_0} d^{p}_{\mathcal{H}_1}(T) 
            + \sum_{T\in \mathcal{T}_0} d^{p}_{\mathcal{H}_2}(T).
    \end{align}
    Fix $T \in \mathcal{T}_0$. 
    Let $z_{u} \coloneqq d_{\mathcal{H}}(T \cup \{u\})$, $z_{v} \coloneqq d_{\mathcal{H}}(T \cup \{v\})$, and $z_0 \coloneqq d_{\mathcal{H}}(T) - z_u-z_v$. 
    It is clear from the definition that 
    \begin{align*}
        d^{p}_{\mathcal{H}_0}(T)
         = \left(z_u+z_v+z_0\right)^{p}, \quad
        d^{p}_{\mathcal{H}_1}(T)
         = \left(2z_v+z_0\right)^{p}, \quad\text{and}\quad
        d^{p}_{\mathcal{H}_2}(T)
         = \left(2z_u+z_0\right)^{p}. 
    \end{align*}
    Since $p \ge 1$, it follows from Jensen's inequality that 
    \begin{align*}
        \frac{\left(2z_v+z_0\right)^{p} + \left(2z_u+z_0\right)^{p}}{2} 
        \ge \left(\frac{2z_v+z_0 + 2z_u+z_0}{2}\right)^p
        = \left(z_u+z_v+z_0\right)^{p}. 
    \end{align*}
    Therefore, $2d^{p}_{\mathcal{H}_0}(T) \le d^{p}_{\mathcal{H}_1}(T) + d^{p}_{\mathcal{H}_2}(T)$. 
    Summing over all $T\in \mathcal{T}_0$, we obtain~\eqref{equ:deg-T-p-power}, and hence, complete the proof of Proposition~\ref{PROP:star-polynomial-sym-increase}. 
\end{proof}

The following result justifies the definition of $\pi_{t,p}(\mathcal{F})$. 
Its proof can be obtained with a minor modification to that of~{\cite[Proposition~1.8]{BCL22}}, which is itself an extension of the averaging argument used by Katona--Nemetz--Simonovits~\cite{KNS64}.
\begin{proposition}\label{PROP:limit-exist}
    Let $r > t \ge 1 $ be integers and $p \ge 1$ be a real number. For every family $\mathcal{F}$ of $r$-graphs, the limit $\lim\limits_{n\to \infty} \frac{\mathrm{ex}_{t,p}(n,\mathcal{F})}{n^{t+p(r-t)}}$ exists. 
    In particular, if $\mathcal{F}$ is a nondegenerate family of $r$-graphs, then for every $\delta \in [0,1]$, 
    \begin{align}\label{equ:PROP:limit-exist}
        \mathrm{ex}_{t,p}(n-\delta n,\mathcal{F})
        & \le \mathrm{ex}_{t,p}(n,\mathcal{F}) - \delta \left(t+p(r-t)\right) \cdot \mathrm{ex}_{t,p}(n,\mathcal{F})  \notag \\
        & \quad + \delta^2 \left(t+p(r-t)\right)^2 \cdot \mathrm{ex}_{t,p}(n,\mathcal{F}) 
        + o(n^{t+p(r-t)}). 
    \end{align}
\end{proposition}
\textbf{Remark.} 
\begin{itemize}
    \item Inequality~\eqref{equ:PROP:limit-exist} verifies the \textbf{smoothness} (as defined in~\cite{CL24}) of the map $\norm{\cdot}_{t,p} \colon \mathfrak{G}^{r} \to \mathbb{R}$ for all $r > t \ge 1 $ and $p \ge 1$. 
    \item A slight modification of the argument of Katona--Nemetz--Simonovits shows that the limit $\lim\limits_{n\to \infty} \frac{\mathrm{ex}_{t,p}(n,\mathcal{F})}{n^{t+p(r-t)}}$ also exists for $p \in (0,1)$, and we refer the reader to Proposition~\ref{PROP:limit-exits-small-p} in the Appendix for details. 
\end{itemize}

Given two families $\mathcal{F}, \hat{\mathcal{F}}$ of $r$-graphs, we say $\hat{\mathcal{F}} \le_{\mathrm{hom}} \mathcal{F}$ if for every $\hat{F} \in \hat{\mathcal{F}}$ there exists $F\in \mathcal{F}$ such that $F$ is $\hat{F}$-colorable. 
It is clear that for every family $\mathcal{F}$ of $r$-graphs and for every integer $k \ge 1$, $\mathcal{F} \le_{\mathrm{hom}} \mathcal{F}(k)$. 
Another example is $\mathcal{T}_{3} \coloneqq \{F_5, K_{4}^{3-}\} \le_{\mathrm{hom}} \{F_5\}$ since $F_5$ is $K_{4}^{3-}$-colorable. Here, $K_{4}^{3-}$ denote the $3$-graph obtained from the complete $3$-graph $K_{4}^{3}$ by removing one edge. 
The motivation for the definition of $\le_{\mathrm{hom}}$ comes from the following result, which is a simple consequence of the Hypergraph Removal Lemma~\cite{RS09} and Lemma~\ref{LEMMA:Lp-basic-property-global} in Section~\ref{SEC:General-property}. 
\begin{proposition}\label{PROP:tp-density-blowup-lemma}
    Suppose that $\mathcal{F}, \hat{\mathcal{F}}$ are two families of $r$-graphs satisfying $\hat{\mathcal{F}} \le_{\mathrm{hom}} \mathcal{F}$. 
    Then every $\mathcal{F}$-free $r$-graph on $n$ vertices is $\hat{\mathcal{F}}$-free after removing $o(n^r)$ edges. 
    Consequently, (by Lemma~\ref{LEMMA:Lp-basic-property-global}), for all $t > t \ge 1$ and $p > 0$, 
    \begin{align*}
        \pi_{t,p}(\mathcal{F})
        \le \pi_{t,p}(\hat{\mathcal{F}}),  
    \end{align*}
    and, in particular, 
    \begin{align*}
        \pi_{t,p}(\mathcal{F})
        = \pi_{t,p}(\mathcal{F}(k))
        \quad\text{for every}\quad k \ge 1. 
    \end{align*}
\end{proposition}

Given a differentiable map $f \colon \mathbb{R}^{n} \to \mathbb{R}$ with $n$ variables, we use $D_{i}f$ to denote the \textbf{partial derivative} of $f$ with respect to the $i$-th variable.
If $f$ has only one variable, then we use $Df$ to denote the \textbf{derivative} of $f$.  
Given $X\subset \mathbb{R}^{n}$, a map $f \colon X \to \mathbb{R}$ is \textbf{homogeneous of degree $k$} if 
\begin{align*}
    f(\lambda \vec{x}) 
    = \lambda^{k} \cdot f(\vec{x})
    \quad\text{for all $\vec{x} \in X$ and for all $\lambda>0$ with $\lambda \vec{x} \in X$}. 
\end{align*}
Let $r > t \ge 1$ be integers and let $\mathcal{H}$ be an $r$-graph on $[n]$. 
For every real number $p > 0$ define the \textbf{$(t,p)$-Lagrange polynomial} of $\mathcal{H}$ as follows$\colon$
\begin{align*}
    L_{\mathcal{H},t,p}(X_1, \ldots, X_{n})
    \coloneqq 
    \sum_{T\in \partial_{r-t}\mathcal{H}} X_{T}\left(\sum_{I\in L_{\mathcal{H}}(T)}X_I\right)^{p},  
    %\sum_{\{i,j,k\}\in \mathcal{H}}\left(X_{i}X_{j}X_{k}^{p}+X_{i}X_{j}^{p}X_{k}+X_{i}^{p}X_{j}X_{k}\right). 
\end{align*}
where $X_{S} \coloneqq \prod_{i\in S}X_i$ for every $S\subset [n]$. 
The \textbf{$(t,p)$-Lagrangian} of $\mathcal{H}$ is defined as follows$\colon$
\begin{align*}
    \lambda_{t,p}(\mathcal{H})
    \coloneqq \max\left\{L_{\mathcal{H},t,p}(x_1, \ldots, x_{n}) \colon (x_1, \ldots, x_{n}) \in \Delta^{n-1}\right\}, 
\end{align*}
where $\Delta^{n-1} \coloneqq \left\{(x_1, \ldots, x_n) \in \mathbb{R}^{n} \colon x_i \ge 0 \text{ for } i\in [n] \text{ and } x_1 + \cdots + x_n = 1\right\}$ is the standard $(n-1)$-dimensional simplex. 
It is clear from the definition that $L_{\mathcal{H},t,p}$ is a homogeneous map of degree $t+p(r-t)$. 

One of the motivations for defining the $(t,p)$-Lagrangian comes from the following simple fact, which is an extension of a crucial property of Lagrangian (see e.g.~\cite{MS65,FR84}). 

\begin{fact}\label{FACT:t-p-Lagrangian}
    Let $r > t \ge 1$ be integers and $p > 0$ be a real number. 
    Suppose that $\mathcal{H}$ is an $n$-vertex $\mathcal{G}$-colorbale $r$-graph. 
    Then 
    \begin{align*}
        \norm{\mathcal{H}}_{t,p}
        = L_{\mathcal{H},t,p}(1,\ldots,1)
        = L_{\mathcal{H},t,p}(1/n,\ldots, 1/n) \cdot n^{t+p(r-t)}
        \le \lambda_{t,p}(\mathcal{G}) \cdot n^{t+p(r-t)}. 
    \end{align*}
\end{fact}
In the following proposition, we establish a useful inequality for the $(t,p)$-Lagrangian of an $r$-graph for different values of $p$.
\begin{proposition}\label{PROP:star-poly-Holder-inequ}
    Let $r > t \ge 1$ be integers and $\mathcal{H}$ be an $r$-graph on $[n]$. 
    Suppose that $0 \le p_1 < p < p_2$ are real numbers. Then 
    \begin{align*}
        L_{\mathcal{H},t,p}(\vec{x})
        \le \left(L_{\mathcal{H},t,p_1}(\vec{x})\right)^{\frac{p_2-p}{p_2-p_1}}
            \left(L_{\mathcal{H},t,p_2}(\vec{x})\right)^{\frac{p-p_1}{p_2-p_1}}
            \quad\text{for every}\quad \vec{x} = (x_1, \ldots, x_n) \in \Delta^{n-1}. 
    \end{align*}
    In particular, 
    \begin{align*}
        \lambda_{t,p}(\mathcal{H})
        \le \left(\lambda_{t,p_1}(\mathcal{H})\right)^{\frac{p_2-p}{p_2-p_1}}
            \left(\lambda_{t,p_2}(\mathcal{H})\right)^{\frac{p-p_1}{p_2-p_1}}. 
    \end{align*}
    % \begin{align*}
    %     h_{p,r,t,\ell}^{\ast}
    %     \le \left(h_{p_1,r,t,\ell}^{\ast}\right)^{\frac{p_2-p}{p_2-p_1}} \left(h_{p_2,r,t,\ell}^{\ast}\right)^{\frac{p-p_1}{p_2-p_1}}. 
    % \end{align*}
    % In particular, if $h_{p_1,r,t,\ell}^{\ast} = h_{p_2,r,t,\ell}^{\ast}$, then $h_{p,r,t,\ell}^{\ast} = $
\end{proposition}

We will use the following version of H\"{o}lder's inequality in the proof of Proposition~\ref{PROP:star-poly-Holder-inequ}. 

\begin{fact}[H\"{o}lder's inequality]\label{FACT:Holder-Inequality}
    For every $(a_1, \ldots, a_n), (b_1, \ldots, b_n) \in \mathbb{R}^{n}_{\ge 0}$ and positive real numbers $s,t$ with $\frac{1}{s}+\frac{1}{t} = 1$, 
    \begin{align*}
        \sum_{i=1}^{n} a_i b_i 
        \le \left(\sum_{i=1}^{n} a_i^{s}\right)^{\frac{1}{s}}\left(\sum_{i=1}^{n} b_i^{t}\right)^{\frac{1}{t}}. 
    \end{align*}
\end{fact}

\begin{proof}[Proof of Proposition~\ref{PROP:star-poly-Holder-inequ}]
    % Fix $(x_1, \ldots, x_{\ell}) \in \Delta^{\ell-1}$ such that $h_{p,r,t,\ell}^{\ast} = h_{p,r,t,\ell}(x_1, \ldots, x_{\ell})$. 
    Fix $(x_1, \ldots, x_{n}) \in \Delta^{n-1}$. 
    Let $s \coloneqq \frac{p_2-p_1}{p_2-p}$, $t \coloneqq \frac{p_2-p_1}{p-p_1}$, $a_T \coloneqq x_T^{\frac{1}{s}} \left(\sum_{I\in L_{\mathcal{H}}(T)}x_{I}\right)^{\frac{p_1}{s}}$, $b_T\coloneqq x_T^{\frac{1}{t}} \left(\sum_{I\in L_{\mathcal{H}}(T)}x_{I}\right)^{\frac{p_2}{t}}$. 
    Observe that $\frac{1}{s}+\frac{1}{t} = 1$ and $\frac{p_1}{s}+\frac{p_2}{t} = p$. 
    So, it follows from H\"{o}lder's inequality that 
    \begin{align*}
        L_{\mathcal{H},t,p}(\vec{x})
        & =  \sum_{T \in \partial_{r-t}\mathcal{H}} a_T b_T \\
        & \le \left(\sum_{T \in \partial_{r-t}\mathcal{H}} a_T^{s}\right)^{\frac{1}{s}}
         \left(\sum_{T \in \partial_{r-t}\mathcal{H}} b_T^{t}\right)^{\frac{1}{t}} 
         = \left(L_{\mathcal{H},t,p_1}(\vec{x})\right)^{\frac{1}{s}}\left(L_{\mathcal{H},t,p_2}(\vec{x})\right)^{\frac{1}{t}}, 
        % & \le \left(h_{p_1,r,t,\ell}^{\ast}\right)^{\frac{p_2-p}{p_2-p_1}} \left(h_{p_2,r,t,\ell}^{\ast}\right)^{\frac{p-p_1}{p_2-p_1}}, 
    \end{align*}
    proving Proposition~\ref{PROP:star-poly-Holder-inequ}. 
\end{proof}

The following proposition is an extension of~{\cite[Theorem~2.1]{FR84}}. 

\begin{proposition}\label{PROP:tp-Lagrangian-2-covered}
    Let $r > t \ge 1$ be integers, $p \ge 1$ be a real number, and $\mathcal{H}$ be an $r$-graph on $[n]$.
    There exists a vertex set $U \subset [n]$ such that 
    \begin{enumerate}[label=(\roman*)]
        \item the induced subgraph $\mathcal{H}[U]$ is $2$-covered and $\lambda_{t,p}(\mathcal{H}) = \lambda_{t,p}(\mathcal{H}[U])$, 
        \item $D_{i}L_{\mathcal{H},t,p}(x_1, \ldots, x_n) = \left(t+p(r-t)\right) \cdot \lambda_{t,p} (\mathcal{H})$ for every $i\in U$. 
    \end{enumerate}
\end{proposition}
\begin{proof}[Proof of Proposition~\ref{PROP:tp-Lagrangian-2-covered}]
    Fix an $r$-graph $\mathcal{H}$ on $[n]$. 
    Let $\vec{x} = (x_1, \ldots, x_n) \in \Delta^{n-1}$ be a vector satisfying $L_{\mathcal{H},t,p}(\vec{x}) = \lambda_{t,p}(\mathcal{H})$ and such that the size of the set $U \coloneqq \{i\in [n] \colon x_i > 0\}$ is minimized. 
    By relabelling the vertices of $\mathcal{H}$, we may assume that $U = [m]$ for some $m \le n$. We claim that the set $U$ defined here satisfies the conclusions in Proposition~\ref{PROP:tp-Lagrangian-2-covered}. 

    Let $\mathcal{G} \coloneqq \mathcal{H}[U]$. 
    Suppose to the contrary that $\mathcal{G}$ is not $2$-covered. By relabelling the vertices in $U = [m]$, we may assume that $\{1,2\} \subset [m]$ is uncovered in $\mathcal{G}$. 
    Let $\vec{y} \coloneqq (y_1, \ldots, y_m), \vec{z} \coloneqq (z_1, \ldots, z_m) \in \Delta^{m-1}$ be defined as 
    \begin{align*} 
        y_1 = z_2 = x_1+x_2, \quad
        y_2 = z_1 = 0, \quad\text{and}\quad 
        y_k = z_k = x_k \quad\text{for}\quad k \in [3,m]. 
    \end{align*}
    Let $\mathcal{T}_{0}\coloneqq \binom{[3,m]}{t}$, % $\mathcal{T}_{1}$ be the collection of elements in $\binom{[m]\setminus \{2\}}{t}$ containing $1$, and $\mathcal{T}_{2}$ be the collection of elements in $\binom{[m]\setminus \{1\}}{t}$ containing $2$.
    \begin{align*}
        \mathcal{T}_{1} \coloneqq \left\{T \in \binom{[m]\setminus \{2\}}{t} \colon 1 \in T\right\}, 
        \quad\text{and}\quad 
        \mathcal{T}_{2} \coloneqq \left\{T \in \binom{[m]\setminus \{1\}}{t} \colon 2 \in T\right\}. 
    \end{align*}
    Since $\{1,2\}$ is uncovered in $\mathcal{G}$, we have $\partial_{r-t} \mathcal{G} \subset \mathcal{T}_{0} \cup \mathcal{T}_{1} \cup \mathcal{T}_2$. 
    First, notice that 
    \begin{align}\label{equ:PROP:tp-Lagrangian-2-covered-1}
        & \sum_{T\in \mathcal{T}_1} x_T\left(\sum_{I\in L_{\mathcal{G}}(T)} x_I\right)^{p} 
        + \sum_{T\in \mathcal{T}_2} x_T\left(\sum_{I\in L_{\mathcal{G}}(T)} x_I\right)^{p} \notag \\
        & = \frac{x_1}{x_1+x_2} \sum_{T\in \mathcal{T}_1} y_T\left(\sum_{I\in L_{\mathcal{G}}(T)} y_I\right)^{p}
        + \frac{x_2}{x_1+x_2} \sum_{T\in \mathcal{T}_2} z_T\left(\sum_{I\in L_{\mathcal{G}}(T)} z_I\right)^{p} 
    \end{align}
    Next, we consider $\mathcal{T}_0$. Fix $T \in \mathcal{T}_0$ and let 
    \begin{align*}
        & \alpha_0 \coloneqq \sum_{I \in L_{\mathcal{G}}(T) \colon \{1,2\} \cap I = \emptyset} x_I, \quad 
        \alpha_1 \coloneqq \sum_{I \in L_{\mathcal{G}}(T) \colon 1 \in I} x_I, 
        \quad\text{and}\quad  
        \alpha_2 \coloneqq \sum_{I \in L_{\mathcal{G}}(T) \colon 2 \in I} x_I, \\
        & \beta_0 \coloneqq \sum_{I \in L_{\mathcal{G}}(T) \colon \{1,2\} \cap I = \emptyset} y_I, \quad 
        \beta_1 \coloneqq \sum_{I \in L_{\mathcal{G}}(T) \colon 1 \in I} y_I, 
        \quad\text{and}\quad  
        \beta_2 \coloneqq \sum_{I \in L_{\mathcal{G}}(T) \colon 2 \in I} y_I, \\
        & \gamma_0 \coloneqq \sum_{I \in L_{\mathcal{G}}(T) \colon \{1,2\} \cap I = \emptyset} z_I, \quad 
        \gamma_1 \coloneqq \sum_{I \in L_{\mathcal{G}}(T) \colon 1 \in I} z_I, 
        \quad\text{and}\quad  
        \gamma_2 \coloneqq \sum_{I \in L_{\mathcal{G}}(T) \colon 2 \in I} z_I. 
    \end{align*}
    Notice that 
    \begin{align*}
        \alpha_0 = \beta_0 = \gamma_0, 
        \quad \beta_2 = \gamma_1 = 0, 
        \quad
        \beta_1 = \frac{x_1+x_2}{x_1} \alpha_1,
        \quad\text{and}\quad
        \gamma_2 = \frac{x_1+x_2}{x_2} \alpha_2. 
    \end{align*}
    It follows from Jensen's inequality that 
    \begin{align*}
        & \frac{x_1}{x_1+x_2}\left(\sum_{I\in L_{\mathcal{G}}(T)} y_I\right)^{p}
        + \frac{x_2}{x_1+x_2}\left(\sum_{I\in L_{\mathcal{G}}(T)} z_I\right)^{p}  \\
        & = \frac{x_1}{x_1+x_2}\left(\beta_1  + \beta_0\right)^{p} 
        + \frac{x_2}{x_1+x_2}\left( \gamma_2 + \gamma_0\right)^{p} \\
        % & = \frac{x_1}{x_1+x_2}\left(\beta_1 + \beta_2 + \beta_0\right)^{p} 
        % + \frac{x_2}{x_1+x_2}\left(\gamma_1 + \gamma_2 + \gamma_0\right)^{p} \\ 
        % & \ge \left(\frac{x_1}{x_1+x_2}\left(\beta_1 + \beta_2 + \beta_0\right) + \frac{x_2}{x_1+x_2}\left(\gamma_1 + \gamma_2 + \gamma_0\right)\right)^{p}  \\
         & \ge \left(\frac{x_1}{x_1+x_2}\left(\beta_1 +  \beta_0\right) + \frac{x_2}{x_1+x_2}\left( \gamma_2 + \gamma_0\right)\right)^{p}  
         = \left(\alpha_1 + \alpha_2 + \alpha_0\right)^{p}
         = \left(\sum_{I\in L_{\mathcal{G}}(T)} x_I\right)^{p}. 
    \end{align*}
    Therefore, 
    \begin{align*}
        \sum_{T\in \mathcal{T}_0} x_{T} \left(\sum_{I\in L_{\mathcal{G}}(T)} x_I\right)^{p}
        & \le \frac{x_1}{x_1+x_2} \sum_{T\in \mathcal{T}_0} x_{T} \left(\sum_{I\in L_{\mathcal{G}}(T)} y_I\right)^{p}
        + \frac{x_2}{x_1+x_2} \sum_{T\in \mathcal{T}_0} x_{T} \left(\sum_{I\in L_{\mathcal{G}}(T)} z_I\right)^{p} \\
        & = \frac{x_1}{x_1+x_2} \sum_{T\in \mathcal{T}_0} y_{T} \left(\sum_{I\in L_{\mathcal{G}}(T)} y_I\right)^{p}
        + \frac{x_2}{x_1+x_2} \sum_{T\in \mathcal{T}_0} z_{T} \left(\sum_{I\in L_{\mathcal{G}}(T)} z_I\right)^{p}, 
    \end{align*}
    which combined with~\eqref{equ:PROP:tp-Lagrangian-2-covered-1}, implies that 
    \begin{align*}
        \lambda_{t,p}(\mathcal{H})
        = L_{\mathcal{H},t,p}(\vec{x})
         = L_{\mathcal{G},t,p}(x_1, \ldots, x_m) 
        % & = \sum_{T\in \mathcal{T}_0} x_T\left(\sum_{I\in L_{\mathcal{G}}(T)} x_I\right)^{p}
        %  + \sum_{T\in \mathcal{T}_1} x_T\left(\sum_{I\in L_{\mathcal{G}}(T)} x_I\right)^{p}
        %  + \sum_{T\in \mathcal{T}_2} x_T\left(\sum_{I\in L_{\mathcal{H}}(T)} x_I\right)^{p} \\
         \le \frac{x_1}{x_1+x_2} L_{\mathcal{G},t,p}(\vec{y}) + \frac{x_2}{x_1+x_2} L_{\mathcal{G},t,p}(\vec{z}).  
    \end{align*}
    Since $\max\{L_{\mathcal{G},t,p}(\vec{y}), L_{\mathcal{G},t,p}(\vec{z})\} \le \lambda_{t,p}(\mathcal{H})$, it follows from the equality above that $L_{\mathcal{G},t,p}(\vec{y}) = L_{\mathcal{G},t,p}(\vec{z}) = \lambda_{t,p}(\mathcal{H})$. 
    However, this contradicts the minimality of $U$. Therefore, $\mathcal{H}[U]$ is $2$-covered.  

    Since $L_{\mathcal{H},t,p}$ is a homogeneous map of degree $t+p(r-t)$, the equality, $D_{i}L_{\mathcal{H},t,p}(x_1, \ldots, x_n) = \left(t+p(r-t)\right) \cdot \lambda_{L}(\mathcal{H})$ for every $i\in U$, follows easily from Euler's homogeneous function theorem and the theory of Lagrange Multipliers. 
\end{proof}

In the end of this section, we list some inequalities that will be useful later. 

\begin{fact}\label{FACT:inequality-a}
    Let $x\in [0,1]$ be a real number.
    The following statements hold. 
    \begin{enumerate}[label=(\roman*)]
        \item\label{FACT:inequality-a-1} For every $p \in [0,1]$, $1-px-(1-p)x^2 \le (1-x)^p \le 1-px$. 
            % \begin{align*}
            %     -(1-p)x^2 \le (1-x)^p -(1-px) \le 0. 
            % \end{align*}
        \item\label{FACT:inequality-a-2} For every $p \ge 1$, $1-px \le (1-x)^p \le  1-px + p^2 x^2$.  
            % \begin{align*}
            %     0 \le (1-x)^p -(1-px) \le  p^2 x^2. 
            % \end{align*}
    \end{enumerate}
    In particular, 
    \begin{align}\label{equ:FACT:inequality-a-3}
        |(1-x)^p -(1-px)| \le (p^2+1) x^2 
        \quad\text{for every } p \ge 0, 
    \end{align}
    and 
    \begin{align}\label{equ:FACT:inequality-a-4}
        x_1^{p} - (x_1-x_2)^{p} 
        \le p x_1^{p-1}x_2 \quad\text{for all}\quad x_1 \ge x_2 \ge 0,\ p \ge 1. 
    \end{align}
\end{fact}

\begin{fact}\label{FACT:inequality-b}
    Suppose that $x\in [0,1]$ and $p \in (0,1)$ are real numbers. 
    Then 
    \begin{enumerate}[label=(\roman*)]
        \item\label{FACT:inequality-b-1} $(1-x)^p \ge 1-x^p$. 
        \item\label{FACT:inequality-b-2} $x(1-x)^{p} \le \frac{p^p}{(1+p)^{1+p}}$, and equality holds iff $x = \frac{1}{1+p}$. 
    \end{enumerate}
\end{fact}
%%%%%%%%%%%%%%%%%%%%%%%%%%%%%%%%%%%%%%%%%%%%%%%%%%%%%
%%%%%%%%%%%%%%%%%%%%%%%%%%%%%%%%%%%%%%%
\section{General properties of $(t,p)$-norm Tur\'{a}n problems}\label{SEC:General-property}
In this section, we establish general properties of the $(t,p)$-norm of $r$-graphs that are necessary for applying the general theorems from~\cite{CL24} in the context of $(t,p)$-norm Tur\'{a}n problems. 
Specifically, in addition to the symmetrization-increasing property and smoothness already established in Proposition~\ref{PROP:star-polynomial-sym-increase} and Proposition~\ref{PROP:limit-exist}, we will show that the map $\Gamma \colon \mathfrak{G}^{t} \to \mathbb{R}$, defined by $\Gamma(\mathcal{H}) \coloneqq \norm{\mathcal{H}}_{t,p}$ for all $\mathcal{H} \in \mathfrak{G}^{r}$, is \textbf{uniform} (Lemma~\ref{LEMMA:Lp-uniform}), \textbf{locally Lipschitz} (Lemma~\ref{LEMMA:Lp-basic-property-local-Lipschitz}), \textbf{continuous} (Lemma~\ref{LEMMA:Lp-basic-property-global}), and \textbf{locally monotone} (Lemma~\ref{LEMMA:local-monotone}). 
Instead of including the lengthy definitions of the properties mentioned above, we will introduce each property as we prove them later. 
For detailed definitions, we refer the reader to~\cite{CL24}. 
The focus of this section will be the case $p \ge 1$, although some results also apply for $p \in (0,1)$. 

It should be noted that we will eventually apply the theorems from~\cite{CL24} to the map $\Gamma' \colon \mathfrak{G}^{t} \to \mathbb{R}$ (instead of $\Gamma$), defined by $\Gamma'(\mathcal{H}) \coloneqq \norm{\mathcal{H}}_{t,p}\cdot \mathbbm{1}_{\mathcal{F}}(\mathcal{H})$ for all $\mathcal{H} \in \mathfrak{G}^{r}$. 
However, deriving all the aforementioned properties (symmetrization-increasing, smooth, uniform, locally Lipschitz, continuous, and locally monotone) for $\Gamma'$ from those of $\Gamma$ is a straightforward task when $\mathcal{F}$ is nondegenerate and blowup-invariant (see Proposition~\ref{PROP:star-polynomial-sym-increase}). 

The first result of this section provides a handy expression for the $(t,p)$-degree of a vertex. 

\begin{lemma}\label{LEMMA:Lp-degree-expression}
    Let $r > t \ge 1 $ be integers and $p > 0$ be a real number. 
    For every $r$-graph $\mathcal{H}$ and for every $v\in V(\mathcal{H})$, 
    \begin{align*}%label{equ:degree-tp-norm}
        d_{\mathcal{H},t,p}(v)
        & = \sum_{S \in \partial_{r-t}L_{\mathcal{H}}(v)} 
            d_{\mathcal{H}}^{p}(S\cup\{v\})
            + \sum_{T \in \partial_{r-t-1}L_{\mathcal{H}}(v)} 
            \left( d_{\mathcal{H}}^{p}(T) - \left(d_{\mathcal{H}}(T) - d_{\mathcal{H}}(T\cup \{v\})\right)^{p} \right). 
    \end{align*}
    In particular, if $p > 1$, then 
    \begin{align}\label{equ:LEMMA:Lp-degree-expression}
        d_{\mathcal{H}}(v)
        \ge \frac{d_{\mathcal{H},t,p}(v)}{p \binom{r}{t} n^{(p-1)(r-t)}}. 
    \end{align}
\end{lemma}
\begin{proof}[Proof of Lemma~\ref{LEMMA:Lp-degree-expression}]
    It follows from the definition of $(t,p)$-degree that 
    \begin{align*}
        d_{\mathcal{H},t,p}(v)
        & = \sum_{T\in \partial_{r-t}\mathcal{H}} d_{\mathcal{H}}^{p}(T) - 
            \sum_{T\in \partial_{r-t}(\mathcal{H}-v)} d_{\mathcal{H}-v}^{p}(T)   \notag \\
        & = \sum_{\substack{T\in \partial_{r-t}\mathcal{H} \colon \\ v\in T}} 
            d_{\mathcal{H}}^{p}(T)
            + \sum_{\substack{T\in \partial_{r-t}\mathcal{H} \colon \\ v\not\in T}} 
            \left( d_{\mathcal{H}}^{p}(T) - d_{\mathcal{H}-v}^{p}(T) \right) \notag \\
        & = \sum_{\substack{T\in \partial_{r-t}\mathcal{H} \colon \\ v\in T}} 
            d_{\mathcal{H}}^{p}(T)
            + \sum_{\substack{T\in \partial_{r-t}\mathcal{H} \colon \\ v\not\in T}} 
            \left( d_{\mathcal{H}}^{p}(T) - \left(d_{\mathcal{H}}(T) - d_{\mathcal{H}}(T\cup \{v\})\right)^{p} \right). 
    \end{align*}
    Notice that $d_{\mathcal{H}}^{p}(T) - \left(d_{\mathcal{H}}(T) - d_{\mathcal{H}}(T\cup \{v\})\right)^{p} \neq 0$ iff $d_{\mathcal{H}}(T\cup \{v\}) \ge 1$. Therefore, equations 
    \begin{align*}
        \partial_{r-t}L_{\mathcal{H}}(v)
        & = \{T\setminus\{v\} \colon T \in \partial_{r-t}\mathcal{H} \text{ and }  v\in T\}
        \quad\text{and}\quad  \\
        \partial_{r-t-1}L_{\mathcal{H}}(v) 
        & = \left\{T \in \partial_{r-t}\mathcal{H} \colon v\not\in T \text{ and } d_{\mathcal{H}}(T\cup \{v\}) \ge 1 \right\} 
    \end{align*}
    imply the desired form of $d_{\mathcal{H},t,p}(v)$. 

    By~\eqref{equ:FACT:inequality-a-4}, we obtain 
    \begin{align*}
        d_{\mathcal{H},t,p}(v)
        & = \sum_{S \in \partial_{r-t}L_{\mathcal{H}}(v)} 
            d_{\mathcal{H}}^{p}(S\cup\{v\})
            + \sum_{T \in \partial_{r-t-1}L_{\mathcal{H}}(v)} 
            \left( d_{\mathcal{H}}^{p}(T) - \left(d_{\mathcal{H}}(T) - d_{\mathcal{H}}(T\cup \{v\})\right)^{p} \right) \\
        & \le \sum_{S \in \partial_{r-t}L_{\mathcal{H}}(v)} 
            d_{\mathcal{H}}^{p-1+1}(S\cup\{v\}) 
            + \sum_{T \in \partial_{r-t-1}L_{\mathcal{H}}(v)}  p \cdot d_{\mathcal{H}}^{p-1}(T) \cdot d_{\mathcal{H}}(T\cup \{v\}) \\
        & \le \sum_{S \in \partial_{r-t}L_{\mathcal{H}}(v)} 
            d_{\mathcal{H}}(S\cup\{v\}) \cdot n^{(p-1)(r-t)} 
            + \sum_{T \in \partial_{r-t-1}L_{\mathcal{H}}(v)}  p n^{(p-1)(r-t)} \cdot d_{\mathcal{H}}(T\cup \{v\}) \\
        & \le p n^{(p-1)(r-t)} \left(\sum_{S \in \partial_{r-t}L_{\mathcal{H}}(v)} 
            d_{\mathcal{H}}(S\cup\{v\}) 
            + \sum_{T \in \partial_{r-t-1}L_{\mathcal{H}}(v)}  d_{\mathcal{H}}(T\cup \{v\})\right) \\
        & = p n^{(p-1)(r-t)} \left( \binom{r-1}{t-1}d_{\mathcal{H}}(v) + \binom{r-1}{t} d_{\mathcal{H}}(v) \right)
        = p n^{(p-1)(r-t)} \binom{r}{t} d_{\mathcal{H}}(v), 
    \end{align*}
    which implies~\eqref{equ:LEMMA:Lp-degree-expression}.
\end{proof}

The following result is a straightforward corollary of Lemma~\ref{LEMMA:Lp-degree-expression}. 
Its proof can be found in 
Section~\ref{APPENDIX:proof:CORO:Lp-degree-expression-Lagrangian} of the Appendix. 

\begin{corollary}\label{CORO:Lp-degree-expression-Lagrangian}
    Let $r > t \ge 1$ be integers and $p \ge 1$ be a real number. 
    Suppose that $V_1 \cup \cdots \cup V_m = [n]$ is a partition and $\mathcal{H} = \mathcal{G}(V_1, \ldots, V_m)$ is a blowup of an $r$-graph $\mathcal{G}$ on $[m]$. Then for every $i \in [m]$ and $v\in V_i$, 
    \begin{align*}
        D_i L_{\mathcal{G},t,p}(x_1, \ldots, x_m) - o(1)
        \le \frac{d_{t,p,\mathcal{H}}(v)}{n^{t-1+p(r-t)}}
        \le D_i L_{\mathcal{G},t,p}(x_1, \ldots, x_m).  
    \end{align*}
    where $x_i \coloneqq |V_i|/n$ for $i\in [m]$. 
\end{corollary}

In the following lemma, we show that the $(t,p)$-norm is locally monotone, as defined in~\cite{CL24}.

\begin{lemma}\label{LEMMA:local-monotone}
    Let $r > t \ge 1$ be integers and $p \ge 1$ be a real number. 
    Suppose that $\mathcal{H}$ is an $r$-graph and $\mathcal{H}' \subset \mathcal{H}$ is a subgraph. 
    Then for every $v\in V(\mathcal{H}')$, 
    \begin{align}\label{equ:LEMMA:local-monotone}
        d_{\mathcal{H},t,p}(v) - d_{\mathcal{H}',t,p}(v)
        \ge \sum_{T} 
            \left( d_{\mathcal{H}}^{p}(T) - \left(d_{\mathcal{H}}(T) - d_{\mathcal{H}}(T\cup \{v\})\right)^{p} \right), 
    \end{align}
    where the summation is taken over $\partial_{r-t-1}L_{\mathcal{H}}(v)\setminus \partial_{r-t-1}L_{\mathcal{H}'}(v)$. 
    In particular, 
    \begin{align*}
        d_{\mathcal{H},t,p}(v) \ge d_{\mathcal{H}',t,p}(v)
        \quad\text{for every } v\in V(\mathcal{H}'). 
    \end{align*}
\end{lemma}
\begin{proof}[Proof of Lemma~\ref{LEMMA:local-monotone}]
    Fix $v\in V(\mathcal{H}')$. 
    It is clear that $d_{\mathcal{H}}(S\cup\{v\}) \ge d_{\mathcal{H}'}(S\cup\{v\})$ for every $S \in \partial_{r-t}L_{\mathcal{H}'}(v)$. 
    So by Lemma~\ref{LEMMA:Lp-degree-expression}, to prove~\eqref{equ:LEMMA:local-monotone}, it suffices to show that for every $T \in \partial_{r-t-1}L_{\mathcal{H}'}(v)$, 
    \begin{align*}
        d_{\mathcal{H}}^{p}(T) - \left(d_{\mathcal{H}}(T) - d_{\mathcal{H}}(T\cup\{v\})\right)^{p}
        \ge d_{\mathcal{H}'}^{p}(T) - \left(d_{\mathcal{H}'}(T) - d_{\mathcal{H}'}(T\cup\{v\})\right)^{p}. 
    \end{align*}
    Let $\mathcal{G} \coloneqq \mathcal{H}\setminus \mathcal{H}'$. 
    Fix $T \in \partial_{r-t-1}L_{\mathcal{H}'}(v)$. 
    Notice that 
    \begin{align*}
        & \quad d_{\mathcal{H}}^{p}(T) - \left(d_{\mathcal{H}}(T) - d_{\mathcal{H}}(T\cup\{v\})\right)^{p} \\
        & = \left(d_{\mathcal{H}'}(T) + d_{\mathcal{G}}(T)\right)^{p}
        - \left(\left(d_{\mathcal{H}'}(T) + d_{\mathcal{G}}(T)\right) - \left(d_{\mathcal{H}'}(T\cup\{v\}) + d_{\mathcal{G}}(T\cup \{v\})\right)\right)^{p} \\
        & = \left(d_{\mathcal{H}'}(T) + d_{\mathcal{G}}(T)\right)^{p}
        - \left(d_{\mathcal{H}'}(T) - d_{\mathcal{H}'}(T\cup\{v\}) + d_{\mathcal{G}}(T) - d_{\mathcal{G}}(T\cup \{v\})\right)^{p} \\
        & \ge \left(d_{\mathcal{H}'}(T) + d_{\mathcal{G}}(T)\right)^{p}
        - \left(d_{\mathcal{H}'}(T) - d_{\mathcal{H}'}(T\cup\{v\}) + d_{\mathcal{G}}(T) \right)^{p}. 
    \end{align*}
    Consider the polynomial 
    \begin{align*}
        g(X) 
        \coloneqq \left(A + X\right)^{p} - (B + X)^{p}, 
    \end{align*}
    where $A \coloneqq d_{\mathcal{H}'}(T)$ and $B \coloneqq d_{\mathcal{H}'}(T) - d_{\mathcal{H}'}(T\cup\{v\})$. 
    Since $A \ge B \ge 0$ and $p \ge 1$, we have 
    \begin{align*}
        D g(x)
        = p \left((A+x)^{p-1} - (B+x)^{p-1}\right)
        \ge 0
        \quad\text{for}\quad x \ge 0. 
    \end{align*}
    Therefore, by the monotonicity of $g(x)$, 
    \begin{align*}
        & \quad d_{\mathcal{H}}^{p}(T) - \left(d_{\mathcal{H}}(T) - d_{\mathcal{H}}(T\cup\{v\})\right)^{p} \\
        & \ge \left(d_{\mathcal{H}'}(T) + d_{\mathcal{G}}(T)\right)^{p}
        - \left(\left(d_{\mathcal{H}'}(T) - d_{\mathcal{H}'}(T\cup\{v\})\right) + d_{\mathcal{G}}(T) \right)^{p} \\
        & \ge d_{\mathcal{H}'}^{p}(T) 
        - \left(\left(d_{\mathcal{H}'}(T) - d_{\mathcal{H}'}(T\cup\{v\})\right)\right)^{p}, 
    \end{align*}
    proving Lemma~\ref{LEMMA:local-monotone}. 
\end{proof}

In the following lemma, we establish several combinatorial equalities that will be useful for subsequent estimations.

\begin{lemma}\label{LEMMA:tail-estimate-Lp-degree-sum}
    Let $r > t \ge 1 $ be integers and $p > 0$ be a real number. 
    For every $n$-vertex $r$-graph $\mathcal{H}$ and for every $v\in V(\mathcal{H})$,
    \begin{align}\label{equ:double-count-0}
            \sum_{v\in V(\mathcal{H})} \sum_{S \in \partial_{r-t}L_{\mathcal{H}}(v)} 
            d_{\mathcal{H}}^{p}(S\cup\{v\})
            = t \cdot \norm{\mathcal{H}}_{t,p}, 
    \end{align}    
    \begin{align}\label{equ:double-count-1}
        \sum_{v\in V(\mathcal{H})} \sum_{T \in \partial_{r-t-1}L_{\mathcal{H}}(v)} d_{\mathcal{H}}(T\cup \{v\}) \cdot d_{\mathcal{H}}^{p-1}(T)
        = (r-t) \cdot \norm{\mathcal{H}}_{t,p},
    \end{align}
    and 
    \begin{align}\label{equ:double-count-2}
        \sum_{v\in V(\mathcal{H})} \sum_{T \in \partial_{r-t-1}L_{\mathcal{H}}(v)} 
        \left(\frac{ d_{\mathcal{H}}(T\cup \{v\})}{d_{\mathcal{H}}(T)} \right)^2 d_{\mathcal{H}}^{p}(T) 
        = O(n^{t+p(r-t)-\delta_p}), 
    \end{align}
    where 
    \begin{align*}
        \delta_p
        \coloneqq 
        \begin{cases}
            \frac{p}{4}, & \quad\text{if}\quad p \in (0,1], \\
            \frac{p-1}{2p}, & \quad\text{if}\quad p >1. 
        \end{cases}
    \end{align*}
\end{lemma}
\begin{proof}[Proof of Lemma~\ref{LEMMA:tail-estimate-Lp-degree-sum}]
    Equation~\eqref{equ:double-count-0} follows from the following simple double counting$\colon$
    \begin{align*}
        \sum_{v\in V(\mathcal{H})} \sum_{S \in \partial_{r-t}L_{\mathcal{H}}(v)} 
            d_{\mathcal{H}}^{p}(S\cup\{v\})
        = \sum_{T \in \partial_{r-t}\mathcal{H}} 
            |T| \cdot d_{\mathcal{H}}^{p}(T)
        = t \sum_{T \in \partial_{r-t}\mathcal{H}} 
             d_{\mathcal{H}}^{p}(T). 
    \end{align*}
    Notice that for every $T \in \partial_{r-t}\mathcal{H}$, 
    \begin{align*}
         \sum_{v\in V(\mathcal{H})\setminus T} d_{\mathcal{H}}(T\cup \{v\}) 
         = (r-t) \cdot |L_{\mathcal{H}}(T)|
         = (r-t) \cdot d_{\mathcal{H}}(T). 
         % \quad\text{for all } T \in \partial_{r-t}\mathcal{H}. 
    \end{align*}
    Therefore, 
    \begin{align*}
        \sum_{v\in V(\mathcal{H})} \sum_{T \in \partial_{r-t-1}L_{\mathcal{H}}(v)} d_{\mathcal{H}}(T\cup \{v\}) \cdot d_{\mathcal{H}}^{p-1}(T)
        & =  \sum_{T \in \partial_{r-t} \mathcal{H}} \sum_{v\in V(\mathcal{H})\setminus T} d_{\mathcal{H}}(T\cup \{v\}) \cdot d_{\mathcal{H}}^{p-1}(T) \\
        & = \sum_{T \in \partial_{r-t} \mathcal{H}} (r-t) \cdot d_{\mathcal{H}}(T) \cdot d_{\mathcal{H}}^{p-1}(T) \\
        & = (r-t) \cdot \norm{\mathcal{H}}_{t,p}, 
    \end{align*}
    proving~\eqref{equ:double-count-1}. 
    So it suffices to prove~\eqref{equ:double-count-2}. 
    Let 
    \begin{align*}
        \delta
        \coloneqq 
        \begin{cases}
            \frac{1}{4}, & \quad\text{if}\quad p \in (0,1], \\
            \frac{p-1}{2p}, & \quad\text{if}\quad p >1. 
        \end{cases}
    \end{align*}
    Let us partition $\partial_{r-t-1}L_{\mathcal{H}}(v)$ into two sets$\colon$
    \begin{align*}
        \mathcal{T}_{1}
        \coloneqq \left\{T\in \partial_{r-t-1}L_{\mathcal{H}}(v) \colon d_{\mathcal{H}}(T) \ge n^{r-t-1+\delta} \right\}
        \quad\text{and}\quad 
        \mathcal{T}_{2} 
        \coloneqq \partial_{r-t-1}L_{\mathcal{H}}(v) \setminus \mathcal{T}_{1}.
    \end{align*}
    Observe that $\frac{ d_{\mathcal{H}}(T\cup \{v\})}{d_{\mathcal{H}}(T)} \le \frac{n^{r-t-1}}{n^{r-t-1+\delta}} = \frac{1}{n^{\delta}}$ for every $T \in \mathcal{T}_1$. So it follows from~\eqref{equ:double-count-1} that 
    \begin{align}\label{equ:LEMMA-tail-estimate-Lp-degree-sum-1}
        & \quad \sum_{v\in V(\mathcal{H})} \sum_{T \in \partial_{r-t-1}L_{\mathcal{H}}(v)} 
        \left(\frac{ d_{\mathcal{H}}(T\cup \{v\})}{d_{\mathcal{H}}(T)} \right)^2 d_{\mathcal{H}}^{p}(T) \notag \\
        & = \sum_{v\in V(\mathcal{H})} \sum_{T \in \mathcal{T}_1} 
        \left(\frac{ d_{\mathcal{H}}(T\cup \{v\})}{d_{\mathcal{H}}(T)} \right)^2 d_{\mathcal{H}}^{p}(T)
        + \sum_{v\in V(\mathcal{H})} \sum_{T \in \mathcal{T}_2} 
        \left(\frac{ d_{\mathcal{H}}(T\cup \{v\})}{d_{\mathcal{H}}(T)} \right)^2 d_{\mathcal{H}}^{p}(T) \notag \\
        & = \sum_{v\in V(\mathcal{H})} \sum_{T \in \mathcal{T}_1} \frac{ d_{\mathcal{H}}(T\cup \{v\})}{d_{\mathcal{H}}(T)} \cdot 
        d_{\mathcal{H}}(T\cup \{v\}) \cdot d_{\mathcal{H}}^{p-1}(T)  
        + \sum_{v\in V(\mathcal{H})} \sum_{T \in \mathcal{T}_2} 
        \left(\frac{ d_{\mathcal{H}}(T\cup \{v\})}{d_{\mathcal{H}}(T)} \right)^2 d_{\mathcal{H}}^{p}(T) \notag \\
        & \le \frac{1}{n^{\delta}}\sum_{v\in V(\mathcal{H})} \sum_{T \in \mathcal{T}_1} 
        d_{\mathcal{H}}(T\cup \{v\}) \cdot d_{\mathcal{H}}^{p-1}(T)  
        + \sum_{v\in V(\mathcal{H})} \sum_{T \in \mathcal{T}_2} 
        \left(\frac{ d_{\mathcal{H}}(T\cup \{v\})}{d_{\mathcal{H}}(T)} \right)^2 d_{\mathcal{H}}^{p}(T) \notag \\
        & \le \frac{(r-t) \norm{\mathcal{H}}_{t,p}}{n^{\delta}} 
         + \sum_{v\in V(\mathcal{H})} \sum_{T \in \mathcal{T}_2} 
        \left(\frac{ d_{\mathcal{H}}(T\cup \{v\})}{d_{\mathcal{H}}(T)} \right)^2 d_{\mathcal{H}}^{p}(T). 
    \end{align}

    \textbf{Case 1}$\colon$ $p > 1$. 

        Inequality~\eqref{equ:LEMMA-tail-estimate-Lp-degree-sum-1} continues as 
        \begin{align*}
            &\quad \sum_{v\in V(\mathcal{H})} \sum_{T \in \partial_{r-t-1}L_{\mathcal{H}}(v)} 
            \left(\frac{ d_{\mathcal{H}}(T\cup \{v\})}{d_{\mathcal{H}}(T)} \right)^2 d_{\mathcal{H}}^{p}(T) \\
            & \le \frac{(r-t) \norm{\mathcal{H}}_{t,p}}{n^{\delta}}  + \sum_{v\in V(\mathcal{H})} \sum_{T \in \mathcal{T}_2}  d_{\mathcal{H}}^{p}(T) \\
            & \le \frac{(r-t) \norm{\mathcal{H}}_{t,p}}{n^{\delta}}
            + n^{t+1} \left(n^{r-t-1+\delta}\right)^{p}  \\
            & \le n^{t+p(r-t)-\delta} + n^{t+p(r-t)+p\delta - (p-1)} 
            \le 2 n^{t+p(r-t)-\frac{p-1}{2p}}. 
        \end{align*}
        
    \medskip

    \textbf{Case 2}$\colon$ $p \in (0,1]$.
    
    Let $\alpha \coloneqq 1 - \frac{p}{2} \in (0,1)$. Inequality~\eqref{equ:LEMMA-tail-estimate-Lp-degree-sum-1} continues as 
    \begin{align}\label{equ:LEMMA-tail-estimate-Lp-degree-sum-2}
        & \quad \sum_{v\in V(\mathcal{H})} \sum_{T \in \partial_{r-t-1}L_{\mathcal{H}}(v)} 
        \left(\frac{ d_{\mathcal{H}}(T\cup \{v\})}{d_{\mathcal{H}}(T)} \right)^2 d_{\mathcal{H}}^{p}(T) \notag \\
        & \le \frac{(r-t) \norm{\mathcal{H}}_{t,p}}{n^{\delta}} 
         + \sum_{v\in V(\mathcal{H})} \sum_{T\in \mathcal{T}_2}  \left(\frac{ d_{\mathcal{H}}(T\cup \{v\})}{d_{\mathcal{H}}(T)} \right)^{2-\alpha} d_{\mathcal{H}}^{\alpha}(T\cup\{v\}) \cdot d_{\mathcal{H}}^{p-\alpha}(T)  \notag \\
        % d_{\mathcal{H}}^{p+2-p}(T\cup \{v\}) \cdot d_{\mathcal{H}}^{p-2}(T) \\
        & \le \frac{n^{t+p(r-t)}}{n^{\delta}} 
         + \sum_{v\in V(\mathcal{H})} \sum_{T\in \mathcal{T}_2}  d_{\mathcal{H}}^{\alpha}(T\cup\{v\}) \cdot d_{\mathcal{H}}^{p-\alpha}(T)  \notag \\
        & = n^{t+p(r-t)-\frac{1}{4}} 
        + \sum_{T\in \mathcal{T}_2}  d_{\mathcal{H}}^{p-\alpha}(T) \sum_{v \in V(\mathcal{H})\setminus T} 
        d_{\mathcal{H}}^{\alpha}(T\cup \{v\}). 
    \end{align}
    Fix $T \in \mathcal{T}_2$ and let $\mathcal{G} \coloneqq L_{\mathcal{H}}(T)$. Notice that $d_{\mathcal{G}}(v) = d_{\mathcal{H}}(T\cup \{v\})$, and hence, by Jensen's inequality, 
    \begin{align*}
        \sum_{v \in V(\mathcal{H})\setminus T} d_{\mathcal{H}}^{\alpha}(T\cup \{v\})
        \le (n-t)\left(\frac{\sum_{v \in V(\mathcal{H})\setminus T} d_{\mathcal{H}}(T\cup \{v\})}{n-t}\right)^{\alpha}
        \le (r-t)^{\alpha} \cdot d_{\mathcal{H}}^{\alpha}(T) \cdot n^{1-\alpha}. 
    \end{align*}
    Therefore,~\eqref{equ:LEMMA-tail-estimate-Lp-degree-sum-2} continues as 
    \begin{align*}
        \sum_{T\in \mathcal{T}_2}  d_{\mathcal{H}}^{p-\alpha}(T) \sum_{v \in V(\mathcal{H})\setminus T} 
        d_{\mathcal{H}}^{\alpha}(T\cup \{v\})
        & \le (r-t)^{\alpha} n^{1-\alpha} \cdot \sum_{T\in \mathcal{T}_2}  d_{\mathcal{H}}^{p}(T) \\
        & \le r n^{1-\alpha} \cdot n^{p(r-t-1+\delta)} \cdot n^t \\
        & = r n^{t+p(r-t) + 1-\alpha - p + p\delta} 
         = r n^{t+p(r-t) - \frac{p}{4}}. 
        % & \le \frac{(r-t)^{\alpha}  n^{p\delta}}{n^{p/2}} \cdot n^{t+p(r-t)}, 
    \end{align*}
    % where the last inequality follows from $1-\alpha-p = p/2$.
    Therefore, 
    \begin{align*}
        \sum_{v\in V(\mathcal{H})} \sum_{\substack{T\in \partial_{r-t}\mathcal{H} \colon \\ v\not\in T}} 
        \left(\frac{ d_{\mathcal{H}}(T\cup \{v\})}{d_{\mathcal{H}}(T)} \right)^2 d_{\mathcal{H}}^{p}(T)
         \le n^{t+p(r-t)-\frac{1}{4}} 
            + r n^{t+p(r-t) - \frac{p}{4}} 
         \le 2r n^{t+p(r-t) - \frac{p}{4}}. 
    \end{align*}
    This completes the proof of~\eqref{equ:double-count-2}. 
\end{proof}

In the following lemma, we show that the map $\Gamma$ (defined at the beginning of this section) is $(t+p(r-t))$-uniform, as defined in~\cite{CL24}.

\begin{lemma}\label{LEMMA:Lp-uniform}
    Let $r > t \ge 1 $ be integers and $p > 0$ be a real number. 
    For every $n$-vertex $r$-graph $\mathcal{H}$, 
    \begin{align*}
        \left|\sum_{v\in V(\mathcal{H})}d_{\mathcal{H},t,p}(v) - \left(t+p(r-t)\right) \cdot \norm{\mathcal{H}}_{t,p} \right|
        = O(n^{t+p(r-t)-\delta_p}), 
    \end{align*}
    where $\delta_p>0$ is the same constant as defined in Lemma~\ref{LEMMA:tail-estimate-Lp-degree-sum}. 
\end{lemma}
\begin{proof}[Proof of Lemma~\ref{LEMMA:Lp-uniform}]
    Let $v\in V(\mathcal{H})$ be a vertex and $T \in \partial_{r-t-1}L_{\mathcal{H}}(v)$ be a $t$-set. 
    Since $\frac{d_{\mathcal{H}}(T\cup \{v\})}{d_{\mathcal{H}}(T)} \in (0,1]$, it follows from~\eqref{equ:FACT:inequality-a-3} that 
        \begin{align*}
            & \quad \left| d_{\mathcal{H}}^{p}(T) - \left(d_{\mathcal{H}}(T) - d_{\mathcal{H}}(T\cup \{v\})\right)^{p} - p \cdot d_{\mathcal{H}}(T\cup \{v\}) \cdot d_{\mathcal{H}}^{p-1}(T) \right| \\
            & = d_{\mathcal{H}}^{p}(T) \cdot \left| 1- \left(1 - \frac{d_{\mathcal{H}}(T\cup \{v\})}{d_{\mathcal{H}}(T)}\right)^{p} - p \cdot \frac{d_{\mathcal{H}}(T\cup \{v\})}{d_{\mathcal{H}}(T)}\right| \\ 
            & \le d_{\mathcal{H}}^{p}(T) \cdot (p^2+1) \cdot \left(\frac{d_{\mathcal{H}}(T\cup \{v\})}{d_{\mathcal{H}}(T)}\right)^2. 
        \end{align*}
    Combining with Lemma~\ref{LEMMA:Lp-degree-expression},~\eqref{equ:double-count-0},~\eqref{equ:double-count-1}, and~\eqref{equ:double-count-2}, we obtain 
        \begin{align*}
            & \quad \left| \sum_{v\in V(\mathcal{H})} d_{\mathcal{H},t,p}(v) - 
            (t+p(r-t)) \cdot \norm{\mathcal{H}}_{t,p} \right| \\
            & \le  \sum_{v\in V(\mathcal{H})} \sum_{T \in \partial_{r-t-1}L_{\mathcal{H}}(v)} 
            \left| \left( d_{\mathcal{H}}^{p}(T) - \left(d_{\mathcal{H}}(T) - d_{\mathcal{H}}(T\cup \{v\})\right)^{p} \right) 
            - p \cdot d_{\mathcal{H}}(T\cup \{v\}) \cdot d_{\mathcal{H}}^{p-1}(T) \right| \\
            & \le (p^2+1) \cdot \sum_{v\in V(\mathcal{H})} \sum_{T \in \partial_{r-t-1}L_{\mathcal{H}}(v)}  
            \left(\frac{d_{\mathcal{H}}(T\cup \{v\})}{d_{\mathcal{H}}(T)}\right)^2 
            \cdot d_{\mathcal{H}}^{p}(T)
            = O(n^{t+p(r-t)-\delta_p}),  
        \end{align*}
    completing the proof of Lemma~\ref{LEMMA:Lp-uniform}. 
\end{proof}

Next, 
we show that the map $\Gamma$ is continuous, as defined in~\cite{CL24}.
We will need the following simple inequality in the proof. 

\begin{proposition}\label{PROP:tp-norm-Jensen-small}
    Suppose that $r > t \ge 1$ are integers and $p_2 > p_1 > 0$ are real numbers. 
    Then for every $r$-graph $\mathcal{H}$, 
    \begin{align*}
        \norm{\mathcal{H}}_{t,p_1}
        \le \norm{\mathcal{H}}_{t,p_2}^{\frac{p_1}{p_2}} \cdot |\partial_{r-t}\mathcal{H}|^{1-\frac{p_1}{p_2}}. 
    \end{align*}
    In particular, for every $p \in (0,1)$, 
    \begin{align}\label{equ:PROP:tp-norm-Jensen-small}
        \norm{\mathcal{H}}_{t,p}
        \le \left(\binom{r}{t} \cdot |\mathcal{H}|\right)^{p} \cdot |\partial_{r-t} \mathcal{H}|^{1-p}
        \le \left(\binom{r}{t} \cdot |\mathcal{H}|\right)^{p} \cdot n^{t(1-p)}. 
    \end{align}
\end{proposition}
\begin{proof}[Proof of Proposition~\ref{PROP:tp-norm-Jensen-small}]
    Since $p_1/p_2 \in (0,1)$, it follows from Jensen's inequality that 
    \begin{align*}
        \norm{\mathcal{H}}_{t,p_1}
        & = \sum_{T \in \partial_{r-t}\mathcal{H}} \left(d_{\mathcal{H}}^{p_2}(T)\right)^{\frac{p_1}{p_2}} \\
        & \le \left(\frac{\sum_{T \in \partial_{r-t}\mathcal{H}} d_{\mathcal{H}}^{p_2}(T)}{|\partial_{r-t}\mathcal{H}|}\right)^{\frac{p_1}{p_2}} \cdot |\partial_{r-t}\mathcal{H}|
         = \norm{\mathcal{H}}_{t,p_2}^{\frac{p_1}{p_2}} \cdot |\partial_{r-t}\mathcal{H}|^{1-\frac{p_1}{p_2}}, 
    \end{align*}
    proving Proposition~\ref{PROP:tp-norm-Jensen-small}. 
\end{proof}

\begin{lemma}\label{LEMMA:Lp-basic-property-global}
    Let $r > t \ge 1 $ be integers, $\mathcal{H}$ be an $n$-vertex $r$-graph, and $\mathcal{H}' \subset \mathcal{H}$ be a subgraph.
    \begin{enumerate}[label=(\roman*)]
        \item\label{LEMMA:Lp-basic-property-global-1} If $p \ge 1$, then 
            \begin{align*}
                \norm{\mathcal{H}}_{t,p} - \norm{\mathcal{H}'}_{t,p}
                \le p  \binom{r}{t} \cdot |\mathcal{H} \setminus \mathcal{H}'| \cdot n^{(p-1)(r-t)}. 
            \end{align*}
        \item\label{LEMMA:Lp-basic-property-global-2} If $p \in (0,1)$, then 
            \begin{align*}
                \norm{\mathcal{H}}_{t,p} - \norm{\mathcal{H}'}_{t,p}
                \le \norm{\mathcal{H}\setminus \mathcal{H}'}_{t,p}
                \le \left(\binom{r}{t} \cdot |\mathcal{H}\setminus \mathcal{H}'|\right)^{p} \cdot n^{t(1-p)}.
            \end{align*}
    \end{enumerate}
    In particular, if $|\mathcal{H} \setminus \mathcal{H}'| = o(n^r)$, then $\norm{\mathcal{H}}_{t,p} -\norm{\mathcal{H}'}_{t,p} = o(n^{t+p(r-t)})$ for every $p > 0$. 
\end{lemma}
\begin{proof}[Proof of Lemma~\ref{LEMMA:Lp-basic-property-global}]
    Notice from the definition that 
    \begin{align}\label{equ:LEMMA-Lp-basic-property-1}
        \norm{\mathcal{H}}_{t,p} - \norm{\mathcal{H}'}_{t,p}
        & = \sum_{T \in \partial_{r-t}\mathcal{H}} \left(d_{\mathcal{H}}^{p}(T) - d_{\mathcal{H}'}^{p}(T)\right) \notag \\
        & = \sum_{T \in \partial_{r-t}\mathcal{H}} \left(d_{\mathcal{H}}^{p}(T) - \left(d_{\mathcal{H}}(T) - d_{\mathcal{H}\setminus \mathcal{H}'}(T) \right)^{p}\right) \notag \\
        & = \sum_{T \in \partial_{r-t}\mathcal{H}} \left(1 - \left(1 - \frac{d_{\mathcal{H}\setminus \mathcal{H}'}(T)}{d_{\mathcal{H}}(T)} \right)^{p}\right) \cdot d_{\mathcal{H}}^{p}(T). 
    \end{align}

    \textbf{Case 1}$\colon$ $p > 1$. 
    
    By Fact~\ref{FACT:inequality-a}~\ref{FACT:inequality-a-2}, Inequality~\eqref{equ:LEMMA-Lp-basic-property-1} continues as 
    \begin{align*}
        \norm{\mathcal{H}}_{t,p} - \norm{\mathcal{H}'}_{t,p}
        & \le \sum_{T \in \partial_{r-t}\mathcal{H}} p \cdot \frac{d_{\mathcal{H}\setminus \mathcal{H}'}(T)}{d_{\mathcal{H}}(T)} \cdot d_{\mathcal{H}}^{p}(T) \\
        & = \sum_{T \in \partial_{r-t}\mathcal{H}} p \cdot d_{\mathcal{H}\setminus \mathcal{H}'}(T) \cdot d_{\mathcal{H}}^{p-1}(T) \\
        & \le \sum_{T \in \partial_{r-t}\mathcal{H}} p \cdot d_{\mathcal{H}\setminus \mathcal{H}'}(T) \cdot n^{(p-1)(r-t)}
        = p n^{(p-1)(r-t)} \binom{r}{t} \cdot |\mathcal{H} \setminus \mathcal{H}'|. 
    \end{align*}

    \medskip

    \textbf{Case 2}$\colon$ $p \in (0,1)$.

    Inequality~\eqref{equ:LEMMA-Lp-basic-property-1} continues as
    \begin{align*}
        \norm{\mathcal{H}}_{t,p} - \norm{\mathcal{H}'}_{t,p}
        & \le \sum_{T \in \partial_{r-t}\mathcal{H}} \left(1 - \left(1 - \left(\frac{d_{\mathcal{H}\setminus \mathcal{H}'}(T)}{d_{\mathcal{H}}(T)} \right)^{p}\right) \right) \cdot d_{\mathcal{H}}^{p}(T) \\
        & = \sum_{T \in \partial_{r-t}\mathcal{H}}  d_{\mathcal{H}\setminus \mathcal{H}'}^{p}(T)
        = \norm{\mathcal{H}\setminus\mathcal{H}'}_{t,p}
        \le \left(\binom{r}{t} \cdot |\mathcal{H}\setminus \mathcal{H}'|\right)^{p} \cdot n^{t(1-p)}, 
    \end{align*}
    where the first inequality follows from Fact~\ref{FACT:inequality-b}~\ref{FACT:inequality-b-1} and the last inequality follows from Inequality~\eqref{equ:PROP:tp-norm-Jensen-small}. 
\end{proof}

In the following lemma, we show that the map $\Gamma$ is locally Lipschitz, as defined in~\cite{CL24}.

\begin{lemma}\label{LEMMA:Lp-basic-property-local-Lipschitz}
    Let $r > t \ge 1 $ be integers, $\mathcal{H}$ be an $n$-vertex $r$-graph, and $B \subset V(\mathcal{H})$ be a vertex set. 
    \begin{enumerate}[label=(\roman*)]
        \item\label{LEMMA:Lp-basic-property-local-Lipschitz-1} If $p \ge 2$, then for every $v \in V(\mathcal{H}) \setminus B$, 
        \begin{align*}
            d_{\mathcal{H}, t, p}(v) - d_{\mathcal{H}-B, t, p}(v)
            \le 2^{r+1} p^2 |B| n^{t-2+p(r-t)}. 
        \end{align*}
        \item\label{LEMMA:Lp-basic-property-local-Lipschitz-2} If $p \in (1,2)$, then for every $v \in V(\mathcal{H}) \setminus B$, 
        \begin{align*}
            d_{\mathcal{H}, t, p}(v) - d_{\mathcal{H}-B, t, p}(v)
            \le 10 p \left(\frac{|B|}{n}\right)^{p-1} n^{t-1+p(r-t)} + o(n^{t-1+p(r-t)}). 
        \end{align*}
    \end{enumerate}
\end{lemma}
\begin{proof}[Proof of Lemma~\ref{LEMMA:Lp-basic-property-local-Lipschitz}]
    Let $V \coloneqq V(\mathcal{H})$, $U \coloneqq V\setminus B$, $\mathcal{G} \coloneqq \mathcal{H}[U]$, and fix a vertex $v \in U$. 
    % By enlarging the set $B$ if necessary, we may assume that $|B| \gg 1$. 
    % Fix $v \in U$. 
    Let $f(X_1, X_2) \coloneqq X_1^p - (X_1 - X_2)^{p}$.  Simple calculations show that 
    \begin{align*}
        D_{1}f(X_1, X_2)
         = p \left(X_{1}^{p-1} - (X_1-X_2)^{p-1}\right) \quad\text{and}\quad 
        D_{2}f(X_1, X_2)
         = p (X_1 - X_2)^{p-1}.
    \end{align*}
    Let $\mathcal{B} \coloneqq \left\{T\in \partial_{r-t-1}L_{\mathcal{H}}(v) \colon T\cap B \neq \emptyset \right\}$. 
    Notice that 
    \begin{align*}
        |\mathcal{B}| \le |B| n^{t-1}, 
        \quad\text{and}\quad 
        \partial_{r-t-1}L_{\mathcal{H}}(v)
        =  \mathcal{B}
         \cup \partial_{r-t-1}L_{\mathcal{G}}(v).   
    \end{align*}
    By Lemma~\ref{LEMMA:Lp-degree-expression}, we obtain  
    \begin{align*}
        d_{\mathcal{H},t,p}(v)
        & = \sum_{S \in \partial_{r-t}L_{\mathcal{H}}(v)} 
            d_{\mathcal{H}}^{p}(S\cup\{v\})
            + \sum_{T \in \partial_{r-t-1}L_{\mathcal{H}}(v)} 
            f\left(d_{\mathcal{H}}(T), d_{\mathcal{H}}(T\cup \{v\})\right) \\
        & = \norm{L_{\mathcal{H}}(v)}_{t-1,p} 
            + \left(\sum_{T \in \mathcal{B}} + \sum_{T \in \partial_{r-t-1}L_{\mathcal{G}}(v)} \right) f\left(d_{\mathcal{H}}(T), d_{\mathcal{H}}(T\cup \{v\})\right). 
    \end{align*}
    First, it follows from Lemma~\ref{LEMMA:Lp-basic-property-global}~\ref{LEMMA:Lp-basic-property-global-1} that 
    \begin{align*}
        \norm{L_{\mathcal{H}}(v)}_{t-1,p} - \norm{L_{\mathcal{G}}(v)}_{t-1,p}
        & \le p \binom{r-1}{t-1} \cdot \left(|L_{\mathcal{H}}(v)| - |L_{\mathcal{G}}(v)|\right) \cdot n^{(p-1)((r-1)-(t-1))} \\
        & \le p 2^r |B| n^{r-2} \cdot n^{(p-1)(r-t)}
        = p 2^r |B| n^{t-2+p(r-t)}. 
    \end{align*}
    Second, since (by~\eqref{equ:FACT:inequality-a-4})
    \begin{align*}
        f\left(d_{\mathcal{H}}(T), d_{\mathcal{H}}(T\cup \{v\})\right) 
        \le p \cdot d_{\mathcal{H}}^{p-1}(T) \cdot d_{\mathcal{H}}(T\cup \{v\})
        \le p n^{(r-t)(p-1)+(r-t-1)}
        = p n^{p(r-t)-1}, 
    \end{align*}
    we obtain 
    \begin{align*}
        \sum_{T \in \mathcal{B}} 
        f\left(d_{\mathcal{H}}(T), d_{\mathcal{H}}(T\cup \{v\})\right) 
        & \le |B| n^{t-1} \cdot p n^{p(r-t)-1} 
        = p|B|n^{t-2+p(r-t)}. 
    \end{align*}
    Therefore, to prove Lemma~\ref{LEMMA:Lp-basic-property-local-Lipschitz}, we have to upper bound 
    \begin{align*}
        \Delta
        \coloneqq \sum_{T \in \partial_{r-t-1}L_{\mathcal{G}}(v)} 
        \left( f\left(d_{\mathcal{H}}(T), d_{\mathcal{H}}(T\cup \{v\})\right) 
        - f\left(d_{\mathcal{G}}(T), d_{\mathcal{G}}(T\cup \{v\})\right) \right). 
    \end{align*}
    First, we consider the case $p \ge 2$ and prove Lemma~\ref{LEMMA:Lp-basic-property-local-Lipschitz}~\ref{LEMMA:Lp-basic-property-local-Lipschitz-1}.  

    \textbf{Case 1.1$\colon$} $t = r-1$. 

    Notice that in this case, $\partial_{r-t-1}L_{\mathcal{G}}(v) = L_{\mathcal{G}}(v)$ and $d_{\mathcal{H}}(T\cup \{v\}) = d_{\mathcal{G}}(T\cup \{v\}) = 1$ for every $T \in L_{\mathcal{G}}(v)$. 
    By the Mean Value Theorem, for every $T \in L_{\mathcal{G}}(v)$ there exists $z_T$ with $d_{\mathcal{G}}(T) \le z_{T} \le d_{\mathcal{H}}(T) \le n$ such that 
    \begin{align*}
        \Delta 
        & = \sum_{T \in L_{\mathcal{G}}(v)} 
            \left( d_{\mathcal{H}}^{p}(T) - \left(d_{\mathcal{H}}(T) - 1\right)^{p} \right) - \left( d_{\mathcal{G}}^{p}(T) - \left(d_{\mathcal{G}}(T) - 1\right)^{p} \right) \\
        & \le \sum_{T \in L_{\mathcal{G}}(v)}  p  \left(z_T^{p-1} - \left(z_T - 1\right)^{p-1}\right)  \left(d_{\mathcal{H}}(T) - d_{\mathcal{G}}(T)\right). 
        % & \le |L_{\mathcal{G}}(v)| \cdot p (p-1) \left(\max_{T} z_T^{p-2} \right) \left( \max_{T}  \left(d_{\mathcal{H}}(T) - d_{\mathcal{G}}(T)\right) \right) \\
        % & \le  n^{r-1} p^2 n^{p-2} |B| 
        % = p^2 |B| n^{r-3+p}. 
    \end{align*}
    Since $p-1\ge 1$, it follows from~\eqref{equ:FACT:inequality-a-4} that 
    \begin{align*}
        z_T^{p-1} - \left(z_T - 1\right)^{p-1}
        \le (p-1) z_{T}^{p-2}
        %\le (p-1) \cdot d_{\mathcal{H}}^{p-2}(T)
        \le p n^{p-2}. 
    \end{align*}
    Therefore, the above inequality on $\Delta$ continues as 
    \begin{align*}
        \Delta 
         \le |L_{\mathcal{G}}(v)| \cdot p \cdot p n^{p-2} \left( \max_{T}  \left(d_{\mathcal{H}}(T) - d_{\mathcal{G}}(T)\right) \right) 
         \le  n^{r-1} p^2 n^{p-2} |B| 
        = p^2 |B| n^{t-2+p(r-t)}. 
    \end{align*}

    \medskip 

    \textbf{Case 1.2$\colon$} $t \le r-2$.

    Similarly, by the Mean Value Theorem, for every $T \in \partial_{r-t-1}L_{\mathcal{G}}(v)$, there exists $(z_{1,T}, z_{2,T})$ with $d_{\mathcal{G}}(T) \le z_{1,T} \le d_{\mathcal{H}}(T) \le n^{r-t}$
     and $d_{\mathcal{G}}(T\cup \{v\}) \le z_{2,T} \le d_{\mathcal{H}}(T\cup \{v\}) \le n^{r-t-1}$ such that 
    \begin{align*}
        \Delta 
        & \le \sum_{T \in \partial_{r-t-1}L_{\mathcal{G}}(v)} \Big( D_1 f(z_{1,T}, z_{2,T}) \cdot \left(d_{\mathcal{H}}(T)- d_{\mathcal{G}}(T)\right) \\
        & \quad + D_2 f(z_{1,T}, z_{2,T}) \cdot \left(d_{\mathcal{H}}(T\cup \{v\})- d_{\mathcal{G}}(T\cup \{v\})\right) \Big) \\
        & \le \sum_{T \in \partial_{r-t-1}L_{\mathcal{G}}(v)}  \left( D_1 f(z_{1,T}, z_{2,T}) \cdot |B|n^{r-t-1}
        + D_2 f(z_{1,T}, z_{2,T}) \cdot |B|n^{r-t-2} \right).  
    \end{align*}
    Since $p-1 \ge 1$, it follows from~\eqref{equ:FACT:inequality-a-4} that 
    \begin{align*}
        D_1 f(z_{1,T}, z_{2,T})
        & = p \left(z_{1,T}^{p-1} - (z_{1,T} - z_{2,T})^{p-1}\right) \\
        & \le p(p-1) z_{1,T}^{p-2} z_{2,T}
        \le p^2 n^{(p-2)(r-t)+(r-t-1)}
        = p^2 n^{(p-1)(r-t)-1}. 
    \end{align*}
    In addition, 
    \begin{align*}
        D_2 f(z_{1,T}, z_{2,T})
        = p (z_{1,T} - z_{2,T})^{p-1}
        \le p n^{(p-1)(r-t)}. 
    \end{align*}
    Therefore, the above inequality on $\Delta$ continues as 
    \begin{align*}
        \Delta 
        & \le |\partial_{r-t-1}L_{\mathcal{G}}(v)| |B| \left(p^2 n^{(p-1)(r-t)-1 + (r-t-1)} + p n^{(p-1)(r-t)+(r-t-2)}\right) \\
        & \le |B| n^{t} \left(p^2 n^{p(r-t)-2} + p n^{p(r-t)-2}\right)  
        \le 2p^2 |B| n^{t-2+p(r-t)}. 
    \end{align*}
    This completes the proof for Lemma~\ref{LEMMA:Lp-basic-property-local-Lipschitz}~\ref{LEMMA:Lp-basic-property-local-Lipschitz-1}. 

    Next, we consider the case $p \in (1,2)$ and prove Lemma~\ref{LEMMA:Lp-basic-property-local-Lipschitz}~\ref{LEMMA:Lp-basic-property-local-Lipschitz-2}. 
    Let %$\delta \coloneqq \max\left\{\frac{|B|}{n}, \frac{10^8}{n} \right\}$, 
    %$\delta \coloneqq \frac{|B|}{n}$, 
    \begin{align*}
        \delta \coloneqq \frac{|B|}{n},
        \quad 
        \mathcal{T}_1 
         \coloneqq \left\{T \in \partial_{r-t-1}L_{\mathcal{G}}(v) \colon d_{\mathcal{H}}(T) \le 2 \delta n^{r-t} \right\},  
        \quad\text{and}\quad 
        \mathcal{T}_2 & \coloneqq \partial_{r-t-1}L_{\mathcal{G}}(v) \setminus \mathcal{T}_1. 
    \end{align*}
    For simplicity, let us assume that $|B| \gg 1$. The case $|B| = O(1)$ can be handled similarly by setting $\delta = 100/n$.
    
    Since $p > 1$, it follows from~\eqref{equ:FACT:inequality-a-4} that for every $T \in \mathcal{T}_1$, 
    \begin{align*}
        d_{\mathcal{H}}^{p}(T) - \left(d_{\mathcal{H}}(T) - d_{\mathcal{H}}(T\cup \{v\})\right)^{p}
        & \le p \cdot d_{\mathcal{H}}^{p-1}(T) \cdot d_{\mathcal{H}}(T\cup \{v\}) \\
        & \le p (2 \delta n^{r-t})^{p-1} n^{r-t-1}
        \le  2 p \delta^{p-1} n^{p(r-t)-1}. 
    \end{align*}
    Consequently, 
    \begin{align*}
        \sum_{T \in \mathcal{T}_1} \left( d_{\mathcal{H}}^{p}(T) - \left(d_{\mathcal{H}}(T) - d_{\mathcal{H}}(T\cup \{v\})\right)^{p} \right)
         \le |\mathcal{T}_1| \cdot 2 p \delta^{p-1} n^{p(r-t)-1}
         \le 2 p \delta^{p-1} n^{t-1+p(r-t)}. 
    \end{align*}
    Therefore, it suffices to consider the upper bound for 
    \begin{align*}
        \Delta' 
        \coloneqq \sum_{T \in \mathcal{T}_2} \Big( \left( d_{\mathcal{H}}^{p}(T) - \left(d_{\mathcal{H}}(T) - d_{\mathcal{H}}(T\cup \{v\})\right)^{p} \right)
        - \left( d_{\mathcal{G}}^{p}(T) - \left(d_{\mathcal{G}}(T) - d_{\mathcal{G}}(T\cup \{v\})\right)^{p} \right) \Big). 
    \end{align*}
    
    \textbf{Case 2.1$\colon$} $t = r-1$.  
    
    By the Mean Value Theorem, for every $T \in \mathcal{T}_2 \subset \partial_{r-t-1}L_{\mathcal{G}}(v) = L_{\mathcal{G}}(v)$ there exists $z_T$ with $d_{\mathcal{H}}(T) \ge z_{T} \ge d_{\mathcal{G}}(T) \ge d_{\mathcal{H}}(T) - |B| n^{r-t-1} \ge \delta n^{r-t}$ such that 
    \begin{align*}
        \Delta'
        & = \sum_{T \in \mathcal{T}_2} \Big( \left( d_{\mathcal{H}}^{p}(T) - \left(d_{\mathcal{H}}(T) - 1\right)^{p} \right) - \left( d_{\mathcal{G}}^{p}(T) - \left(d_{\mathcal{G}}(T) - 1\right)^{p} \right) \Big)\\
        & \le \sum_{T \in \mathcal{T}_2}  p  \left(z_T^{p-1} - \left(z_T - 1\right)^{p-1}\right)  \left(d_{\mathcal{H}}(T) - d_{\mathcal{G}}(T)\right) 
        \le |B| \sum_{T \in \mathcal{T}_2}  p  \left(z_T^{p-1} - \left(z_T - 1\right)^{p-1}\right). 
    \end{align*}
    Since $p-1 \in (0,1)$, the function $X^{p-1} - (X-1)^{p-1}$ is decreasing in $X$. So 
    \begin{align*}
        z_T^{p-1} - \left(z_T - 1\right)^{p-1} 
        \le (\delta n^{r-t})^{p-1} - (\delta n^{r-t} - 1)^{p-1}
        \le 2 (\delta n^{r-t})^{p-2}
        = 2 \delta^{p-2} n^{p-2}. 
    \end{align*}
    Therefore, the inequality on $\Delta'$ continues as 
    \begin{align*}
        \Delta'
         \le |\mathcal{T}_2| |B| \cdot 2 \delta^{p-2} n^{p-2} 
         \le n^{t} \cdot \delta n \cdot 2 \delta^{p-2} n^{p-2} 
         = 2 \delta^{p-1} n^{t-1+p}.
    \end{align*}

    \medskip 

    \textbf{Case 2.2$\colon$} $t \le r-2$.  

    Similarly, by the Mean Value Theorem, for every $T \in \mathcal{T}_2$, there exists $(z_{1,T}, z_{2,T})$ with $\delta n^{r-t} \le d_{\mathcal{G}}(T) \le z_{1,T} \le d_{\mathcal{H}}(T) \le n^{r-t}$
     and $d_{\mathcal{G}}(T\cup \{v\}) \le z_{2,T} \le d_{\mathcal{H}}(T\cup \{v\}) \le n^{r-t-1}$ such that 
    \begin{align*}
        \Delta' 
        & = \sum_{T \in \mathcal{T}_2} \Big( \left( d_{\mathcal{H}}^{p}(T) - \left(d_{\mathcal{H}}(T) - d_{\mathcal{H}}(T\cup \{v\})\right)^{p} \right)
        - \left( d_{\mathcal{G}}^{p}(T) - \left(d_{\mathcal{G}}(T) - d_{\mathcal{G}}(T\cup \{v\})\right)^{p} \right) \Big) \\
        & \le \sum_{T \in \mathcal{T}_2} 
        \Big( D_1 f(z_{1,T}, z_{2,T}) \cdot \left(d_{\mathcal{H}}(T)- d_{\mathcal{G}}(T)\right)
        + D_2 f(z_{1,T}, z_{2,T}) \cdot \left(d_{\mathcal{H}}(T\cup \{v\})- d_{\mathcal{G}}(T\cup \{v\})\right) \Big)\\
        & \le \sum_{T \in \mathcal{T}_2} 
        \Big( D_1 f(z_{1,T}, z_{2,T}) \cdot |B|n^{r-t-1}
        + D_2 f(z_{1,T}, z_{2,T}) \cdot |B|n^{r-t-2} \Big). 
    \end{align*}
    Since $p-1 \in (0,1)$, the function $X_{1}^{p-1} - (X_1-X_2)^{p-1}$ is decreasing in $X_1$. 
    Hence, for every $T \in \mathcal{T}_2$, 
    \begin{align*}
        D_1 f(z_{1,T}, z_{2,T})
         = p \left(z_{1,T}^{p-1} - (z_{1,T}-z_{2,T})^{p-1}\right) 
        & \le p \left((\delta n^{r-t})^{p-1} - (\delta n^{r-t}-n^{r-t-1})^{p-1}\right) \\
        & \le 2p (\delta n^{r-t})^{p-2} n^{r-t-1} 
        = 2p \delta^{p-2} n^{(p-1)(r-t)-1}. 
    \end{align*}
    In addition, 
    \begin{align*}
        D_2 f(z_{1,T}, z_{2,T})
        = p\left(z_{1,T}-z_{2,T}\right)^{p-1} 
        \le p n^{(p-1)(r-t)}. 
    \end{align*}
    Therefore, the inequality on $\Delta'$ continues as 
    \begin{align*}
        \Delta' 
        & \le |\mathcal{T}_2||B|\left(2p \delta^{p-2} n^{(p-1)(r-t)-1} \cdot n^{r-t-1} + p n^{(p-1)(r-t)} \cdot n^{r-t-2}\right) \\
        & \le n^t \cdot \delta n  \left(2p \delta^{p-2} n^{p(r-t)-2}  + p n^{p(r-t)-2}\right) 
        % = \left(2p \delta^{p-1} n + p \delta n \right) n^{t-2+p(r-t)}
        \le 3p \delta^{p-1} n^{t-1+p(r-t)}. 
    \end{align*}
    This completes the proof of Lemma~\ref{LEMMA:Lp-basic-property-local-Lipschitz}. 
\end{proof}

%%%%%%%%%%%%
With all the necessary general properties established for $\Gamma$ (as defined at the beginning of this section), we are now ready to state the consequence of the general theorems from~\cite{CL24} in the context of $(t,p)$-norm Tur\'{a}n problems. 

Let $r > t \ge 1$ be integers and $p > 0$ be a real number. 
Let $\mathcal{F}$ be a family of $r$-graphs and $\mathfrak{H}$ be a hereditary family of $\mathcal{F}$-free $r$-graphs. 
\begin{enumerate}[label=(\roman*)]
    % \item We say $\mathcal{F}$ is \textbf{blowup-invariant} if every blowup of every $\mathcal{F}$-free $r$-graph remains $\mathcal{F}$-free. 
    \item We say $\mathcal{F}$ is \textbf{symmetrized-stable} with respect to $\mathfrak{H}$ if every symmetrized $\mathcal{F}$-free $r$-graph is contained in $\mathfrak{H}$. 
    \item We say $\mathcal{F}$ is \textbf{$(t,p)$-edge-stable} with respect to $\mathfrak{H}$ if for every $\delta > 0$ there exist $\varepsilon > 0$ and $N_0$ such that every $\mathcal{F}$-free $r$-graph $\mathcal{H}$ on $n \ge N_0$ vertices with $\norm{\mathcal{H}}_{t,p} \ge (1-\varepsilon) \cdot \mathrm{ex}_{t,p}(n,\mathcal{F})$ is contained in $\mathfrak{H}$ after removing at most $\delta n^r$ edges. 
    \item We say $\mathcal{F}$ is \textbf{$(t,p)$-degree-stable} with respect to $\mathfrak{H}$ if there exist $\varepsilon>0$ and $N_0$ such that every $\mathcal{F}$-free $r$-graph $\mathcal{H}$ on $n \ge N_0$ vertices with $\delta_{t,p}(\mathcal{H}) \ge (1-\varepsilon) \cdot \mathrm{exdeg}_{t,p}(n,\mathcal{F})$ is contained in $\mathfrak{H}$. 
    \item We say $\mathcal{F}$ is \textbf{$(t,p)$-vertex-extendable} with respect to $\mathfrak{H}$ if there exist $\varepsilon>0$ and $N_0$ such that the following holds for every $\mathcal{F}$-free $r$-graph $\mathcal{H}$ on $n \ge N_0$ vertices with $\delta_{t,p}(\mathcal{H}) \ge (1-\varepsilon) \cdot \mathrm{exdeg}_{t,p}(n,\mathcal{F}) \colon$ 
    if $\mathcal{H}-v \in \mathfrak{H}$ for some $v \in V(\mathcal{H})$, then $\mathcal{H} \in \mathfrak{H}$. 
\end{enumerate}

The following theorem, which extends~{\cite[Theorem~1.7]{LMR23unif}}, follows as a consequence of~{\cite[Theorem~1.8]{CL24}} (see also~{\cite[Theorem~4.9]{CL24}}).

\begin{theorem}\label{THM:Lp-general-a}
    Let $r > t \ge 1$ be integers and $p \ge 1$ be a real number. Let $\mathcal{F}, \hat{\mathcal{F}}$ be two nondegenerate families of $r$-graphs such that $\hat{\mathcal{F}} \le_{\mathrm{hom}} \mathcal{F}$.  
    Let $\mathfrak{H}$ be a hereditary family of  $\hat{\mathcal{F}}$-free $r$-graphs. 
    Suppose that 
    \begin{enumerate}[label=(\roman*)]
        \item\label{THM:Lp-general-a-1} $\hat{\mathcal{F}}$ is blowup-invariant, and 
        \item\label{THM:Lp-general-a-2} $\hat{\mathcal{F}}$ is symmetrized-stable with respect to $\mathfrak{H}$.
    \end{enumerate}
    Then for every integer $n \ge 1$, 
    \begin{align}\label{equ:THM:Lp-general-a}
        \mathrm{ex}_{t,p}(n, \hat{\mathcal{F}}) 
        = \max\left\{\norm{\mathcal{G}}_{t,p} \colon \text{$\mathcal{G} \in \mathfrak{H}$ and $v(\mathcal{G}) = n$}\right\}. 
    \end{align}
    If, in addition,
    \begin{enumerate}[label=(\roman*)]
        \setcounter{enumi}{2}
        \item $\mathfrak{H}$ is $\mathcal{F}$-free, and
        \item both $\hat{\mathcal{F}}$ and $\mathcal{F}$ are $(t,p)$-vertex-extendable with respect to $\mathfrak{H}$.
    \end{enumerate}
    Then $\mathcal{F}$ is $(t,p)$-degree-stable with respect to $\mathfrak{H}$. 
\end{theorem}
\textbf{Remark.}
It should be noted that~\eqref{equ:THM:Lp-general-a} is trivial (see e.g.~{\cite[Fact~1.5]{CL24}}). 
The nontrivial part, namely the stability part, of Theorem~\ref{THM:Lp-general-a} follows by applying~{\cite[Theorem~1.8]{CL24}} to maps $\Gamma, \hat{\Gamma} \coloneqq \mathfrak{G}^{r} \to \mathbb{R}$ defined by 
\begin{align*}
    \Gamma(\mathcal{H})
    \coloneqq \norm{\mathcal{H}}_{t,p}\cdot \mathbbm{1}_{\mathcal{F}}(\mathcal{H})
    \quad\text{and}\quad 
    \hat{\Gamma}(\mathcal{H})
    \coloneqq \norm{\mathcal{H}}_{t,p}\cdot \mathbbm{1}_{\hat{\mathcal{F}}}(\mathcal{H}) 
    \quad\text{for all}\quad 
    \mathcal{H} \in \mathfrak{G}^{r}. 
\end{align*}

The following theorem, which extends~{\cite[Theorem~1.1]{HLZ24}}, is a consequence of~{\cite[Theorem~1.7]{CL24}} (see also~{\cite[Theorem~4.10]{CL24}}). 

\begin{theorem}\label{THM:Lp-general-b}
    Let $r > t \ge 1$ be integers and $p \ge 1$ be a real number. Let $\mathcal{F}$ be a nondegenerate family of $r$-graphs and $\mathfrak{H}$ be a hereditary family of $\mathcal{F}$-free $r$-graphs. 
    Suppose that $\mathcal{F}$ is $(t,p)$-edge-stable and $(t,p)$-vertex-extendable with respect to $\mathfrak{H}$. 
    Then $\mathcal{F}$ is $(t,p)$-degree-stable with respect to $\mathfrak{H}$. 
\end{theorem}

The following result shows that to prove Theorems~\ref{THM:Lp-F5-p-large} and~\ref{THM:Lp-clique-expansion-p-large}, one only needs to focus on the degree-stability part. Its proof is essentially the same as that of~{\cite[Proposition~4.11]{CL24}}, so we omit it here.
Recall that $\mathfrak{K}_{\ell}^{r}$ is the collection of all $K_{\ell}^{r}$-colorable $r$-graphs. 

\begin{proposition}\label{PROP:Lp-min-degree-extremal}
    Let $r > t \ge 1$ be integers and $p > 0$ be a real number. Let $\mathcal{F}$ be a nondegenerate family of $r$-graphs that is $(t,p)$-degree-stable with respect to $\mathfrak{K}_{\ell}^{r}$ for some $\ell \ge r$. 
    Then for large $n$, 
    \begin{align*}
        \mathrm{ex}_{t,p}(n,\mathcal{F})
        = \max\left\{\norm{\mathcal{G}}_{t,p} \colon \mathcal{G} \in \mathfrak{K}_{\ell}^{r} \text{ and } v(\mathcal{G}) = n\right\}. 
    \end{align*}
\end{proposition}

%%%%%%%%%%%%%%%%%%%%%%%%%%%%%%%%%%%%%%%
\section{Proof of Theorem~\ref{THM:Lp-F5-p-large}}\label{SEC:Proof-Lp-F5-large}
%%%%%%%%%%%%%%%%%%%%%%%%%%%%%%%%%%%%%%%
In this section, we prove Theorem~\ref{THM:Lp-F5-p-large}. Our proof is to apply Theorem~\ref{THM:Lp-general-a} with $\mathcal{F} = \{F_5\}$, $\hat{\mathcal{F}} = \mathcal{T}_{3} \coloneqq \{K_{4}^{3-}, F_5\}$, and $\mathfrak{H} = \mathfrak{S}$, where 
\begin{align*}
    \mathfrak{S}
    \coloneqq \left\{\mathcal{G} \in \mathfrak{G}^{3} \colon \text{$\mathcal{G}$ is $\mathcal{S}$-colorable for some STS $\mathcal{S}$}\right\}. 
\end{align*}
Note that Conditions~\ref{THM:Lp-general-a-1} and~\ref{THM:Lp-general-a-2} in Theorem~\ref{THM:Lp-general-a} are straightforward to verify in this case (see ~{\cite[Lemma~4.2]{LMR23unif}} and the comment before~{\cite[Lemma~4.3]{LMR23unif}}). 
Therefore, we only need to show that $F_5$ is $(2,p)$-vertex-extendable with respect to $\mathfrak{S}$ for every $p \ge 1$. 

Since $\mathfrak{S}$ is a rather large family to handle, instead of addressing all STSs, we will show in the next section that it is sufficient to prove that $F_5$ is $(2,p)$-vertex-extendable with respect to a very simple subfamily of $\mathfrak{S}$, namely, $\mathfrak{K}_{3}^{3}$, the collection of all $3$-partite $3$-graphs.
To achieve this reduction, we will use results of Brown--Sidorenko~{\cite[Proposition~2]{BS94}} and Liu--Mubayi--Reiher~{\cite[Lemma~5.1]{LMR23induced}} on graph inducibility problems along with some very technical and nontrivial calculations, to show that if a member in $\mathfrak{S}$ has near-extremal minimum $(2,p)$-degree, then it must be $3$-partite (Proposition~\ref{PORP:min-deg-STS-color}). 

%%%%%%%%%%%%%%%%%%%%%%%%%%%%%%%%%%%%%%%%%%%%%
\subsection{Excluding nontrivial Steiner triple systems}\label{SUBSEC:F5-large-STS}
For every real number $p \ge 0$ define 
    \begin{align*}
        g_{p}(x_1, x_2) 
         \coloneqq x_{1}^{p}x_2+ x_{1}x_{2}^{p},  \quad\text{and}\quad
        h_{p}(x_1, x_2, x_3) 
         \coloneqq x_1x_2x_3^{p}+x_1x_2^{p}x_3+x_1^{p}x_2x_3. 
    \end{align*}
For convenience, define 
\begin{align*}
    g_{p}^{\ast} 
     \coloneqq \max_{(x_1, x_2) \in \Delta^{1}} g_{p}(x_1, x_2), \quad\text{and}\quad
    h_{p}^{\ast} 
     \coloneqq \max_{(x_1, x_2, x_3) \in \Delta^{2}} h_{p}(x_1, x_2, x_3). 
\end{align*}
Notice that $h_{p}(x_1, x_2, x_3) = L_{K_{3}^{3},2,p}(x_1, x_2, x_3)$ and $h_{p}^{\ast} = \lambda_{2,p}(K_{3}^{3})$.
The main task of this subsection is to establish the following result, which extends~{\cite[Lemma~4.3]{LMR23unif}}. 
\begin{proposition}\label{PORP:min-deg-STS-color}
    For every real number $p \ge 1$, there exists $N_0$ such that the following holds for all $n \ge N_0$. 
    Suppose that $\mathcal{H} \in \mathfrak{S}$ is an $n$-vertex $3$-graph satisfying  
    %$\delta_{2,p}(\mathcal{H}) \ge (1-10^{-3})(p+2)\cdot h^{\ast}_{p} \cdot n^{1+p}$,
    \begin{align*}
        \delta_{2,p}(\mathcal{H}) \ge (1-10^{-3})(p+2)\cdot h^{\ast}_{p} \cdot n^{1+p}. 
    \end{align*}
    Then $\mathcal{H}$ is $3$-partite.    
\end{proposition}

The key inequality for proving Proposition~\ref{PORP:min-deg-STS-color} is as follows. 
To avoid being distracted too much from the proof of the main result (Theorem~\ref{THM:Lp-F5-p-large}), we postpone the proof for the key inequality to Section~\ref{SUBSEC:proof-F5-inequalities}. 

\begin{lemma}\label{LEMMA:T3-induction-number-of-parts}
    For every real number $p \ge 2$,  
    \begin{align*}
        \frac{p \cdot h_{p-1}^{\ast}  +  g_{p}^{\ast} }{(p+2)\cdot h_{p}^{\ast} } < 6.88. 
    \end{align*}
\end{lemma}

A crucial step in establishing Proposition~\ref{PORP:min-deg-STS-color} is the following lemma concerning the $(2,p)$-Lagrangian of STSs. 

\begin{lemma}\label{LEMMA:tp-Lagrangian-STS}
    Suppose that $m \in 6\mathbb{N}+\{1,3\}$,  $\mathcal{S}$ is a STS on $[m]$, and $p \ge 1$ is a real number. 
    If there exists a vector $\vec{x} \coloneqq (x_1, \ldots, x_m) \in \Delta^{m-1}$ satisfying 
    \begin{align*}
        \min\{x_i\colon i\in [m]\} > 0
        \quad\text{and}\quad 
        \min\left\{D_{i} L_{\mathcal{S}, 2, p}(\vec{x}) \colon i\in [m]\right\}\ge (1-10^{-3}) (p+2) \cdot h_{p}^{\ast}, 
    \end{align*}
    then $m=3$. 
    In particular, 
    \begin{align}\label{equ:LEMMA:tp-Lagrangian-STS-1}
        \lambda_{2, p}(\mathcal{S})
        = \lambda_{2, p}(K_{3}^{3})
        = h_{p}^{\ast}. 
    \end{align}
\end{lemma}

The following inequality will be useful for proving Lemma~\ref{LEMMA:tp-Lagrangian-STS}. 
The integral case corresponds to an old result of Brown--Sidorenko~{\cite[Proposition~2]{BS94}} concerning the inducibility problem of stars in a graph. 
The general case can be derived with a slight modification of their argument. For completeness, we include its proof in Section~\ref{SUBSEC:proof-F5-inequalities}. 

\begin{lemma}\label{LEMMA:inducibility-star}%[{\cite[Proposition~2]{BS94}}]
    Let $p \ge 2$ be a real number. 
    For every integer $n \ge 2$ and for every $(x_1, \ldots, x_{n}) \in \Delta^{n-1}$,  
    \begin{align*}
        L_{K_n, 1, p}(x_1, \ldots, x_n)
        = \sum_{1\le i < j \le n} \left(x_i x_j^p + x_i^p x_j \right) 
        \le g_{p}^{\ast}. 
    \end{align*}
\end{lemma}

\begin{proof}[Proof of Lemma~\ref{LEMMA:tp-Lagrangian-STS}]
    We prove this lemma by induction on $p$. 
    Notice that the "In particular" part follows easily from  
    Proposition~\ref{PROP:star-polynomial-sym-increase} and the fact that every $2$-covered subgraph of an STS is also an STS. 
    So it suffices to prove the first statement in Lemma~\ref{LEMMA:tp-Lagrangian-STS}.
    Additionally, since the number of vertices in every STS is in $6\mathbb{N}+\{1,3\}$, we just need to show that $m < 7$. 

    \textbf{Base case:} $p\in [1,2]$. 

    The case $p=1$ corresponds to~{\cite[Lemma~4.3]{LMR23unif}}, so we may assume that $p > 1$.  
    It follows from the assumption $\min\left\{D_{i} L_{\mathcal{S}, 2, p}(\vec{x}) \colon i\in [m]\right\}\ge (1-10^{-3}) (p+2) \cdot h_{p}^{\ast}$ that 
    \begin{align*}
        m\cdot (1-10^{-3})(p+2) \cdot h^{\ast}_{p}
        & \le \sum_{i\in [m]} D_{i}L_{\mathcal{S}, 2, p}(x_1, \ldots, x_{m}) \\
        & = \sum_{i\in [m]}\sum_{\{j,k\} \in L_{\mathcal{S}}(i)} \left(x_jx_k^p + x_j^px_k+px_i^{p-1}x_jx_k\right) \\
        & = p \cdot \sum_{i\in [m]}\sum_{\{j,k\} \in L_{\mathcal{S}}(i)}  x_i^{p-1}x_jx_k 
        + \sum_{\{j,k\} \in \binom{[m]}{2}} \left(x_jx_k^p + x_j^px_k\right) \\
        & = p \cdot L_{\mathcal{S},2,p-1}(\vec{x}) + L_{K_{m},1,p}(\vec{x}),   
    \end{align*}
    where, in the second to last equality, we used the property of STS that every pair of vertices in $\mathcal{S}$ is contained in exactly one edge. 
    The inequality above can be rewritten as 
    \begin{align}\label{equ:STS-a}
        m
         \le \frac{1}{1-10^{-3}} \cdot \frac{p \cdot L_{\mathcal{S},2,p-1}(\vec{x}) + L_{K_{m},1,p}(\vec{x})}{(p+2) \cdot h^{\ast}_{p}},
    \end{align}
    It follows from Proposition~\ref{PROP:star-poly-Holder-inequ} (with $(p_1,p_2) = (1,2)$) that 
    \begin{align*}
        L_{K_{m},1,p}(\vec{x}) 
        \le \left(L_{K_{m},1,1}(\vec{x})\right)^{2-p} \left(L_{K_{m},1,2}(\vec{x})\right)^{p-1}
        \le 1^{2-p} \left(\frac{1}{4}\right)^{p-1}
        = \frac{1}{4^{p-1}}. 
    \end{align*}
    Here, we used the inequalities 
    \begin{align*}
        L_{K_{m},1,1}(\vec{x})
        & = \sum_{\{i,j\}\in \binom{[m]}{2}}\left(x_ix_j + x_ix_j\right)
        = 2 \sum_{\{i,j\}\in \binom{[m]}{2}} x_ix_j 
        \le 1, \quad\text{and}\quad \\
        L_{K_{m},1,2}(\vec{x})
        & = \sum_{\{i,j\}\in \binom{[m]}{2}}\left(x_ix_j^{2} + x_i^{2}x_j\right)
        \le g_{2}^{\ast}
        = \frac{1}{4}, 
    \end{align*}
    where the second inequality follows from Lemma~\ref{LEMMA:inducibility-star} and some simple calculations. 

    In addition, it follows from Proposition~\ref{PROP:star-poly-Holder-inequ} (with $(p_1,p_2) = (0,1)$) that 
    \begin{align*}
        L_{\mathcal{S},2,p-1}(x_1, \ldots, x_{m})
        & \le \left(L_{\mathcal{S},2,0}(x_1, \ldots, x_{m})\right)^{1-(p-1)} \left(L_{\mathcal{S},2,1}(x_1, \ldots, x_{m})\right)^{p-1-0}  \\
        & = \left(\sum_{\{i,j,k\} \in \mathcal{S}}\left(x_ix_j +x_jx_k+x_kx_i\right) \right)^{2-p} \left(\sum_{\{i,j,k\} \in \mathcal{S}} 3 x_ix_jx_k\right)^{p-1}  \\
        & =  \left(\sum_{\{i,j\}\in \binom{[m]}{2}}x_ix_j\right)^{2-p} \left(3\sum_{\{i,j,k\} \in \mathcal{S}}x_ix_jx_k\right)^{p-1} 
         \le \left(\frac{1}{2}\right)^{2-p}\left(\frac{1}{9}\right)^{p-1},  
    \end{align*}
    where $3\sum_{\{i,j,k\} \in \mathcal{S}}x_ix_jx_k \le 1/9$ follows from~\eqref{equ:LEMMA:tp-Lagrangian-STS-1} with $p=1$. 
    % \begin{align*}
    %     \sum_{\{i,j,k\} \in \mathcal{S}}x_ix_jx_k \le (1/3)^3 = 1/27. 
    % \end{align*}
    Therefore, Inequality~\eqref{equ:STS-a} continues as  
    \begin{align*}
        m
         \le \frac{1}{1-10^{-3}} \cdot  \frac{p \cdot L_{\mathcal{S},2,p-1}(\vec{x}) + L_{K_{m},1,p}(\vec{x})}{(p+2) \cdot h^{\ast}_{p}} 
         \le \frac{1}{1-10^{-3}} \cdot  \frac{p\cdot \left(\frac{1}{2}\right)^{2-p}\left(\frac{1}{9}\right)^{p-1} + \frac{1}{4^{p-1}}}{(p+2)\cdot\left(\frac{1}{3}\right)^{p+1}}
        < 5,
    \end{align*}
    where the last inequality is verified using Mathematica. 

    \medskip 

    \textbf{Inductive step:} Now suppose that $p > 2$. 

    It follows from~\eqref{equ:STS-a}, Lemma~\ref{LEMMA:inducibility-star}, and the inductive hypothesis that 
    \begin{align*}
        m
        & \le \frac{1}{1-10^{-3}} \cdot \frac{p \cdot L_{\mathcal{S},2,p-1}(\vec{x}) + L_{K_{m},1,p}(\vec{x})}{(p+2) \cdot h^{\ast}_{p}} \\
        & \le \frac{1}{1-10^{-3}} \cdot \frac{p \cdot \lambda_{2,p-1}(K_{3}^{3}) + g_{p}^{\ast}}{(p+2) \cdot h^{\ast}_{p}} 
         \le \frac{1}{1-10^{-3}} \cdot \frac{p \cdot h_{p-1}^{\ast} + g_{p}^{\ast}}{(p+2) \cdot h^{\ast}_{p}}
        % \le \frac{1}{1-10^{-3}} \cdot \frac{p \cdot h_{p-1}^{\ast}  +  g_{p}^{\ast} }{(p+2)\cdot h_{p}\left(\frac{1}{p+2}, \frac{1}{p+2}, \frac{p}{p+2}\right)}
        < 7,
    \end{align*}
    where the last inequality follows from Lemma~\ref{LEMMA:T3-induction-number-of-parts}. 
\end{proof}

Now we are ready to prove Proposition~\ref{PORP:min-deg-STS-color}. 
\begin{proof}[Proof of Proposition~\ref{PORP:min-deg-STS-color}]
    Suppose to the contrary that $\mathcal{H}$ is not $3$-partite. Then it means that there is a surjective homomorphism from $\mathcal{H}$ to some STS $\mathcal{S}$ on $m \ge 7$ vertices. 
    For simplicity, let us assume that $V(\mathcal{S}) = [m]$. 
    Fix a surjective homomorphism $\psi$ from $\mathcal{H}$ to $\mathcal{S}$. 
    Let $V_i \coloneqq \psi^{-1}(i)$ and $x_i \coloneqq |V_i|/n$ for $i\in [m]$. Since $\psi$ is surjective, we have $x_i > 0$ for every $i\in [m]$. 
    In addition, by Corollary~\ref{CORO:Lp-degree-expression-Lagrangian}, for every $i\in [m]$ and for every $v\in V_i$, 
    \begin{align*}
        (1-10^{-3})(p+2) \cdot h^{\ast}_{p} \cdot n^{1+p}
        \le d_{\mathcal{H}}(v)
        \le D_{i}L_{\mathcal{S}, 2, p}(x_1, \ldots, x_{m}) \cdot n^{1+p}. 
    \end{align*}
    Then it follows from Lemma~\ref{LEMMA:tp-Lagrangian-STS} that $m = 3$, a contradiction. 
\end{proof}

%%%%%%%%%%%%%%%%%%%%%%%%%%%%%%%%%%%%%%%
\subsection{Vertex-extendability}
In this subsection, we show that $F_5$ is $(2,p)$-vertex-extendable with respect to $\mathfrak{S}$, which, as discussed in the previous subsection, would complete the proof of Theorem~\ref{THM:Lp-F5-p-large}. 

% Recall that, by Proposition~\ref{PROP:Lp-min-degree-extremal} and the definition of $(2,p)$-vertex-extendability, is simplified to showing that $F_5$ is $(2,p)$-vertex-extendable with respect to $\mathfrak{K}_{3}^{3}$. 

\begin{proposition}\label{PROP:vertex-extendability-F5}
    %Let $p \ge 1$ be a real number. 
    The $3$-graph $F_5$ is $(2,p)$-vertex-extendable with respect to $\mathfrak{S}$ for every real number $p \ge 1$. 
\end{proposition}

Before proving Proposition~\ref{PROP:vertex-extendability-F5}, let us present some useful results. 

\begin{proposition}\label{PROP:2-p-density-F5-p-large}
    For every real $p > 1$, $\pi_{2,p}(F_5) = 2 \cdot \lambda_{2,p}(K_3^{3}) = 2 h_{p}^{\ast}$. 
\end{proposition}
\begin{proof}[Proof of Proposition~\ref{PROP:2-p-density-F5-p-large}]
    % The lower bound $\pi_{2,p}(F_5) \ge 2 \cdot \lambda_{2,p}(K_3^{3})$ is clear, so it suffices to show that $\pi_{2,p}(F_5) \le 2 \cdot \lambda_{2,p}(K_3^{3})$. 
    By Proposition~\ref{PROP:tp-density-blowup-lemma}, it suffices to show that $\pi_{2,p}(\mathcal{T}_3) = 2 \cdot \lambda_{2,p}(K_3^{3})$.  
    %where $\mathcal{T}_3 \coloneqq \{K_{4}^{3}, F_5\}$. 
    Since $\mathcal{T}_3$ is blowup-invariant and symmetrized-stable with respect to $\mathfrak{S}$, 
    it follows from~\eqref{equ:THM:Lp-general-a} that 
    \begin{align*}
        \mathrm{ex}_{2,p}(n,\mathcal{T}_3)
        = \max\left\{\norm{\mathcal{G}}_{2,p} \colon \text{$\mathcal{G} \in \mathfrak{S}$ and $v(\mathcal{G}) = n$} \right\}. 
    \end{align*}
    Combining with Fact~\ref{FACT:t-p-Lagrangian}, we obtain  
    \begin{align*}
        \pi_{2,p}(\mathcal{T}_3)
        = 2 \cdot \max\left\{\lambda_{2,p}(\mathcal{S}) \colon \text{$\mathcal{S}$ is an STS}\right\}
        = 2 \cdot \lambda_{2,p}(K_{3}^{3}),  
    \end{align*}
    where the last equality follows from Lemma~\ref{LEMMA:tp-Lagrangian-STS}. 
\end{proof}

A key ingredient in proving Proposition~\ref{PROP:vertex-extendability-F5} is the following extension of~{\cite[Lemma~4.4]{LMR23unif}}. 

\begin{lemma}\label{LEMMA:vertex-extendability-F5-3-partite}
    For every $\varepsilon>0$ there exist $\delta>0$ and $N_{0}>0$ such that the following holds for all $n \ge N_{0}$. 
    Let $V_1 \cup V_2 \cup V_{3} = [n]$ be a partition with $|V_i| \ge \varepsilon n$ for $i\in [3]$. 
    Suppose that $\mathcal{H}$ is an $F_5$-free $3$-graph on $[n]\cup \{v_{\ast}\}$ such that 
    \begin{enumerate}[label=(\roman*)]
        \item\label{LEMMA:vertex-extendability-F5-3-partite-1} $\mathcal{H}' \coloneqq \mathcal{H}-v_{\ast}$ is a subgraph of $K^{3} \coloneqq K^{3}[V_1, V_2, V_{3}]$,
        \item\label{LEMMA:vertex-extendability-F5-3-partite-2} $d_{\mathcal{H}'}(u) \ge d_{K^3}(u) - \delta n^{2}$ for all $u \in V(\mathcal{H}')$, and 
        \item\label{LEMMA:vertex-extendability-F5-3-partite-3} $d_{\mathcal{H}}(v_{\ast}) \ge \varepsilon n^{2}$. 
    \end{enumerate}
    Then $\mathcal{H}$ is $3$-partite. %In other words, $F_5$ is $(2,p)$-vertex-extendable with respect to $\mathfrak{K}_{3}^{3}$. 
\end{lemma}%
\begin{proof}[Proof of Lemma~\ref{LEMMA:vertex-extendability-F5-3-partite}]
    Fix $\varepsilon>0$. 
    Let $\delta>0$ be sufficiently small and $n$ sufficiently large. 
    Let $\mathcal{H}$ and $v_{\ast}$ be as assumed in Lemma~\ref{LEMMA:vertex-extendability-F5-3-partite}. 

    \begin{claim}\label{CLAIM:vertex-extendability-F5-a}
        For every $e\in L_{\mathcal{H}}(v_{\ast})$ and for every $i\in [3]$, we have $|e\cap V_i| \le 1$. 
    \end{claim}
    \begin{proof}[Proof of Claim~\ref{CLAIM:vertex-extendability-F5-a}]
        Suppose to the contrary that this is not true. 
        By symmetry, we may assume that there exist vertices $u_1, u_2\in V_1$ such that $\{v_{\ast},u_1, u_2\} \in \mathcal{H}$. 
        It follows from the Inclusion-Exclusion Principle, Assumptions~\ref{LEMMA:vertex-extendability-F5-3-partite-1}, and~\ref{LEMMA:vertex-extendability-F5-3-partite-2} that 
        \begin{align*}
            |L_{\mathcal{H}'}(u_1) \cap L_{\mathcal{H}'}(u_2)|
            \ge |L_{K^{3}}(u_1)| - 2\delta n^2
            \ge \varepsilon^2 n^2 - 2 \delta n^2 
            >0. 
        \end{align*}
        Choose an arbitrary $\{w_1, w_2\} \in L_{\mathcal{H}'}(u_1) \cap L_{\mathcal{H}'}(u_2)$. 
        Notice that $\{v_{\ast}u_1u_2, u_1w_1w_2, u_2w_1w_2\}$ is a copy of $F_5$ in $\mathcal{H}$, a contradiction. 
    \end{proof}
    It follows from Claim~\ref{CLAIM:vertex-extendability-F5-a} that $L_{\mathcal{H}}(v_{\ast})$ is a $3$-partite graph with parts $V_1, V_2, V_3$. 
    By the Pigeonhole Principle, we may assume that at least $d_{\mathcal{H}}(v_{\ast})/3 \ge \varepsilon n^2/3$ edges of $L_{\mathcal{H}}(v_{\ast})$ are crossing $V_2$ and $V_3$. 
    Suppose to the contrary that there exists $\{u_1, u_2\} \in L_{\mathcal{H}}(v_{\ast})$ with $u_1 \in V_1$. 
    Then similar to the proof of Claim~\ref{CLAIM:vertex-extendability-F5-a}, we have 
    \begin{align*}
            |L_{\mathcal{H}'}(u_1) \cap L_{\mathcal{H}}(v_{\ast})|
            \ge \varepsilon n^2/3 - \delta n^2
            > 2 n. 
    \end{align*}
    Therefore, there exists $\{w_1, w_2\} \in L_{\mathcal{H}'}(u_1) \cap L_{\mathcal{H}}(v_{\ast})$ that is disjoint from $\{u_1, u_2\}$. 
    However, $\{v_{\ast}u_1u_2, v_{\ast}w_1w_2, u_1w_1w_2\}$ is a copy of $F_5$ in $\mathcal{H}$, a contradiction.
    Therefore, $L_{\mathcal{H}}(v_{\ast})$ is a bipartite graph with parts $V_2$ and $V_3$, meaning that $\mathcal{H}$ is $3$-partite. 
\end{proof}

Now we are ready to prove Proposition~\ref{PROP:vertex-extendability-F5}. 

\begin{proof}[Proof of Proposition~\ref{PROP:vertex-extendability-F5}]
    Fix $0< \varepsilon \ll \varepsilon_1 \ll \varepsilon_2 \ll \varepsilon_3 \ll (1/3)^{p}$ to be sufficiently small and let $n$ be sufficiently large. 
    Let $\mathcal{H}$ be an $(n+1)$-vertex $F_5$-free $3$-graph with 
    \begin{align*}
        \delta_{2,p}(\mathcal{H}) 
        \ge \left(1-\frac{\varepsilon}{2}\right)\frac{(2+p) \cdot \pi_{2,p}(F_5) \cdot n^{1+p}}{2} 
        =  \left(1-\frac{\varepsilon}{2}\right) (2+p) h_{p}^{\ast} \cdot n^{1+p}. 
    \end{align*}
    Suppose that $v_{\ast} \in V(\mathcal{H})$ is a vertex satisfying $\mathcal{G} \coloneqq \mathcal{H}-v_{\ast} \in \mathfrak{S}$. 
    By Lemma~\ref{LEMMA:Lp-basic-property-local-Lipschitz} (with $B = \{v_{\ast}\}$),  
    \begin{align*}
        \delta_{2,p}(\mathcal{G})
        \ge \delta_{2,p}(\mathcal{H}) - o(n^{1+p})
        \ge \left(1-\varepsilon\right) (2+p) h_{p}^{\ast} \cdot n^{1+p}. 
    \end{align*}
    By Proposition~\ref{PORP:min-deg-STS-color}, this implies that $\mathcal{G}$ is $3$-partite. 
    In addition, by Lemma~\ref{LEMMA:Lp-uniform}, 
    \begin{align}\label{equ:proof:PROP:vertex-extendability-F5-1}
        \norm{\mathcal{G}}_{2,p}
        \ge \frac{1}{2+p} \sum_{v\in V} d_{2,p,\mathcal{G}}(v) - o(n^{2+p})
        \ge (1-2\varepsilon) h_{p}^{\ast} \cdot n^{2+p}. 
    \end{align}
    % Combining with Fact~\ref{FACT:t-p-Lagrangian}, we obtain 
    % \begin{align}\label{equ:proof:PROP:vertex-extendability-F5-1.5}
    %     \norm{\mathcal{G}}_{2,p}
    %     = \frac{1}{2+p} \sum_{v\in V} d_{2,p,\mathcal{G}}(v) - o(n^{2+p})
    %     \ge (1-2\varepsilon) h_{p}^{\ast} \cdot n^{2+p}. 
    % \end{align}
    Let $V\coloneqq V(\mathcal{H}) \setminus \{v_{\ast}\}$ and $V_1 \cup V_2 \cup V_3 = V$ be a partition such that $\mathcal{G} \subset K^{3} \coloneqq K^{3}[V_1, V_2, V_3]$. 
    Let $x_i \coloneqq |V_i|/n$ for $i\in \{1,2,3\}$.
    By Fact~\ref{FACT:t-p-Lagrangian} and~\eqref{equ:proof:PROP:vertex-extendability-F5-1}, 
    \begin{align}\label{equ:proof:PROP:vertex-extendability-F5-1.5}
        h_{p}(x_1,x_2,x_3)
        = \frac{\norm{K^3}_{2,p}}{n^{2+p}}
        \ge \frac{\norm{\mathcal{G}}_{2,p}}{n^{2+p}}
        \ge (1-2\varepsilon) h_{p}^{\ast}. 
    \end{align}
    Let $x_{\ast} \coloneqq \min\{x_1,x_2,x_3\}$ and 
    %$M \coloneqq \max\left\{|D_{i,j} h_{p}(z_1,z_2,z_3)| \colon (z_1,z_2,z_3) \in \Delta^2,\ i\in \{1,2,3\}\right\}$. 
    \begin{align*}
        M 
        \coloneqq \max_{1\le i \le j \le 3}\left\{|D_{i,j} h_{p}(z_1,z_2,z_3)| \colon (z_1,z_2,z_3) \in \Delta^2,\ \min\{z_1,z_2,z_3\} \ge \frac{1}{4}\left(\frac{1}{3}\right)^{p+2}\right\}.
    \end{align*}
    Note that $M<\infty$ is a constant depending only on $p$. 
    First, it follows from $3 x_{\ast} 
        \ge h_{p}(x_{\ast},1,1)
        \ge h_{p}(x_1,x_2,x_3)
        \ge (1-2\varepsilon) h_{p}^{\ast}
        \ge \frac{1}{2} \cdot 3 \left(\frac{1}{3}\right)^{p+2}$
    % \begin{align*}
    %     3 x_{\ast} 
    %     \ge h_{p}(x_{\ast},1,1)
    %     \ge h_{p}(x_1,x_2,x_3)
    %     \ge (1-2\varepsilon) h_{p}^{\ast}
    %     \ge \frac{1}{2} \cdot 3 \left(\frac{1}{3}\right)^{p+2}
    % \end{align*}
    that  
    \begin{align}\label{equ:proof:PROP:vertex-extendability-F5-2}
        x_{\ast} 
        \ge  \frac{1}{2}\left(\frac{1}{3}\right)^{p+2}
    \end{align}
    Second, by~\eqref{equ:proof:PROP:vertex-extendability-F5-1.5} and compactness (see e.g. the proof of~{\cite[Lemma~6.2]{CL24}}), there exists $(y_1,y_2,y_3) \in \Delta^2$ with $h_{p}(y_1,y_2,y_3) = h_{p}^{\ast}$ such that $\max_{i\in \{1,2,3\}} |y_i- x_i| \le \varepsilon_1$. 
    Hence, we have $\min\{y_1,y_2,y_3\} \ge x_{\ast} - \varepsilon_1 \ge \frac{1}{4}\left(\frac{1}{3}\right)^{p+2}$, which, together with Proposition~\ref{PROP:tp-Lagrangian-2-covered} and Taylor's remainder theorem, implies that for $i\in \{1,2,3\}$, 
    \begin{align*}
        D_i h_p(x_1,x_2,x_3)
        \le D_i h_p(y_1,y_2,y_3)  + M \cdot \max_{ i\in \{1,2,3\}} |y_i-x_i|
        \le (2+p) h_{p}^{\ast} + \varepsilon_2. 
    \end{align*}
    Combining with Corollary~\ref{CORO:Lp-degree-expression-Lagrangian}, we know that for every $i \in \{1,2,3\}$ and for every $v\in V_i$, 
    \begin{align*}%\label{equ:proof:PROP:vertex-extendability-F5-3}
        d_{K^{3},2,p}(v) - d_{\mathcal{G},2,p}(v)
        & \le D_i h_p(x_1, x_2, x_3) \cdot n^{1+p} 
            - \left(1-\varepsilon\right) (2+p) h_{p}^{\ast} \cdot n^{1+p} \notag \\
        & \le 2 \varepsilon_{2} (2+p) h_{p}^{\ast} \cdot n^{1+p}. 
    \end{align*}
    By~\eqref{equ:LEMMA:local-monotone} and the fact that $z^p - (z-1)^p \ge p z^{p-1} - p^2 z^{p-2}$ for $z \ge 1$ and $p >0$, we obtain 
    \begin{align*}
        2 \varepsilon_{2} (2+p) h_{p}^{\ast} \cdot n^{1+p}
        \ge d_{K^{3},2,p}(v) - d_{\mathcal{G},2,p}(v)
        & \ge \sum_{T\in L_{K^3}(v)\setminus L_{\mathcal{G}}(v)} \left(d_{K^{3}}^{p}(T) - \left(d_{K^{3}}(T) - 1\right)^{p}\right) \\
        & \ge \sum_{T\in L_{K^3}(v)\setminus L_{\mathcal{G}}(v)} \left((x_{\ast} n)^{p} - \left(x_{\ast} n - 1\right)^{p}\right) \\
        & \ge |L_{K^3}(v)\setminus L_{\mathcal{G}}(v)| \left(p (x_{\ast} n)^{p-1} - p^2 (x_{\ast} n)^{p-2}\right) \\
        & \ge \left(d_{K^3}(v) - d_{\mathcal{G}}(v)\right) \frac{p (x_{\ast} n)^{p-1}}{2}, 
    \end{align*}
    which implies that 
    \begin{align}\label{equ:proof:PROP:vertex-extendability-F5-3}
        d_{K^3}(v) - d_{\mathcal{G}}(v)
        \le \frac{2\cdot 2 \varepsilon_{2} (2+p) h_{p}^{\ast} \cdot n^{1+p}}{p x_{\ast}^{p-1} n^{p-1}}
        \le \varepsilon_3 n^2. 
    \end{align}
    On the other hand, since $d_{\mathcal{H},2,p}(v_{\ast}) \ge \delta_{2,p}(\mathcal{H}) \ge \left(1-\frac{\varepsilon}{2}\right) (2+p) h_{p}^{\ast} \cdot n^{1+p}$, it follows from~\eqref{equ:LEMMA:Lp-degree-expression} that 
    \begin{align}\label{equ:proof:PROP:vertex-extendability-F5-4}
        d_{\mathcal{H}}(v_{\ast}) 
        \ge \frac{d_{\mathcal{H},2,p}(v_{\ast})}{p \binom{3}{2} n^{p-1}}
        \ge \frac{\left(1-\frac{\varepsilon}{2}\right) (2+p) h_{p}^{\ast}}{3p} n^2
        \ge \frac{1}{2} \frac{p+2}{3p} \cdot 3\left(\frac{1}{3}\right)^{p+2} n^2
        \ge \left(\frac{1}{3}\right)^{p+3} n^2. 
    \end{align}
    Now,~\eqref{equ:proof:PROP:vertex-extendability-F5-2},~\eqref{equ:proof:PROP:vertex-extendability-F5-3},~\eqref{equ:proof:PROP:vertex-extendability-F5-4} and Lemma~\ref{LEMMA:vertex-extendability-F5-3-partite} imply that $\mathcal{H}$ is $3$-partite. 
    In particular, $\mathcal{H} \in \mathfrak{S}$, proving Proposition~\ref{PROP:vertex-extendability-F5}. 
\end{proof}
%%%%%%%%%%%%%%%%%%%%%%%%%%%%%%%%%%%%%%
\subsection{Proofs for Lemmas~\ref{LEMMA:T3-induction-number-of-parts} and~\ref{LEMMA:inducibility-star}}\label{SUBSEC:proof-F5-inequalities}
We prove Lemmas~\ref{LEMMA:T3-induction-number-of-parts} and~\ref{LEMMA:inducibility-star} in this subsection. 
First, let us present the proof of Lemma~\ref{LEMMA:inducibility-star}. 

\begin{proof}[Proof of Lemma~\ref{LEMMA:inducibility-star}]
    Suppose to the contrary that this is not true. 
    Let $n$ be the minimum integer such that there exists $(x_1, \ldots, x_{n}) \in \Delta^{n-1}$ with $\sum_{1\le i < j \le n} \left(x_i x_j^p + x_i^p x_j \right) > g_{p}^{\ast}$. 
    By symmetry, we may assume that $x_1 \ge x_2 \ge \cdots \ge x_n$. 
    It follows from the assumption that $n \ge 3$ and $x_n > 0$. 
    Let $y_i \coloneqq x_i$ for $i\in [n-2]$ and $y_{n-1} \coloneqq x_{n-1}+x_{n}$. 
    Then 
    \begin{align*}
        & \sum_{1\le i < j \le n-1} \left(y_i y_j^p + y_i^p y_j \right)
        - \sum_{1\le i < j \le n} \left(x_i x_j^p + x_i^p x_j \right) \\
        & = y_{n-1}\left(\sum_{i\in [n-2]}y_i^p\right) + y_{n-1}^{p}\left(\sum_{i\in [n-2]}y_i\right) \\
        & \quad - (x_{n-1}+x_{n})\left(\sum_{i\in [n-2]}x_i^p\right) - \left(x_{n-1}^{p}+x_{n}^{p}\right)\left(\sum_{i\in [n-2]}x_i\right) - \left(x_{n-1}x_{n}^{p}+ x_{n-1}^{p}x_{n}\right) \\
        & = \left(\left(x_{n-1}+x_{n}\right)^{p} - \left(x_{n-1}^{p}+x_{n}^{p}\right)\right)\left(\sum_{i\in [n-2]}x_i\right) - \left(x_{n-1}x_{n}^{p}+ x_{n-1}^{p}x_{n}\right) \\
        & \ge \left(\left(x_{n-1}+x_{n}\right)^{p} - \left(x_{n-1}^{p}+x_{n}^{p}\right)\right)\left(\sum_{i\in [n-2]}x_i\right) 
            - \left(x_{n}^{p}+ x_{n-1}^{p-1}x_{n}\right)\left(\sum_{i\in [n-2]}x_i\right) \\
        % & = \left(\left(x_{n-1}+x_{n}\right)^{p} - \left(x_{n-1}^{p}+x_{n}^{p}\right) - \left(x_{n}^{p}+ x_{n-1}^{p-1}x_{n}\right)\right)\left(\sum_{i\in [n-2]}x_i\right) \\
        % & = \left(\left(x_{n-1}+x_{n}\right)^{p} - 2x_{n}^{p} - x_{n-1}^{p} - x_{n-1}^{p-1}x_n \right)\left(\sum_{i\in [n-2]}x_i\right) \\
        & = \left(\left(1+ \frac{x_n}{x_{n-1}}\right)^{p} - 2\left(\frac{x_n}{x_{n-1}}\right)^{p} - 1 - \frac{x_n}{x_{n-1}} \right) x_{n-1}^{p}\left(\sum_{i\in [n-2]}x_i\right) \\
        & \ge \left(\left(1+ \frac{x_n}{x_{n-1}}\right)^{2} - 2\left(\frac{x_n}{x_{n-1}}\right)^{2} - 1 - \frac{x_n}{x_{n-1}} \right) x_{n-1}^{p}\left(\sum_{i\in [n-2]}x_i\right) \\
        & = \left(\left(1- \frac{x_n}{x_{n-1}}\right) \frac{x_n}{x_{n-1}} \right) x_{n-1}^{p}\left(\sum_{i\in [n-2]}x_i\right)
        > 0, 
    \end{align*}
    contradicting the minimality of $n$. 
    Here we used the fact that $\sum_{i\in [n-2]}x_i \ge x_{n-1} \ge x_{n}$ and $0 \le \frac{x_n}{x_{n-1}} \le 1$. 
\end{proof}

Next, we present some estimations (Lemma~\ref{LEMMA:T3-inequalities}) for the values of $g_{p}^{\ast}$ and $h_{p}^{\ast}$, as determining the exact values of both appears quite difficult. 

The following lemma concerning $g_{p}$ follows easily from~{\cite[Lemma~5.1]{LMR23induced}}.  For completeness, we include its short proof here. 

\begin{lemma}\label{LEMMA:max-star-graph}
    For every real number $p > 3$, the function $x^{p}(1-x)+ x(1-x)^{p}$ is maximized on $[0,1/2]$ at a unique point $x_{\ast}$ with $x_{\ast} \in \left(\frac{1}{p+1}, \frac{1}{p-1}\right)$. 
\end{lemma}
\begin{proof}[Proof of Lemma~\ref{LEMMA:max-star-graph}]
    Let $y \coloneqq 2x(1-x) \in [0,1/2]$. 
    Then 
    \begin{align*}
        x^{p}(1-x)+ x(1-x)^{p}
        = \frac{y}{2^{p}} \left(\left(1-\sqrt{1-2y}\right)^{p-1}+ \left(1+\sqrt{1-2y}\right)^{p-1}\right) =: s_{p}(y). 
    \end{align*}
    It follows from~{\cite[Lemma~5.1]{LMR23induced}}\footnote{Even though the lemma is stated for $p \ge 4$, its proof works for all $p > 3$.} that there exists a unique $y_{\ast} \in \left(\frac{2p}{(p+1)^2}, \frac{2}{p+1}\right)$ such that $s_{p}(y)$ attains its maximum at $y_{\ast}$ on $[0,1/2]$.
    Therefore, $x^{p}(1-x)+ x(1-x)^{p}$ attains its maximum on $[0,1/2]$ at a unique point $x_{\ast} \in \left(\frac{1}{p+1}, \frac{1}{2}- \frac{1}{2}\sqrt{1-\frac{4}{p+1}}\right) \subset \left(\frac{1}{p+1}, \frac{1}{p-1}\right)$. 
\end{proof}

Using Lemmas~\ref{LEMMA:inducibility-star} and~\ref{LEMMA:max-star-graph}, we obtain the following inequalities for $g_{p}^{\ast}$. 

\begin{lemma}\label{LEMMA:g-p-ast-small}
    The following statements hold. 
    \begin{enumerate}[label=(\roman*)]
        \item\label{LEMMA:g-p-ast-small-1} $g_{p}^{\ast} = \frac{1}{2^{p}}$ if $2 \le p \le 3$, 
        \item\label{LEMMA:g-p-ast-small-2} $g_{p}^{\ast} \le \frac{1}{2^3} \left(\frac{p}{p+1}\right)^{p-3}$ if $p > 3$. 
    \end{enumerate}
\end{lemma}
\begin{proof}[Proof of Lemma~\ref{LEMMA:g-p-ast-small}]
    Notice that $g_{p}^{\ast} \ge g_{p}(1/2, 1/2) = 1/2^p$, so it suffices to prove the upper bound. 
    Simple calculations show that $g_{2}^{\ast} = 1/2^2$ and $g_{3}^{\ast} = 1/2^3$. 
    So it follows from Proposition~\ref{PROP:star-poly-Holder-inequ} (with $(p_1,p_2)=(2,3)$ and $p_2=3$) that 
    \begin{align*}
        g_{p}^{\ast} \le
        \left(g_{2}^{\ast}\right)^{3-p} \left(g_{3}^{\ast}\right)^{p-2}
        = \frac{1}{2^p}. 
    \end{align*}
    Next, we prove Lemma~\ref{LEMMA:g-p-ast-small}~\ref{LEMMA:g-p-ast-small-2}. 
    Let $x_{\ast} \in [0,1/2]$ be the unique point such that $g_{p}^{\ast} = g_{p}(x_{\ast}, 1-x_{\ast})$. 
    By Lemma~\ref{LEMMA:max-star-graph}, $x_{\ast}\ge \frac{1}{p+1}$, and hence, $1-x_{\ast} \le \frac{p}{p+1}$. 
    Therefore, 
    \begin{align*}
        g_{p}^{\ast} = g_{p}(x_{\ast}, 1-x_{\ast})
        & = x_{\ast}^{p}(1-x_{\ast}) + x_{\ast} (1-x_{\ast})^{p} \\
        & = x_{\ast}^{p-3} \cdot x_{\ast}^{3}(1-x_{\ast}) + (1-x_{\ast})^{p-3} \cdot x_{\ast} (1-x_{\ast})^{3}  \\
        & \le \left(x_{\ast}^{3}(1-x_{\ast}) + x_{\ast} (1-x_{\ast})^{3}\right) (1-x_{\ast})^{p-3} 
         \le \frac{1}{2^3} \left(\frac{p}{p+1}\right)^{p-3},  
    \end{align*}
    where the last inequality follows from $x_{\ast}^{3}(1-x_{\ast}) + x_{\ast} (1-x_{\ast})^{3} = g_{3}(x_{\ast},1-x_{\ast}) \le 1/2^3$ and $1-x_{\ast} \le p/(p+1)$. 
\end{proof}

In the next lemma, we establish additional inequalities for $g_{p}^{\ast}$ and use them to bound $h_{p}^{\ast}$. 

\begin{lemma}\label{LEMMA:T3-inequalities}
    The following inequalities hold. 
    \begin{enumerate}[label=(\roman*)]
        \item\label{LEMMA:T3-inequalities-1} $h_{p}^{\ast} \ge  h_{p}\left(\frac{1}{p+2}, \frac{1}{p+2}, \frac{p}{p+2}\right) > \frac{1}{e^2 (p+2)^2}$ for $p \ge 0$, 
        \item\label{LEMMA:T3-inequalities-2} $g_{p}\left(\frac{1}{p-1}, \frac{p-2}{p-1}\right) < \frac{p-2}{(p-1)^{p+1}} + \frac{1}{e(p-1)} $ for $p > 1$, 
        % \item $g_{p-1}\left(\frac{1}{p-2}, \frac{p-3}{p-2}\right) \le \frac{p-3}{(p-2)^{p}} + \frac{1}{p-2} \cdot \frac{1}{e}$ for $p > 2$,  
        \item\label{LEMMA:T3-inequalities-3} 
        % $h_{p-1}^{\ast}  \le \frac{3p^{p}}{2(p+1)^{p+1}} \cdot g_{p-1}^{\ast} < \frac{3}{2ep} \cdot g_{p-1}^{\ast}$ for $p \ge 1$,
        $h_{p}^{\ast} < \frac{3 \cdot g_{p}^{\ast}}{2e(p+1)}$ for $p \ge 0$, 
        \item\label{LEMMA:T3-inequalities-4} $g_{p}^{\ast}  \le \frac{p^2-p+2}{(p-2)(p+1)} \cdot g_{p}\left(\frac{1}{p-1}, \frac{p-2}{p-1}\right)$ for $p \ge 5$.      
    \end{enumerate}
    Here, $e = 2.718 \cdots$ is Euler's number.
\end{lemma}
\begin{proof}[Proof of Lemma~\ref{LEMMA:T3-inequalities}]
    We will use the following inequality$\colon$
    %$\left(\frac{z}{z+k}\right)^{z} > e^{-k} > \left(\frac{z}{z+k}\right)^{z+k}$ when $z \ge 0, k > 0$ multiple times. 
    \begin{align*}
        \left(\frac{z}{z+k}\right)^{z+k}
        < \frac{1}{e^{k}}
        < \left(\frac{z}{z+k}\right)^{z} 
        \quad\text{for all}\quad 
        z \ge 0,\ k > 0. 
    \end{align*}
    Lemma~\ref{LEMMA:T3-inequalities}~\ref{LEMMA:T3-inequalities-1} follows from 
    \begin{align*}
        h_{p}\left(\frac{1}{p+2}, \frac{1}{p+2}, \frac{p}{p+2}\right)
        &  = 2p\left(\frac{1}{p+2}\right)^{p+2} + \frac{1}{(p+2)^2} \cdot \left(\frac{p}{p+2}\right)^p  \\
        & > \frac{1}{(p+2)^2} \cdot \left(\frac{p}{p+2}\right)^p
        > \frac{1}{(p+2)^2} \cdot\frac{1}{e^2}.  
    \end{align*}
    Lemma~\ref{LEMMA:T3-inequalities}~\ref{LEMMA:T3-inequalities-2} follows from 
    \begin{align*}
        g_{p}\left(\frac{1}{p-1}, \frac{p-2}{p-1}\right)
         = \frac{p-2}{(p-1)^{p+1}}  + \frac{1}{p-1}\left(\frac{p-2}{p-1}\right)^{p} 
         < \frac{p-2}{(p-1)^{p+1}} + \frac{1}{p-1} \cdot \frac{1}{e}. 
    \end{align*}
    % % 
    % \begin{align*}
    %     (p+2)\cdot g_{p-1}\left(\frac{1}{p-2}, \frac{p-3}{p-2}\right)
    %     & = (p+2)(p-3)\left(\frac{1}{p-2}\right)^{p} + \frac{p+2}{p-2}\left(\frac{p-3}{p-2}\right)^{p} \\
    %     & \le \frac{(p+2)(p-3)}{(p-2)^{p}} + \frac{p+2}{p-2} \cdot \frac{1}{e}.
    % \end{align*}
    %
    Now, let $(x_1, x_2, x_3) \in \Delta^2$ be a vector such that $h_{p}^{\ast}=  h_{p}(x_1, x_2, x_3)$. 
    Notice that 
    \begin{align*}
         x_1x_2x_3^{p}+x_1x_2^{p}x_3
        & = x_1 \left(x_2x_3^{p} + x_2^{p}x_3\right) \\
        % & = x_1 \left(x_2+x_3\right)^{p} \left(\frac{x_2}{x_2+x_3}\left(\frac{x_3}{x_2+x_3}\right)^{p-1} + \left(\frac{x_2}{x_2+x_3}\right)^{p-1}\frac{x_3}{x_2+x_3}\right) \\
        & = x_1 \left(x_2+x_3\right)^{p+1} \cdot g_{p}\left(\frac{x_2}{x_2+x_3}, \frac{x_3}{x_2+x_3}\right) \\
        & \le x_1 \left(1-x_1\right)^{p+1} \cdot g_{p}^{\ast} 
        < \frac{g^*_{p}}{e (p+1)}, 
        %\frac{p^p}{(p+1)^{p+1}} \cdot g^*_{p-1},
    \end{align*}
    where the last inequality follows from the fact that $x(1-x)^{p+1} \le \frac{(p+1)^{p+1}}{(p+2)^{p+2}} = \frac{1}{p+1} \left(\frac{p+1}{p+2}\right)^{p+2} < \frac{1}{e (p+1)}$ for $x\in [0,1]$ and $p \ge 0$. 
    Consequently, 
    \begin{align*}
        2 h_{p}^{\ast} 
         = \left(x_1x_2x_3^{p}+x_1x_2^{p}x_3\right) + \left(x_1x_2^{p}x_3+x_1^{p}x_2x_3\right) + 
        \left( x_1x_2x_3^{p}+x_1^{p}x_2x_3\right) 
        < \frac{3\cdot g^*_{p}}{e (p+1)}, 
        %\frac{3}{2}\frac{p^p}{(p+1)^{p+1}} \cdot g^*_{p}. 
    \end{align*}
    which proves Lemma~\ref{LEMMA:T3-inequalities}~\ref{LEMMA:T3-inequalities-3}. 
    
    The following claim will be useful for proving Lemma~\ref{LEMMA:T3-inequalities}~\ref{LEMMA:T3-inequalities-4}. 
    \begin{claim}\label{CLAIM:derative-positive}
    Suppose that $p \ge 5$ is a real number. 
    Then $(x^{p-1}+(1-x)^{p-1})((p+1)x-1)$ is increasing on the interval $\left(0, \frac{2}{p+1} \right)$. 
    \end{claim}
    \begin{proof}[Proof of Claim~\ref{CLAIM:derative-positive}]
        Let $f(X) \coloneqq (X^{p-1}+(1-X)^{p-1})((p+1)X-1)$.
        It suffices to show that $D f(x) \ge 0$ for $x\in \left(0, \frac{2}{p+1} \right)$. 
        Fix $x \in \left(0, \frac{2}{p+1} \right)$. 
        Since $p \ge 5$, we have $x \le \frac{2}{p+1} < \frac{1}{2}$. Therefore, $(x^{p-2}-(1-x)^{p-2})\leq 0$, and it follows that 
        \begin{align*}
            D f(x)
            & = (p-1)(x^{p-2}-(1-x)^{p-2})((p+1)x-1)
                + (p+1)(x^{p-1}+(1-x)^{p-1}) \\
            & \ge (p-1)(x^{p-2}-(1-x)^{p-2})\left((p+1)\cdot \frac{2}{p+1}-1 \right)
                + (p+1)(x^{p-1}+(1-x)^{p-1}) \\
            & = (p-1)(x^{p-2}-(1-x)^{p-2})
                + (p+1)(x^{p-1}+(1-x)^{p-1}) \\
            & \ge - (p-1) (1-x)^{p-2} + (p+1) (1-x)^{p-1} \\
            & = (p+1)(1-x)^{p-2} \left( \frac{2}{p+1}-x\right)
            \ge 0, 
        \end{align*}
        proving Claim~\ref{CLAIM:derative-positive}. 
    \end{proof}
    For convenience, let $\tilde{g}_{p}(x) \coloneqq g_{p}(x,1-x)$ for $x\in [0,1]$. 
    Suppose that $x_{\ast}$ is the unique point in $[0,1/2]$ where $\tilde{g}_{p}(x)$ attains its maximum. 
    By Lemma~\ref{LEMMA:max-star-graph}, $x_{\ast} \in \left(\frac{1}{p+1}, \frac{1}{p-1}\right)$. 
    Let $(x_1, x_2) \coloneqq \left(\frac{1}{p-1}, \frac{p-2}{p-1}\right)$. 
    It follows from the Mean Value Theorem that there exists $y_1 \in [x_{\ast}, x_1] \subset \left[\frac{1}{p+1}, \frac{1}{p-1}\right]$ such that $g_{p}^{\ast} = \tilde{g}_{p}(x_{\ast}) = \tilde{g}_{p}(x_1) + D \tilde{g}_{p}(y_1) \left(x_{\ast}-x_1\right)$. 
    Let $y_2 \coloneqq 1- y_1$ and $\alpha \coloneqq x_1 - x_{\ast}  \le \frac{1}{p-1} - \frac{1}{p+1} = \frac{2}{(p-1)(p+1)}$. 
    Simple calculations show that 
    \begin{align*}
        D \tilde{g}_{p}(y_1) \left(x_{\ast} - x_1\right)
        & = \alpha \left(y_1^{p-1}((p+1)y_1 - p) + y_2^{p-1}((p+1) y_1 - 1)\right) \\
        & \le \alpha \left(y_1^{p-1} + y_2^{p-1}\right) \left((p+1)y_1 - 1\right)
        \le \frac{2 \left(y_1^{p-1} + y_2^{p-1}\right) \left((p+1)y_1 - 1\right)}{(p-1)(p+1)}.
    \end{align*}
    Since $y_1 \le x_1 = \frac{1}{p-1} < \frac{2}{p+1}$ (recall that $p \ge 5$), by Claim~\ref{CLAIM:derative-positive}, the inequality above continues as  
    \begin{align*}
        D \tilde{g}_{p}(y_1) \left(x_{\ast} - x_1\right)
        & \le \frac{2 \left(x_1^{p-1} + x_2^{p-1}\right) \left((p+1)x_1 - 1\right)}{(p-1)(p+1)} \\
        & = \frac{2 \left(x_1^{p-1} + x_2^{p-1}\right) \frac{2 x_1 x_2}{x_2}}{(p-1)(p+1)} 
         = \frac{4 \left(x_1^{p}x_2 + x_1 x_2^{p}\right)}{(p-1)(p+1) x_2} 
        = \frac{4 \cdot g_{p}(x_1, x_2)}{(p-2)(p+1)}. 
    \end{align*}
    Therefore (recall that $\tilde{g}_{p}(x_1) = g_{p}(x_1, 1-x_1) = g_{p}(x_1, x_2)$), 
    \begin{align*}
       g_{p}^{\ast} 
       & = \tilde{g}_{p}(x_1) + D \tilde{g}_{p}(y_1) \left(x_{\ast} - x_1\right) \\
       & \le g_{p}(x_1, x_2) + \frac{4 \cdot g_{p}(x_1, x_2)}{(p-2)(p+1)}
       = \frac{p^2-p+2}{(p-2)(p+1)} \cdot g_{p}(x_1, x_2), 
    \end{align*}
    proving Lemma~\ref{LEMMA:T3-inequalities}~\ref{LEMMA:T3-inequalities-4}. 
\end{proof}

Now we are ready to prove Lemma~\ref{LEMMA:T3-induction-number-of-parts}. 
\begin{proof}[Proof of Lemma~\ref{LEMMA:T3-induction-number-of-parts}]
    We consider the following two cases. 
    
    \textbf{Case 1:} $2 \le p \le  8$. 

    It follows from Lemma~\ref{LEMMA:T3-inequalities}~\ref{LEMMA:T3-inequalities-3} and Lemma~\ref{LEMMA:g-p-ast-small} that 
    \begin{align*}
        \frac{p \cdot h_{p-1}^{\ast}  +  g_{p}^{\ast} }{(p+2)\cdot h_{p}^{\ast}}
        & \le \frac{p \cdot \frac{3}{2ep} \cdot g_{p-1}^{\ast}  +  g_{p}^{\ast} }{(p+2)\cdot h_{p}^{\ast} } \\
        & \le \frac{\frac{3}{2e} \cdot \max\left\{\frac{1}{2^{p-1}}, \frac{1}{2^3}\left(\frac{p-1}{p}\right)^{p-4}\right\}  +  \max\left\{\frac{1}{2^{p}}, \frac{1}{2^3}\left(\frac{p}{p+1}\right)^{p-3}\right\} }{(p+2)\cdot h_{p}\left(\frac{1}{p+2}, \frac{1}{p+2}, \frac{p}{p+2}\right)} 
         < 6.2, 
    \end{align*}
    where the last inequality is verified using Mathematica.
    %and can be rigorously derived through further calculations.

    \medskip 

    \textbf{Case 2:} $p > 8$.

    Let $\Phi \coloneqq \frac{p \cdot h_{p-1}^{\ast}  +  g_{p}^{\ast} }{(p+2)\cdot h_{p}\left(\frac{1}{p+2}, \frac{1}{p+2}, \frac{p}{p+2}\right)}$. 
    Since $h_{p}^{\ast} \ge h_{p}\left(\frac{1}{p+2}, \frac{1}{p+2}, \frac{p}{p+2}\right)$, we obtain $\frac{p \cdot h_{p-1}^{\ast}  +  g_{p}^{\ast} }{(p+2)\cdot h_{p}^{\ast} } \le \Phi$. So it suffices to prove that $\Phi < 6.88$. 
    By Lemma~\ref{LEMMA:T3-inequalities}, $\Phi$ is upper bounded by 
    \begin{align*}
        & \quad \frac{p \cdot \frac{3}{2ep} \cdot \frac{(p-1)^2+3}{(p-1)^2-1} \cdot g_{p-1}\left(\frac{1}{p-2}, \frac{p-3}{p-2}\right)  + \frac{p^2-p+2}{(p-2)(p+1)} \cdot g_{p}\left(\frac{1}{p-1}, \frac{p-2}{p-1}\right)}{(p+2) \cdot \frac{1}{(p+2)^2 e^2}} \\
        & = (p+2)e^2 \left(\frac{3}{2e} \cdot \frac{(p-1)^2+3}{(p-1)^2-1} \cdot g_{p-1}\left(\frac{1}{p-2}, \frac{p-3}{p-2}\right)  + \frac{p^2-p+2}{(p-2)(p+1)} \cdot g_{p}\left(\frac{1}{p-1}, \frac{p-2}{p-1}\right) \right)\\
        & = \frac{3e}{2} \cdot \frac{(p-1)^2+3}{(p-1)^2-1} \cdot \left(\frac{(p+2)(p-3)}{(p-2)^p} + \frac{p+2}{p-2} \cdot \frac{1}{e}\right)  \\
        & \quad + \frac{p^2-p+2}{(p-2)(p+1)} \cdot \left(\frac{(p+2)(p-2)}{(p-1)^{p+1}} + \frac{p+2}{p-1} \cdot \frac{1}{e}\right)e^2 . 
    \end{align*}
    It is not hard to verify that $\frac{(p-1)^2+3}{(p-1)^2-1}, \frac{(p+2)(p-3)}{(p-2)^{p}}, \frac{p+2}{p-2}, \frac{p^2-p+2}{(p-2)(p+1)}, \frac{(p+2)(p-2)}{(p-1)^{p+1}},  \frac{p+2}{p-1}$ are all decreasing in $p$ on $[8, \infty)$. So plugging $p=8$ into the inequality above we obtain  
    \begin{align*}
        \Phi
        \le \frac{65}{24}+\frac{72163555 e}{47029248}+\frac{580 e^2}{363182463}
        < 6.88,
    \end{align*}
    proving Lemma~\ref{LEMMA:T3-induction-number-of-parts}.
\end{proof}

%%%%%%%%%%%%%%%%%%%%%%%%%%%%%%%%%%%%%%%
\section{Proofs for Theorems~\ref{THM:Lp-F5-p-small-a},~\ref{THM:Lp-F5-p-small-b}, and~\ref{THM:Lp-clique-expansion-p-small}}\label{SEC:Proof-Lp-F5-small}
In this section, we prove Theorems~\ref{THM:Lp-F5-p-small-a},~\ref{THM:Lp-F5-p-small-b}, and~\ref{THM:Lp-clique-expansion-p-small}. 
Note from the first inequality in~\eqref{equ:PROP:tp-norm-Jensen-small} that to bound $\norm{\mathcal{H}}_{t,p}$ of an $\mathcal{F}$-free $r$-graph $\mathcal{H}$ when $p < 1$, it suffices to bound the product $|\mathcal{H}|^{p} \cdot |\partial_{r-t}\mathcal{H}|^{1-p}$. 
This is where the results on the feasible region problem introduced in~\cite{LM21feasible} can be applied. 
Extending the classical Kruskal--Katona Theorem~\cite{KR63,KA68} and the Tur\'{a}n problem, the feasible region problem of $\mathcal{F}$ studies the maximum size of an $n$-vertex $\mathcal{F}$-free $r$-graph $\mathcal{H}$ under the constraint that $|\partial_{r-t}\mathcal{H}|$ is fixed. 
By using results from~\cite{LM21feasible} (specifically Theorems~\ref{THM:LM-feasible-region-T3} and~\ref{THM:LM-feasible-region-expansion} below), we can reduce the task of bounding $|\mathcal{H}|^{p} \cdot |\partial_{r-t}\mathcal{H}|^{1-p}$ to a simple optimization problem with only one-variable, namely $|\partial_{r-t}\mathcal{H}|$. 
Similarly, the stability for the $(t,p)$-norm Tur\'{a}n problem can be derived easily by applying the corresponding stability theorems on the feasible region problem established in~\cite{LM23KKstability,Liu20a}.   

Given integers $r > t \ge 1$ and a real number $p > 0$, for every $n$-vertex $r$-graph $\mathcal{H}$ let 
\begin{align*}
    \rho_{t,p}(\mathcal{H})
    \coloneqq \frac{\norm{\mathcal{H}}_{t,p}}{\binom{n}{t} \cdot n^{p(r-t)}}
    \quad\text{and}\quad 
    \rho(\mathcal{H})
    \coloneqq \frac{|\mathcal{H}|}{\binom{n}{r}}. 
\end{align*}
First, we use the following theorem on $\mathcal{T}_{3}$ to prove Theorems~\ref{THM:Lp-F5-p-small-a},~\ref{THM:Lp-F5-p-small-b}. 
\begin{theorem}[{\cite[Theorems~4.4 and~4.6]{LM21feasible}}]\label{THM:LM-feasible-region-T3}
    Suppose that $\mathcal{H}$ is an $n$-vertex $\mathcal{T}_3$-free $3$-graph. 
    Then 
    \begin{align*}
        \rho(\mathcal{H})
        \le \min\left\{\frac{\left(\rho(\partial\mathcal{H})\right)^{3/2}}{\sqrt{6}},\ \rho(\partial\mathcal{H}) \cdot \left(1-\rho(\partial\mathcal{H})\right)\right\} + o(1). 
    \end{align*}
\end{theorem}

\begin{proof}[Proof of Theorems~\ref{THM:Lp-F5-p-small-a} and~\ref{THM:Lp-F5-p-small-b}]
    The constructions for the lower bounds come from balanced blowups of STSs, so it suffices to focus on the upper bound. 
    Since $\mathcal{T}_{3} \le_{\mathrm{hom}} F_5$,
    by Proposition~\ref{PROP:tp-density-blowup-lemma}, it suffices to show that 
    \begin{align*}
        \pi_{2,p}(\mathcal{T}_3) 
        \le 
            \begin{cases}
                \frac{2}{3^{1+p}}, & \quad\text{if}\quad p \in [1/2, 1), \\
                \frac{p^p}{(p+1)^{p+1}}, & \quad\text{if}\quad p \in (0,1/2).  
            \end{cases}
    \end{align*}
    Let $n$ be a sufficiently large integer, and $\mathcal{H}$ be an $n$-vertex $\mathcal{T}_3$-free $3$-graph with $\norm{\mathcal{H}}_{2,p} = \mathrm{ex}_{2,p}(n,\mathcal{T}_3)$. 
    Let $x \coloneqq \rho(\partial\mathcal{H})$ and $y \coloneqq \rho(\mathcal{H})$. 
    By~\eqref{equ:PROP:tp-norm-Jensen-small}, 
    \begin{align*}
        \frac{\norm{\mathcal{H}}_{2,p}}{\binom{n}{2}\cdot n^{p}}
        \le 2! \cdot \left(\binom{3}{2}\cdot \frac{\rho(\mathcal{H})}{3!}\right)^{p} \cdot \left(\frac{\rho(\partial\mathcal{H})}{2!}\right)^{1-p} + o(1) 
        = y^{p} x^{1-p} + o(1). 
    \end{align*}
    Combining with Theorem~\ref{THM:LM-feasible-region-T3}, we obtain 
    \begin{align}\label{equ:proof-THM:Lp-F5-p-small}
        \pi_{2,p}(\mathcal{T}_3) -o(1)
        = \frac{\norm{\mathcal{H}}_{2,p}}{\binom{n}{2}\cdot n^{p}}
        % = \rho_{2,p}(\mathcal{H})
        \le y^{p} x^{1-p}
        \le \min\left\{ \frac{x^{1+\frac{p}{2}}}{6^{\frac{p}{2}}},\ x(1-x)^{p} \right\}. 
    \end{align}
    Simple calculations using Fact~\ref{FACT:inequality-b}~\ref{FACT:inequality-b-2} give the desired upper bound for $\pi_{2,p}(\mathcal{T}_3)$. 

    Next, we prove the stability part. 
    Suppose that $\mathcal{H}$ is an $n$-vertex $F_5$-free $3$-graph with $\norm{\mathcal{H}}_{2,p} = (1-o(1)) \pi_{2,p}(F_5) n^{2+p}/2$. 
    Since $\mathcal{T}_{3}\le_{\mathrm{hom}} F_5$, it follows from Proposition~\ref{PROP:tp-density-blowup-lemma} that there exists a $\mathcal{T}_{3}$-free subgraph $\mathcal{G} \subset \mathcal{H}$ with 
    \begin{align*}
        \norm{\mathcal{G}}_{2,p} 
        = \norm{\mathcal{H}}_{2,p} -o(n^{t+p(r-t)}) 
        = (1-o(1)) \frac{\pi_{2,p}(F_5)}{2}n^{2+p}
        = (1-o(1)) \frac{\pi_{2,p}(\mathcal{T}_{3})}{2}n^{2+p}.
    \end{align*}
    Let $x \coloneqq \rho(\partial\mathcal{G})$ and $y \coloneqq \rho(\mathcal{G})$. 
    Similar to~\eqref{equ:proof-THM:Lp-F5-p-small}, we obtain 
    \begin{align*}
        \pi_{2,p}(\mathcal{T}_3) -o(1)
        \le \frac{\norm{\mathcal{G}}_{2,p}}{\binom{n}{2}\cdot n^{p}}
        \le y^{p} x^{1-p}
        \le \min\left\{ \frac{x^{1+\frac{p}{2}}}{6^{\frac{p}{2}}},\ x(1-x)^{p} \right\}.
    \end{align*}
    If $p \in [1/2,1]$, then simple calculations show that 
    \begin{align*}
        |x - 1/3| = o(1)
        \quad\text{and}\quad 
        |y - 2/9| =o(1), 
    \end{align*}
    which, by~{\cite[Theorem~1.5]{LM23KKstability}}, implies that $\mathcal{G}$ is $3$-partite after removing $o(n^3)$ edges. 

    If $p = 1/k$ for some $k \in 6\mathbb{N}+\{0,2\}$, then simple calculations show that 
    \begin{align*}
        |x - k/(k+1)| = o(1)
        \quad\text{and}\quad 
        |y - k/(k+1)^2| =o(1), 
    \end{align*}
    which, by~{\cite[Theorem~1.6]{Liu20a}}, implies that $\mathcal{G}$ is $\mathcal{S}$-coloralbe for some $\mathcal{S} \in \mathrm{STS}(k+1)$ after removing $o(n^3)$ edges. 
\end{proof}

Next, we use the following theorem on expansions to prove Theorem~\ref{THM:Lp-clique-expansion-p-small}. 

Let $\ell \ge r \ge 2$ be integers.
An $r$-graph $F$ is a \textbf{weak expansion} of $K_{\ell+1}$ if it can be obtained from $K_{\ell+1}$ by adding $r-2$ vertices into each edge. 
A key difference from the expansion $H_{\ell+1}^{r}$ of $K_{\ell+1}$ is that these added $(r-2)$-sets do not need to be pairwise disjoint. 
We use $\mathcal{K}_{\ell+1}^{r}$ to denote the collection of all weak expansions of $K_{\ell+1}$. 
A useful fact is that $\mathcal{K}_{\ell+1}^{r} \le_{\mathrm{hom}} H_{F}^{r}$ for every graph $F$ with $\chi(F) = \ell+1$ (see~{\cite[Section~3]{MU06}}). 

\begin{theorem}[{\cite[Theorem~1.7]{LM21feasible}}]\label{THM:LM-feasible-region-expansion}
    Let $\ell \ge r > t \ge 1$ be integers. 
    Suppose that $\mathcal{H}$ is an $n$-vertex $\mathcal{K}_{\ell+1}^{r}$-free $r$-graph. 
    Then 
    \begin{align*}
        \rho(\partial_{r-t}\mathcal{H})
        \le t! \binom{\ell}{t} \left(\frac{1}{\ell}\right)^{t} + o(1)
        \quad\text{and}\quad 
        \rho(\mathcal{H}) 
        \le r! \binom{\ell}{r} \left(\frac{\rho(\partial_{r-t}\mathcal{H})}{t! \binom{\ell}{t}}\right)^{\frac{r}{t}} + o(1). 
    \end{align*}
\end{theorem}

\begin{proof}[Proof of Theorem~\ref{THM:Lp-clique-expansion-p-small}]
    Let $r \ge 2$ be an integer and $F$ be a graph with $\chi(F) = \ell + 1> r$.
    Since $\mathcal{K}_{\ell+1}^{r} \le_{\mathrm{hom}} H_{F}^{r}$, by Proposition~\ref{PROP:tp-density-blowup-lemma}, it suffices for the first part of Theorem~\ref{THM:Lp-clique-expansion-p-small} to show that 
    \begin{align*}
        \pi_{t,p}(\mathcal{K}_{\ell+1}^{r})
        \le t!\binom{\ell}{t}\binom{\ell-t}{r-t}^{p}\left(\frac{1}{\ell}\right)^{t+p(r-t)}. 
    \end{align*}
    Let $n$ be a sufficiently large integer, and $\mathcal{H}$ be an $n$-vertex $\mathcal{K}_{\ell+1}^{r}$-free $r$-graph with $\norm{\mathcal{H}}_{t,p} = \mathrm{ex}_{t,p}(n,\mathcal{K}_{\ell+1}^{r})$. 
    Let $x \coloneqq \rho(\partial_{r-t}\mathcal{H})$ and $y \coloneqq \rho(\mathcal{H})$. 
    By~\eqref{equ:PROP:tp-norm-Jensen-small}, 
    \begin{align*}
        \pi_{t,p}(\mathcal{K}_{\ell+1}^{r}) 
        %=\rho_{2,p}(\mathcal{H})
        = \frac{\norm{\mathcal{H}}_{t,p}}{\binom{n}{t}\cdot n^{p(r-t)}} + o(1)
        & \le t! \cdot \left(\binom{r}{t}\cdot \frac{\rho(\mathcal{H})}{r!}\right)^{p} \cdot \left(\frac{\rho(\partial_{r-t}\mathcal{H})}{t!}\right)^{1-p} + o(1)  \\
        & = \frac{y^{p} x^{1-p}}{\left((r-t)!\right)^{p}} + o(1). 
    \end{align*}
    Combining with Theorem~\ref{THM:LM-feasible-region-T3}, we obtain $x \le t! \binom{\ell}{t} \left(\frac{1}{\ell}\right)^{t} + o(1)$ and 
    \begin{align*}
        \pi_{t,p}(\mathcal{K}_{\ell+1}^{r}) 
        & \le \frac{1}{\left((r-t)!\right)^{p}} \left(r!\binom{\ell}{r}\left(\frac{x}{t!\binom{\ell}{t}}\right)^{\frac{r}{t}}\right)^{p} x^{1-p} + o(1)  \\
        & \le \frac{1}{\left((r-t)!\right)^{p}} \left(r!\binom{\ell}{r}\left(\frac{1}{\ell}\right)^{r}\right)^{p} \left(t! \binom{\ell}{t} \left(\frac{1}{\ell}\right)^{t}\right)^{1-p} + o(1) \\
        & = t! \binom{\ell}{t} \binom{\ell-t}{r-t}^{p} \left(\frac{1}{\ell}\right)^{t+p(r-t)} + o(1),  
    \end{align*}
    completing the proof of first part 
    of Theorem~\ref{THM:Lp-clique-expansion-p-small}. 
    The proof for the stability part closely follows that of Theorem~\ref{THM:Lp-F5-p-small-a} with~{\cite[Theorem~1.5]{LM23KKstability}} replaced by~{\cite[Theorem~1.8]{LM23KKstability}}, so we omit the details here. 
\end{proof}

%%%%%%%%%%%%%%%%%%%%%%%%%%%%%%%%%%%%%%%
\section{Concluding remarks}\label{SEC:Remark}
Given integers $r > t \ge 1$ and an integer $p \ge 1$, let $S_{t,p}^{r}$ denote the $r$-graph with $p$ edges $\{E_1, \ldots, E_{p}\}$ such that there exists a $t$-set $T$ with $E_i \cap E_j = T$ for $1\le i < j \le p$. 
Given an $r$-graph, let $\mathrm{hom}(S_{t,p}^{r}, \mathcal{H})$ and $\mathrm{inj}(S_{t,p}^{r}, \mathcal{H})$ denote the number of homomorphisms\footnote{Both $S_{t,p}^{r}$ and $\mathcal{H}$ are vertex-labelled.} and injective homomorphisms from $S_{t,p}^{r}$ to $\mathcal{H}$, respectively.  
It is easy to see that $\mathrm{hom}(S_{t,p}^{r}, \mathcal{H}) = t! \left((r-t)!\right)^{p} \cdot \norm{\mathcal{H}}_{t,p}$ and $\mathrm{inj}(S_{t,p}^{r}, \mathcal{H}) = \mathrm{hom}(S_{t,p}^{r}, \mathcal{H}) + O(n^{t+p(r-t)-1})$ for every $\mathcal{H} \in \mathfrak{G}$. 

For a family $\mathcal{F}$ of $r$-graphs, determining the maximum value of $\mathrm{inj}(S_{t,p}^{r}, \mathcal{H})$ in an $n$-vertex $\mathcal{F}$-free $r$-graph $\mathcal{H}$ is equivalent to solving the generalized Tur\'{a}n problem $\mathrm{ex}(n,S_{t,p}^{r}, \mathcal{F})$, a central topic in Extremal Combinatorics (see e.g.~\cite{Erdos62,AS16}). 
The results presented in this paper (Theorems~\ref{THM:Lp-F5-p-large},~\ref{THM:Lp-clique-expansion-p-large},~\ref{THM:Lp-general-a}, and~\ref{THM:Lp-general-b}) can be shown to apply to $\mathrm{inj}(S_{t,p}^{r}, \mathcal{H})$ with minor modifications to the current proofs. 

Recall that Theorem~\ref{THM:Lp-F5-p-small-b} determined $\pi_{2,p}(F_5)$ for $p$ around points in $\{k^{-1} \colon k \in 6\mathbb{N}^{+} + \{0,2\}\}$. 
It seems plausible to conjecture that for every $p \in (0,1/2)$, the (asymptotic) extremal construction for the $(2,p)$-norm Tur\'{a}n problem of $F_5$ is a blowup of some STS. 

\begin{conjecture}
    For every $p \in (0,1/2)$, 
    \begin{align*}
        \pi_{2,p}(F_5)
        = \max\left\{2 \cdot \lambda_{2,p}(\mathcal{S}) \colon \text{$\mathcal{S}$ is an STS}\right\}. 
    \end{align*}
\end{conjecture}

In general, one could consider the $(t,p)$-norm version of every extremal problem (provided it is meaningful), here we list only a few of them. 

The following question is an extension of the Erd\H{o}s--Rademacher Problem~\cite{Erdos55}. 

\begin{problem}
    Let $n \ge r > t \ge 1$ be integers and $p > 0$ be a real number. 
    Let $F$ be an $r$-graph. 
    Suppose that $m$ is an integer greater than $\mathrm{ex}_{t,p}(n,F)$. 
    Determine 
    \begin{align*}
        \min\left\{\mathrm{inj}(F,\mathcal{H}) \colon \norm{\mathcal{H}}_{t,p} = m \text{ and } v(\mathcal{H}) = n\right\}. 
    \end{align*}
\end{problem}

The following question is an extension of the feasible region problem introduced in~\cite{LM21feasible}. 
\begin{problem}
    Let $n \ge r > t \ge 1$ be integers and $p > 0$ be a real number. 
    Let $\mathcal{F}$ be a family of $r$-graphs. 
    For every feasible positive integer $m$, determine 
    \begin{align*}
        \min\left\{\norm{\mathcal{H}}_{t,p} \colon |\partial\mathcal{H}| = m,\ v(\mathcal{H}) = n, \text{ and $\mathcal{H}$ is $\mathcal{F}$-free} \right\}. 
    \end{align*}
\end{problem}
\textbf{Remark.}
The case $\mathcal{F} = \emptyset$ is  an extension of the Kruskal--Katona Theorem~\cite{KA68,KR63}. 

The follow question is an extension of the Kleitman--West Problem (see e.g.~\cite{AK78,Har91,AC99,DGS16,RW18,GLM21}). 
\begin{problem}
    Let $n \ge r > t \ge 1$ be integers and $p > 0$ be a real number. 
    For every feasible positive integer $m$, determine 
    \begin{align*}
        \min\left\{\norm{\mathcal{H}}_{t,p} \colon |\mathcal{H}| = m \text{ and } v(\mathcal{H}) = n\right\}. 
    \end{align*}
\end{problem}
%%%%%%%%%%%%%%%%%%%%%%%%%%%%%%%%%%%%%%%
\bibliographystyle{alpha}%abbrv
\bibliography{LpNorm}
%%%%%%%%%%%%%%%%%%%%%%%%%%%%%%%%%%%%%%%
\begin{appendix}
\section{Non-vertex-extendability for $p < 1$}\label{APPENDIX:SEC:construction}
\begin{theorem}
    The $3$-graph $F_5$ is not $(2,1/2)$-vertex-extendable with respect to $\mathfrak{K}^3_{3}$, where $\mathfrak{K}^3_{3}$ is the collection of all $3$-partite $3$-graphs.
\end{theorem}
\begin{proof}
    Fix $\delta,\varepsilon_1,\varepsilon$ to be sufficiently small and $n$ to be sufficiently large with $0\le n^{-4}\ll \delta\ll \varepsilon_1\ll \varepsilon\ll 1$.
    Let
    \begin{equation}\label{eq: definition-of-dn}
        d(n)\coloneqq (1-\varepsilon)\mathrm{exdeg}_{2, 1/2}=(1-\varepsilon)\cdot \dfrac{ \left(2+1/2\right) \mathrm{ex}_{2,1/2}(n,F_5)}{n}.
    \end{equation}
    \begin{claim}\label{CLAIM: eq-lower-bound}
        $d(n)\le (1-\dfrac{\varepsilon}{2})\cdot \dfrac{5}{2}\left(\dfrac{n}{3}\right)^{\frac{3}{2}}.$
    \end{claim}
    \begin{proof}
        By Theorem~\ref{THM:Lp-F5-p-small-a}, 
        $\lim_{n\to \infty}\frac{2  \mathrm{ex}_{2,1/2}(n,F_5)}{n^{2+1/2}}=\pi_{2,\frac{1}{2}}(F_5)=\frac{2}{3^{1.5}}.$
        So by $n$ is sufficiently large, 
        $ \mathrm{ex}_{2,\frac{1}{2}}(n,F_5)\le   \left(\frac{1}{3^{1.5}} +\delta\right)n^{2.5}\le (1+9\delta )\cdot \frac{n^{2.5}}{3^{1.5}}.$
        Then \eqref{eq: definition-of-dn} continues as
        \begin{align*}
            d(n) \le (1-\varepsilon)\cdot \dfrac{(1+1/2)(1+9\delta)\cdot \dfrac{n^{2.5}}{3^{1.5}}}{n}  =(1-\varepsilon)(1+9\delta)\cdot \dfrac{5}{2}\left(\dfrac{n}{3}\right)^{\frac{3}{2}}\le (1-\frac{\varepsilon}{2}) \cdot \dfrac{5}{2}\left(\dfrac{n}{3}\right)^{\frac{3}{2}},
        \end{align*}
        where the last inequality holds by $0< \delta \ll \varepsilon\ll 1$.
    \end{proof}%CLAIM
    To prove the theorem, it suffices to show there exists an $(n+1)$-vertex $F_5$-free $3$-graph $\mathcal{H}$ with $\delta_{2,\frac{1}{2}}(\mathcal{H})\ge d(n)$ and $\mathcal{H}-v_\ast\in \mathfrak{K}^3_{3}$ for some $v_\ast\in V(\mathcal{H})$, but $\mathcal{H}\notin \mathfrak{K}^3_{3}.$
%%%%%%%%%%%%%%%%%%%%%%%%%%%%%%%%%%%%%%%%%%%%%%%%
%Insert a Tikz
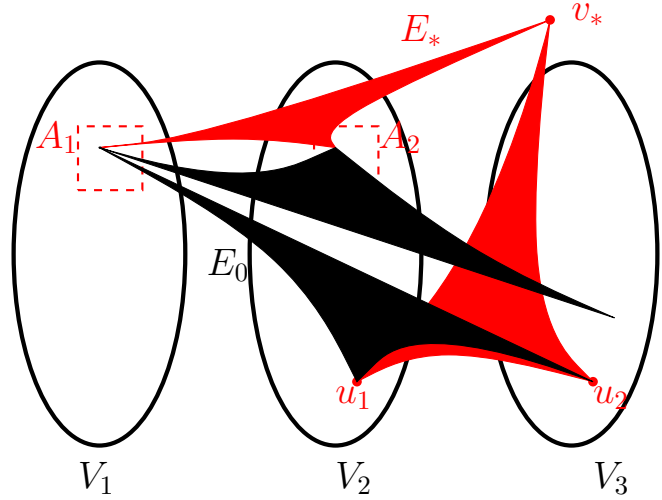
\usetikzlibrary{arrows.meta,patterns}
\usetikzlibrary{arrows.meta}
\makeatletter

\pgfdeclarearrow{
  name = ipe _linear,
  defaults = {
    length = +1bp,
    width  = +.666bp,
    line width = +0pt 1,
  },
  setup code = {
    % Control points
    \pgfarrowssetbackend{0pt}
    \pgfarrowssetvisualbackend{
      \pgfarrowlength\advance\pgf@x by-.5\pgfarrowlinewidth}
    \pgfarrowssetlineend{\pgfarrowlength}
    \ifpgfarrowreversed
      \pgfarrowssetlineend{\pgfarrowlength\advance\pgf@x by-.5\pgfarrowlinewidth}
    \fi
    \pgfarrowssettipend{\pgfarrowlength}
    % Convex hull
    \pgfarrowshullpoint{\pgfarrowlength}{0pt}
    \pgfarrowsupperhullpoint{0pt}{.5\pgfarrowwidth}
    % The following are needed in the code:
    \pgfarrowssavethe\pgfarrowlinewidth
    \pgfarrowssavethe\pgfarrowlength
    \pgfarrowssavethe\pgfarrowwidth
  },
  drawing code = {
    \pgfsetdash{}{+0pt}
    \ifdim\pgfarrowlinewidth=\pgflinewidth\else\pgfsetlinewidth{+\pgfarrowlinewidth}\fi
    \pgfpathmoveto{\pgfqpoint{0pt}{.5\pgfarrowwidth}}
    \pgfpathlineto{\pgfqpoint{\pgfarrowlength}{0pt}}
    \pgfpathlineto{\pgfqpoint{0pt}{-.5\pgfarrowwidth}}
    \pgfusepathqstroke
  },
  parameters = {
    \the\pgfarrowlinewidth,%
    \the\pgfarrowlength,%
    \the\pgfarrowwidth,%
  },
}

\pgfdeclarearrow{
  name = ipe _pointed,
  defaults = {
    length = +1bp,
    width  = +.666bp,
    inset  = +.2bp,
    line width = +0pt 1,
  },
  setup code = {
    % Control points
    \pgfarrowssetbackend{0pt}
    \pgfarrowssetvisualbackend{\pgfarrowinset}
    \pgfarrowssetlineend{\pgfarrowinset}
    \ifpgfarrowreversed
      \pgfarrowssetlineend{\pgfarrowlength}
    \fi
    \pgfarrowssettipend{\pgfarrowlength}
    % Convex hull
    \pgfarrowshullpoint{\pgfarrowlength}{0pt}
    \pgfarrowsupperhullpoint{0pt}{.5\pgfarrowwidth}
    \pgfarrowshullpoint{\pgfarrowinset}{0pt}
    % The following are needed in the code:
    \pgfarrowssavethe\pgfarrowinset
    \pgfarrowssavethe\pgfarrowlinewidth
    \pgfarrowssavethe\pgfarrowlength
    \pgfarrowssavethe\pgfarrowwidth
  },
  drawing code = {
    \pgfsetdash{}{+0pt}
    \ifdim\pgfarrowlinewidth=\pgflinewidth\else\pgfsetlinewidth{+\pgfarrowlinewidth}\fi
    \pgfpathmoveto{\pgfqpoint{\pgfarrowlength}{0pt}}
    \pgfpathlineto{\pgfqpoint{0pt}{.5\pgfarrowwidth}}
    \pgfpathlineto{\pgfqpoint{\pgfarrowinset}{0pt}}
    \pgfpathlineto{\pgfqpoint{0pt}{-.5\pgfarrowwidth}}
    \pgfpathclose
    \ifpgfarrowopen
      \pgfusepathqstroke
    \else
      \ifdim\pgfarrowlinewidth>0pt\pgfusepathqfillstroke\else\pgfusepathqfill\fi
    \fi
  },
  parameters = {
    \the\pgfarrowlinewidth,%
    \the\pgfarrowlength,%
    \the\pgfarrowwidth,%
    \the\pgfarrowinset,%
    \ifpgfarrowopen o\fi%
  },
}

% For correcting minipage width in stretched nodes
\newdimen\ipeminipagewidth
\def\ipestretchwidth#1{%
  \pgfmathsetlength{\ipeminipagewidth}{#1/\ipenodestretch}}

\tikzstyle{ipe import} = [
  % General ipe defaults
  x=1bp, y=1bp,
%
  % Nodes
  ipe node stretch/.store in=\ipenodestretch,
  ipe stretch normal/.style={ipe node stretch=1},
  ipe stretch normal,
  ipe node/.style={
    anchor=base west, inner sep=0, outer sep=0, scale=\ipenodestretch
  },
%
  % Use a special key for the mark scale, so that the default can be overriden.
  % (This doesn't happen with the scale= key; those accumulate.)
  ipe mark scale/.store in=\ipemarkscale,
  ipe mark tiny/.style={ipe mark scale=1.1},
  ipe mark small/.style={ipe mark scale=2},
  ipe mark normal/.style={ipe mark scale=3},
  ipe mark large/.style={ipe mark scale=5},
  ipe mark normal, % Set default
  ipe circle/.pic={
    \draw[line width=0.2*\ipemarkscale]
      (0,0) circle[radius=0.5*\ipemarkscale];
    \coordinate () at (0,0);
  },
  ipe disk/.pic={
    \fill (0,0) circle[radius=0.6*\ipemarkscale];
    \coordinate () at (0,0);
  },
  ipe fdisk/.pic={
    \filldraw[line width=0.2*\ipemarkscale]
      (0,0) circle[radius=0.5*\ipemarkscale];
    \coordinate () at (0,0);
  },
  ipe box/.pic={
    \draw[line width=0.2*\ipemarkscale, line join=miter]
      (-.5*\ipemarkscale,-.5*\ipemarkscale) rectangle
      ( .5*\ipemarkscale, .5*\ipemarkscale);
    \coordinate () at (0,0);
  },
  ipe square/.pic={
    \fill
      (-.6*\ipemarkscale,-.6*\ipemarkscale) rectangle
      ( .6*\ipemarkscale, .6*\ipemarkscale);
    \coordinate () at (0,0);
  },
  ipe fsquare/.pic={
    \filldraw[line width=0.2*\ipemarkscale, line join=miter]
      (-.5*\ipemarkscale,-.5*\ipemarkscale) rectangle
      ( .5*\ipemarkscale, .5*\ipemarkscale);
    \coordinate () at (0,0);
  },
  ipe cross/.pic={
    \draw[line width=0.2*\ipemarkscale, line cap=butt]
      (-.5*\ipemarkscale,-.5*\ipemarkscale) --
      ( .5*\ipemarkscale, .5*\ipemarkscale)
      (-.5*\ipemarkscale, .5*\ipemarkscale) --
      ( .5*\ipemarkscale,-.5*\ipemarkscale);
    \coordinate () at (0,0);
  },
%
  % Arrow sizes (for TikZ arrows)
  /pgf/arrow keys/.cd,
  ipe arrow normal/.style={scale=1},
  ipe arrow tiny/.style={scale=.4},
  ipe arrow small/.style={scale=.7},
  ipe arrow large/.style={scale=1.4},
  ipe arrow normal,
  /tikz/.cd,
%
  % Approximations to ipe arrows
  % Put in a style to allow to reset default scale when "ipe arrow normal" is
  % changed.  I think this is the only way, since all the parameters to arrows
  % are expanded when the tip is declared.
  ipe arrows/.style={
    ipe normal/.tip={
      ipe _pointed[length=1bp, width=.666bp, inset=0bp,
                   quick, ipe arrow normal]},
    ipe pointed/.tip={
      ipe _pointed[length=1bp, width=.666bp, inset=0.2bp,
                   quick, ipe arrow normal]},
    ipe linear/.tip={
      ipe _linear[length = 1bp, width=.666bp,
                  ipe arrow normal, quick]},
    ipe fnormal/.tip={ipe normal[fill=white]},
    ipe fpointed/.tip={ipe pointed[fill=white]},
    ipe double/.tip={ipe normal[] ipe normal},
    ipe fdouble/.tip={ipe fnormal[] ipe fnormal},
    % These should maybe use [bend], but that often looks bad unless it's on an
    % actual arc.
    ipe arc/.tip={ipe normal},
    ipe farc/.tip={ipe fnormal},
    ipe ptarc/.tip={ipe pointed},
    ipe fptarc/.tip={ipe fpointed},
  },
  ipe arrows, % Set default sizes
]

% I'm not sure how to do this in a .style, since the #args get confused.
\tikzset{
  rgb color/.code args={#1=#2}{%
    \definecolor{tempcolor-#1}{rgb}{#2}%
    \tikzset{#1=tempcolor-#1}%
  },
}
\definecolor{red}{rgb}{1,0,0}
\definecolor{blue}{rgb}{0,0,1}
\definecolor{brown}{rgb}{0.647,0.165,0.165}
\definecolor{darkblue}{rgb}{0,0,0.545}
\definecolor{darkcyan}{rgb}{0,0.545,0.545}
\definecolor{darkgray}{rgb}{0.663,0.663,0.663}
\definecolor{darkgreen}{rgb}{0,0.392,0}
\definecolor{darkmagenta}{rgb}{0.545,0,0.545}
\definecolor{darkorange}{rgb}{1,0.549,0}
\definecolor{darkred}{rgb}{0.545,0,0}
\definecolor{gold}{rgb}{1,0.843,0}
\definecolor{gray}{rgb}{0.745,0.745,0.745}
\definecolor{green}{rgb}{0,1,0}
\definecolor{lightblue}{rgb}{0.678,0.847,0.902}
\definecolor{lightcyan}{rgb}{0.878,1,1}
\definecolor{lightgray}{rgb}{0.827,0.827,0.827}
\definecolor{lightgreen}{rgb}{0.565,0.933,0.565}
\definecolor{lightyellow}{rgb}{1,1,0.878}
\definecolor{navy}{rgb}{0,0,0.502}
\definecolor{orange}{rgb}{1,0.647,0}
\definecolor{pink}{rgb}{1,0.753,0.796}
\definecolor{purple}{rgb}{0.627,0.125,0.941}
\definecolor{seagreen}{rgb}{0.18,0.545,0.341}
\definecolor{turquoise}{rgb}{0.251,0.878,0.816}
\definecolor{violet}{rgb}{0.933,0.51,0.933}
\definecolor{yellow}{rgb}{1,1,0}
\definecolor{black}{rgb}{0,0,0}
\definecolor{white}{rgb}{1,1,1}

\begin{center}
\begin{tikzpicture}[ipe import, font=\Large, scale=0.5]
  \draw[ultra thick]
    (80, 560) ellipse[x radius=64, y radius=144];
  \draw[ultra thick]
    (256, 560) ellipse[x radius=64, y radius=144];
  \draw[ultra thick]
    (432, 560) ellipse[x radius=64, y radius=144];
  \pic[ipe mark scale=3.0, red]
     at (416, 736) {ipe disk};
  \draw[red, thick, dashed]
    (64, 656) rectangle (112, 608);
  \draw[red, thick, dashed]
    (240, 656) rectangle (288, 608);
  \pic[ipe mark scale=3.0, red]
     at (272, 464) {ipe disk};
  \pic[ipe mark scale=3.0, red]
     at (448, 464) {ipe disk};
  \filldraw[shift={(416, 736)}, cm={1,0,0.6,0.6,(0,0)}, red, thin, fill opacity=0.3]
    (0, 0)
     .. controls (-160, -144) and (-200, -152) .. (-220, -156)
     .. controls (-240, -160) and (-240, -160) .. (-240, -160)
     .. controls (-240, -160) and (-240, -160) .. (-224, -157.3333)
     .. controls (-208, -154.6667) and (-176, -149.3333) .. (-146.6667, -149.3333)
     .. controls (-117.3333, -149.3333) and (-90.6667, -154.6667) .. (-77.3333, -157.3333)
     .. controls (-64, -160) and (-64, -160) .. (-64, -160)
     .. controls (-64, -160) and (-64, -160) .. (-69.3333, -154.6667)
     .. controls (-74.6667, -149.3333) and (-85.3333, -138.6667) .. (-74.6667, -112)
     .. controls (-64, -85.3333) and (-32, -42.6667) .. (-16, -21.3333)
     .. controls (0, 0) and (0, 0) .. (0, 0);
  \filldraw[shift={(416.001, 736.003)}, cm={1,0,0.3333,0.9444,(0,0)}, red, thin, fill opacity=0.3]
    (0, 0)
     .. controls (16, -176) and (-16, -232) .. (-32, -260)
     .. controls (-48, -288) and (-48, -288) .. (-48, -288)
     .. controls (-48, -288) and (-48, -288) .. (-40, -282.6667)
     .. controls (-32, -277.3333) and (-16, -266.6667) .. (13.3333, -266.6667)
     .. controls (42.6667, -266.6667) and (85.3333, -277.3333) .. (106.6667, -282.6667)
     .. controls (128, -288) and (128, -288) .. (128, -288)
     .. controls (128, -288) and (128, -288) .. (114.6667, -277.3333)
     .. controls (101.3333, -266.6667) and (74.6667, -245.3333) .. (49.3333, -178.6667)
     .. controls (24, -112) and (0, 0) .. (0, 0);
  \node[ipe node  , text=red]
     at (432, 736.004) {$v_\ast$};
  \node[ipe node  , text=red]
     at (256, 448) {$u_1$};
  \node[ipe node  , text=red]
     at (448, 448) {$u_2$};
  \node[ipe node  , text=red]
     at (32, 640) {$A_1$};
  \node[ipe node  , text=red]
     at (288, 640) {$A_2$};
  \node[ipe node  ]
     at (64, 384) {$V_1$};
  \node[ipe node  ]
     at (256, 384) {$V_2$};
  \node[ipe node  ]
     at (448, 384) {$V_3$};
  \node[ipe node  , text=red]
     at (304, 720) {$E_\ast$};
  \filldraw[shift={(80, 640)}, cm={1,0,-0.25,1,(0,0)}, black, fill opacity=0.3]
    (0, 0)
     .. controls (96, -32) and (136, -16) .. (156, -8)
     .. controls (176, 0) and (176, 0) .. (176, 0)
     .. controls (176, 0) and (176, 0) .. (189.3333, -13.3333)
     .. controls (202.6667, -26.6667) and (229.3333, -53.3333) .. (258.6667, -74.6667)
     .. controls (288, -96) and (320, -112) .. (336, -120)
     .. controls (352, -128) and (352, -128) .. (352, -128)
     .. controls (352, -128) and (352, -128) .. (264, -96)
     .. controls (176, -64) and (0, 0) .. (0, 0);
  \filldraw[black, thick, line cap=round, fill opacity=0.3]
    (80, 640)
     -- (80, 640);
  \node[ipe node  , fill opacity=0.8]
     at (160, 544) {$E_0$};
  % \node[ipe node  ]
  %    at (176, 320) {Figure 1: $\mathcal{H}$};
  \filldraw[black, thin, fill opacity=0.3]
    (448, 464)
     .. controls (320, 512) and (296, 488) .. (284, 476)
     .. controls (272, 464) and (272, 464) .. (272, 464)
     .. controls (272, 464) and (272, 464) .. (264, 480)
     .. controls (256, 496) and (240, 528) .. (208, 557.3333)
     .. controls (176, 586.6667) and (128, 613.3333) .. (104, 626.6667)
     .. controls (80, 640) and (80, 640) .. (80, 640)
     .. controls (80, 640) and (80, 640) .. (172, 596)
     .. controls (264, 552) and (448, 464) .. (448, 464);
\end{tikzpicture}
%\caption{Construction for non-vertex}
\end{center}
%%%%%%%%%%%%%%%%%%%%%%%%%%%%%%%%%%%%%%%%%%%%%%%% Tikz End    
    %
    Given disjoint sets $V_1,V_2,V_3$ of size $n/3$ and let $T_3(n,3)\coloneqq \left\{\{v_1,v_2,v_3\}\colon v_i\in V_i\textit{ for }i\in [3]\right\}.$
    For $i\in [2]$, fix a subset $A_i\subset V_i$ of size $\varepsilon_1 n$.
    Pick two vertices $u_1\in V_2\setminus A_2$ and $u_2\in V_3$.
    Let $$\mathcal{H}_1\coloneqq T_3(n,3) \setminus E_0$$ where
    $E_0\coloneqq \left\{ \{a_1, u_1, u_2\},\{a_1, a_2, v_3\} \colon a_1\in A_1,a_2\in A_2,v_3\in V_3\right\}.$
    Add a new vertex $v_\ast$ to $\mathcal{H}_1$ and obtain $3$-graph
    $$\mathcal{H}\coloneqq \mathcal{H}_1\cup E_\ast\cup \left\{\{v_\ast, u_1,u_2\}\right\},$$
    where
    $E_\ast\coloneqq \left\{\{a_1, a_2, v_\ast\}\colon a_1\in A_1,a_2\in A_2\right\}.$
    Next we going to show $\mathcal{H}$ is the required counterexample.
    \begin{claim}\label{CLIAM: H-not-3-partite}
        $\mathcal{H}-v_\ast\in \mathfrak{K}^3_{3}$, but $\mathcal{H}\notin \mathfrak{K}^3_{3}$.
    \end{claim}
    \begin{proof}
        Notices that $\mathcal{H}-v_\ast=\mathcal{H}_1\subset T_3(n,3)$, so we have $\mathcal{H}-v_\ast\in \mathfrak{K}^3_{3}$.
        
        Suppose $\mathcal{H}\in \mathfrak{K}^3_{3} $ with $S_1,S_2,S_3$ be its corresponding vertices partition.
        For each $i\in [3]$, we fix a vertex $w_i$ with $w_i\in V_i\setminus \left( A_1\cup A_2\cup \{u_1,u_2\}\right)$.
        Observe that $\{w_1,w_2,w_3\}\in \mathcal{H}$, then we may assume $w_i\in S_i$ for $i\in [3]$.

        Pick a vertex $a_1\in A_1$. 
        Notice $\{a_1,w_2,w_3\}\in \mathcal{H}$, then $a_1\in S_1$.
        And similarly, we conclude $u_1\in S_2,u_2\in S_3$.
        Therefore, $\{v_\ast, u_1,u_2\}\in \mathcal{H}$ implies $v_\ast\in S_1$.
        Recall $a_1\in S_1$ and note $\{v_\ast, a_1,a_2\}\in \mathcal{H}$ for some $a_2\in A_2$, then $v_\ast\notin S_1$ which leads to a contradiction.
    \end{proof}%CLAIM
    \begin{claim}\label{CLAIM: H-is-F5-free}
        $\mathcal{H}$ is $F_5$-free.
    \end{claim}
    \begin{proof}
        Suppose otherwise that there exists a copy of $F_5\coloneqq \{e_1,e_2,e_3\}$ in $\mathcal{H}$.
        Since $\mathcal{H}_1\subset T_3(n,3)$, $\mathcal{H}_1$ is $F_5$-free.
        Then $v_\ast\in V(F_5)$, and we assume $e_1\coloneqq \{v_\ast, a_1,a_2\},e_2\coloneqq \{v_\ast, b_1,b_2\}$ to be the edges containing $v_\ast$ in $F_5$.
        Due to $F_5\subset \mathcal{H}$, we obtain $e_1,e_2\in L_{\mathcal{H}}(v_\ast)=E_\ast\cup \left\{\{v_\ast, u_1,u_2\}\right\}.$
        By symmetry, we may assume that $e_1\in E_\ast$ with $a_1\in A_1,a_2\in A_2$.

        Suppose $|e_1\triangle e_2|=2$.
        By symmetry, we may assume $e_1\triangle e_2=\{a_2,b_2\}$ and $a_1=b_1$. 
        Notice that $e_2\neq \left\{\{v_\ast, u_1,u_2\}\right\}$ and recall that $e_2\in E_\ast\cup \left\{\{v_\ast, u_1,u_2\}\right\}$, then $e_2\in E_\ast$ and hence $b_2\in A_2$.
        The structure of $F_5$ shows $\{a_2,b_2\}=e_1\triangle e_2\subset e_3$.
        However, there is no $3$-edge in $\mathcal{H}$ that intersects the set $A_2$ twice, which implies $e_3\notin \mathcal{H}$, a contradiction.

        Suppose $|e_1\triangle e_2|=1$.
        The structure of $F_5$ implies
        $$e_3\in E'\coloneqq \left\{\{a_1,a_2,b_1\},\{a_1,a_2,b_2\},\{b_1,b_2,a_1\},\{b_1,b_2,a_2\}\right\}.$$
        By $e_2\in E_\ast\cup \left\{\{v_\ast, u_1,u_2\}\right\}$, we may assume $b_1=u_1,b_2=u_2$ or $b_1\in A_1,b_2\in A_2$.
        In both cases, one can easily check $E'\cap \mathcal{H}=\emptyset$.
        By above we obtain $e_3\notin \mathcal{H}$, a contradiction.
    \end{proof}%CLAIM
    By Lemma~\ref{LEMMA:Lp-degree-expression}, each vertex $v\in V(\mathcal{H})$ has
    \begin{align}\label{eq: expansion-of-degree-in-3-graph}
        d_{\mathcal{H},2,\frac{1}{2}}(v)&= \sum_{u\in \partial L_{\mathcal{H}}(v )}d_{\mathcal{H}}^{\frac{1}{2}}( \{u,v\})+\sum_{T\in L_{\mathcal{H}}(v )}\left(d^{\frac{1}{2}}_{\mathcal{H}} (T)-\left(d_{\mathcal{H}}(T)-d_{\mathcal{H}}(T\cup\{v\})\right)^{\frac{1}{2}}\right)\notag\\
        &= \sum_{u\in \partial L_{\mathcal{H}}(v )}d_{\mathcal{H}}^{\frac{1}{2}}( \{u,v\})+\sum_{T\in L_{\mathcal{H}}(v )}\left(d^{\frac{1}{2}}_{\mathcal{H}} (T)-\left(d_{\mathcal{H}}(T)-1\right)^{\frac{1}{2}}\right) 
    \end{align}
    \begin{claim}\label{CLAIM: v-star-degree}
        $d_{\mathcal{H},2,\frac{1}{2}} (v_\ast)\ge d(n).$
    \end{claim}
    \begin{proof}
        Let $$E_A\coloneqq  \left\{ \{a_1, a_2\} \colon a_1\in A_1,a_2\in A_2\right\}.$$
        Observe that $E_A\subset L_\mathcal{H}(v_\ast),$ and any $T\in E_A$ has $d_\mathcal{H}(T)=d_{\mathcal{H}}(T\cup\{v_\ast\})=1$, then \eqref{eq: expansion-of-degree-in-3-graph} continues as 
        \begin{align*}
            d_{\mathcal{H},2,\frac{1}{2}} (v_\ast)&\ge \sum_{T\in L_\mathcal{H}(v_\ast)}\left(d^{\frac{1}{2}}_\mathcal{H} (T)-\left(d_\mathcal{H}(T)-1\right)^{\frac{1}{2}}\right)\\
            &\ge  \sum_{T\in E_A}\left(d^{\frac{1}{2}}_\mathcal{H} (T)-\left(d_\mathcal{H}(T)-1\right)^{\frac{1}{2}}\right)\\
            &= \sum_{T\in E_A }\left(1^\frac{1}{2}-\left(1-1\right)^\frac{1}{2}\right)=|E_A|=\left(\varepsilon_1 n\right)^2\ge d(n) ,
        \end{align*}
        where the last inequality holds by $n\gg (1/\varepsilon_1)^{4}$ and Claim~\ref{CLAIM: eq-lower-bound}.
    \end{proof}%CLAIM
    \begin{claim}\label{CLAIM: min-degree-greater-than-dn}
        $\delta_{2,\frac{1}{2}} (\mathcal{H})\ge d(n).$
    \end{claim}
    \begin{proof}
        By Claim~\ref{CLAIM: v-star-degree}, it suffices to show that $d_{\mathcal{H},2,\frac{1}{2}} (v)\ge d(n)$ holds for all $v\in V(\mathcal{H})$ with $v\neq v_\ast$.
        The inequality \eqref{eq: expansion-of-degree-in-3-graph} continues as 
        \begin{align}\label{eq: construction-lower-bound-for-v-1}
            d_{\mathcal{H},2,\frac{1}{2}}(v)\ge \sum_{u\in \partial L_{\mathcal{H}_1}(v )}d_{\mathcal{H}}^{\frac{1}{2}}( \{u,v\})+\sum_{T\in L_{\mathcal{H}_1}(v )}\left(d^{\frac{1}{2}}_{\mathcal{H}} (T)-\left(d_{\mathcal{H}}(T)-1\right)^{\frac{1}{2}}\right). 
        \end{align}
        Observe that $|\partial L_{\mathcal{H}_1}(v)|\ge \dfrac{2n}{3}-\varepsilon_1 n\ge 2(1-3\varepsilon_1)\dfrac{n}{3}$ and each $u\in \partial L_{\mathcal{H}_1}(v)$ has $d_{\mathcal{H}}(\{u,v\})\ge \frac{n}{3}-\varepsilon_1 n= (1-3\varepsilon_1)\frac{n}{3}.$
        Then
        $$ \sum_{u\in \partial L_{\mathcal{H}_1}(v )}d_{\mathcal{H}}^{\frac{1}{2}}( \{u,v\})\ge |\partial L_{\mathcal{H}_1}(v)|\cdot \left((1-3\varepsilon_1)\frac{n}{3}\right)^{\frac{1}{2}}\ge  2 \left((1-3\varepsilon_1)\frac{n}{3}\right)^{\frac{3}{2}}.  $$
        Fact~\ref{FACT:inequality-a}~\ref{FACT:inequality-a-2} implies that $(1-3\varepsilon_1)^{\frac{3}{2}}\ge 1-\frac{9}{2}\varepsilon_1\ge 1-10\varepsilon_1$.
        Therefore, the inequality continues as
        \begin{equation}\label{eq: construction-lower-bound-for-v-2}
            \sum_{u\in \partial L_{\mathcal{H}_1}(v )}d_{\mathcal{H} }^{\frac{1}{2}}( \{u,v\}) \ge 2 \left((1-3\varepsilon_1)\frac{n}{3}\right)^{\frac{3}{2}}\ge 2(1-10\varepsilon_1)\left(\dfrac{n}{3}\right)^{\frac{3}{2}}. 
        \end{equation}
        Since every $T\in L_{\mathcal{H}_1}(v)$ has $d_{\mathcal{H}}(T)\in [1, \frac{n}{3}]$, it follows from Fact~\ref{FACT:inequality-a}~\ref{FACT:inequality-a-1} that
        % $$1-\left(1-\dfrac{1}{d_{\mathcal{H}}(T)}\right)^{\frac{1}{2}}\ge \frac{1}{2}\cdot \dfrac{1}{d_{\mathcal{H}}(T)},$$
        % implying that
        \begin{align}
            d^{\frac{1}{2}}_{\mathcal{H}}(T)-(d_{\mathcal{H}}(T)-1)^{\frac{1}{2}}&= d^{\frac{1}{2}}_{\mathcal{H}}(T)\cdot \left(1-\left(1-\dfrac{1}{d_{\mathcal{H}}(T)}\right)^{\frac{1}{2}}\right)\notag\\
            &\ge d^{\frac{1}{2}}_{\mathcal{H}}(T)\cdot  \dfrac{1}{2d_{\mathcal{H}}(T)}
            = \frac{1}{2}\cdot d^{-\frac{1}{2}}_{\mathcal{H}}(T)\ge\dfrac{1}{2}\left(\dfrac{n}{3}\right)^{-\frac{1}{2}}\notag.
        \end{align}  
        Observe that $|L_{\mathcal{H}_1}(v)|\ge \left(\frac{n}{3}\right)^2-\varepsilon_1 n^2-1\ge (1-10\varepsilon_1)\left(\frac{n}{3}\right)^2.$
        Therefore,
        \begin{align}\label{eq: construction-lower-bound-for-v-3}
            \sum_{T\in L_{\mathcal{H}_1}(v)}d^{\frac{1}{2}}_{\mathcal{H}}(T)-(d_{\mathcal{H}}(T)-1)^{\frac{1}{2}}&\ge |L_{\mathcal{H}_1}(v)|\cdot \dfrac{1}{2} \left(\dfrac{n}{3}\right)^{-\frac{1}{2}}\ge (1-10\varepsilon_1)\dfrac{1}{2}\left(\dfrac{n}{3}\right)^{\frac{3}{2}}.
        \end{align}

        By \eqref{eq: construction-lower-bound-for-v-2} and \eqref{eq: construction-lower-bound-for-v-3}, inequality \eqref{eq: construction-lower-bound-for-v-1} continues as
        $$d_{\mathcal{H},2,\frac{1}{2}}(v)\ge  (1-10\varepsilon_1)2\left(\dfrac{n}{3}\right)^{\frac{3}{2}}+  (1-10\varepsilon_1)\dfrac{1}{2}\left(\dfrac{n}{3}\right)^{\frac{3}{2}}=(1-10\varepsilon_1)\dfrac{5}{2}\left(\dfrac{n}{3}\right)^{\frac{3}{2}}\ge d(n),$$
        where the last inequality holds by Claim~\ref{CLAIM: eq-lower-bound}.
    \end{proof}%CLAIM
    Claim~\ref{CLAIM: H-is-F5-free},~\ref{CLIAM: H-not-3-partite} and ~\ref{CLAIM: min-degree-greater-than-dn} shows $\mathcal{H}$ is a counterexample, which proves that $F_5$ is not $(2,1/2)$-vertex-extendable with respect to $\mathfrak{K}^3_{3}$.
\end{proof}%THM
%%%%%%%%%%%%%%%%%%%%%%%%%%%%%%%%%%%%
\section{Existence of the limit for $p < 1$}\label{APPENDIX:SEC:limit}
\begin{proposition}\label{PROP:limit-exits-small-p}
    Let $r\ge 2$ be an integer and $\mathcal{F}$ be a family of $r$-graphs.
    For every $t\in [r-1]$ and $p\in (0,1]$, the limit $\lim\limits_{n\to  \infty} \frac{\mathrm{ex}_{t,p}(n,\mathcal{F})}{n^{t+(r-t)p}}$ exists.
\end{proposition}

We will use the following standard theorems from Analysis in the proof of Proposition~\ref{PROP:limit-exits-small-p}. 

\begin{theorem}[The Monotone Convergence Theorem]\label{THM: monotone-convergence-theorem}
    If a sequence of real numbers is decreasing and bounded below, then it converges.
\end{theorem}
\begin{fact}[The $p$-series Test]\label{FACT: p-series-test}
    If $\alpha<-1$, then sequence $\left(\sum_{m=1}^n m^\alpha\right)_{n=1}^\infty$ converges.
\end{fact}
\begin{lemma}\label{LEMMA: almost-decrease-sequence-converge}
    If $(a_n)_{n=1}^\infty$ is a bounded sequence with $a_{n-1}\ge a_n-n^\alpha$ where $\alpha<-1$ for all $n\in \mathbb{N}^{\ge 2}$, then the sequence $(a_n)_{n=1}^\infty$ converges.
\end{lemma}
\begin{proof}
    Given the required sequence $(a_n)_{n=1}^\infty$.
    Let $b_n\coloneqq a_n-\sum_{m=1}^{n}m^\alpha$ for all $n\in \mathbb{N}^{\ge 1}$.
    For $n\in \mathbb{N}^{\ge 2}$ we have $a_{n-1}\ge a_n-n^\alpha$, so $b_{n-1}-b_n=a_{n-1}-a_n+n^\alpha\ge 0$.
    Imply that $(b_n)_{n=1}^\infty$ is decreasing.
    By Fact~\ref{FACT: p-series-test}, $\left(\sum_{m=1}^n m^\alpha\right)_{n=1}^\infty$ converges, thus bounded.
    Combining with $(a_n)_{n=1}^\infty$ being bounded, we obtain $(b_n)_{n=1}^\infty$ is also bounded.
    By Theorem~\ref{THM: monotone-convergence-theorem}, $(b_n)_{n=1}^\infty$ converges.
    By $(b_n)_{n=1}^\infty$ and $\left(\sum_{m=1}^n m^\alpha\right)_{n=1}^\infty$ both converge, the sequence $(a_n)_{n=1}^\infty$ also converges.
\end{proof}%LEMMA
\begin{proof}[Proof of Proposition~\ref{PROP:limit-exits-small-p}]
    For $n\ge 1$, let $\pi_n\coloneqq \frac{\mathrm{ex}_{t,p}(n,\mathcal{F})}{n^k}$ where $k\coloneqq t+(r-t)p.$  
    Given a sufficiently small $\varepsilon>0$ and a sufficiently large $N$ with $ N\gg \frac{1}{\varepsilon}$.
    
    Given an $n\ge N$ let $\mathcal{H}$ be an $n$-vertex $\mathcal{F}$-free $r$-graph with $\lVert \mathcal{H}\rVert_{t,p}=\mathrm{ex}_{t,p}(n,\mathcal{F}).$
    By Lemma~\ref{LEMMA:tail-estimate-Lp-degree-sum} and the Pigeonhole Principle, there exists a vertex $v_\ast\in V(\mathcal{H})$ and a real number $\delta_p\in (0,1)$ with 
    $$d_{\mathcal{H},t,p}(v_\ast)\le  \dfrac{k\cdot \lVert \mathcal{H}\rVert_{t,p}+\varepsilon n^{k-\delta_p}}{n} =k\cdot \dfrac{\mathrm{ex}_{t,p}(n,\mathcal{F})}{n} +\varepsilon n^{k+\alpha} ,$$
    where $\alpha\coloneqq -\delta_p-1$.
    Note $\alpha\in (-2,-1)$.
    Notice that $(n-1)$-vertex $r$-graph $\mathcal{H}_1\coloneqq \mathcal{H}-v_\ast$ keeps $\mathcal{F}$-free, then 
    \begin{equation}\label{equ:rate-increase}
        \mathrm{ex}_{t,p}(n-1,\mathcal{F})\ge \lVert \mathcal{H}_1\rVert_{t,p}=\lVert \mathcal{H} \rVert_{t,p} - d_{\mathcal{H},t,p}(v_\ast)\ge \mathrm{ex}_{t,p}(n,\mathcal{F})- k\cdot \dfrac{\mathrm{ex}_{t,p}(n,\mathcal{F})}{n} -\varepsilon n^{k+\alpha}.
    \end{equation}
    \begin{claim}\label{CLAIM: general-decrease}
        $\pi_{n-1}\ge \pi_n -{n^\alpha }.$
    \end{claim}
    \begin{proof}
        Note $k\ge 0$ and $\frac{1}{n}\in [0,1]$.
        It follows from~\eqref{equ:rate-increase} and Fact~\ref{FACT:inequality-a}~\ref{FACT:inequality-a-2} that
        \begin{align*}
            \dfrac{\mathrm{ex}_{t,p}(n-1,\mathcal{F})}{(n-1)^k}
            &\ge \dfrac{\mathrm{ex}_{t,p}(n,\mathcal{F})}{(n-1)^k}-k\cdot \dfrac{\mathrm{ex}_{t,p}(n,\mathcal{F})}{n(n-1)^k}-\dfrac{\varepsilon n^{k+\alpha}}{(n-1)^k} \notag\\
            &\ge \dfrac{\mathrm{ex}_{t,p}(n,\mathcal{F})}{n^k}+\left( \dfrac{\mathrm{ex}_{t,p}(n,\mathcal{F})}{(n-1)^k}-\dfrac{\mathrm{ex}_{t,p}(n,\mathcal{F})}{n^k}-k\cdot \dfrac{\mathrm{ex}_{t,p}(n,\mathcal{F})}{n(n-1)^k}\right)- 2\varepsilon n^\alpha \notag\\
            &= \dfrac{\mathrm{ex}_{t,p}(n,\mathcal{F})}{n^k}+\left( 1-\left(1-\dfrac{1}{n}\right)^k-k\cdot\dfrac{1}{n}\right)\dfrac{\mathrm{ex}_{t,p}(n,\mathcal{F})}{(n-1)^k}- 2\varepsilon n^\alpha \notag\\
            &\overset{\text{Fact}~\ref{FACT:inequality-a}~\ref{FACT:inequality-a-2}}{\ge} \dfrac{\mathrm{ex}_{t,p}(n,\mathcal{F})}{n^k}-\left(\dfrac{k}{n}\right)^2\cdot \dfrac{\mathrm{ex}_{t,p}(n,\mathcal{F})}{(n-1)^k}-2\varepsilon n^\alpha\notag\\
            &\ge \dfrac{\mathrm{ex}_{t,p}(n,\mathcal{F})}{n^k}-k^2n^{-2}\cdot 2-2\varepsilon n^\alpha\ge \dfrac{\mathrm{ex}_{t,p}(n,\mathcal{F})}{n^k}-n^\alpha.\notag
        \end{align*}
    \end{proof}%CLAIM
    Note sequence $\left(\pi_n\right)_{n=N}^\infty$ is bounded.
    By Lemma~\ref{LEMMA: almost-decrease-sequence-converge}, sequence $\left(\pi_n\right)_{n=N}^\infty$ converges.
    Hence $\lim\limits_{n\to  \infty} \frac{\mathrm{ex}_{t,p}(n,\mathcal{F})}{n^{t+(r-t)p}}$ exists.
\end{proof}%PROP
%%%%%%%%%%%%%%%%%%%%%%%%
\section{Estimations for $\alpha_{k}$}\label{APPENDIX:SEC:alpha-k}
\begin{lemma}\label{LEMMA:F5-Lp-p-small-calculations}
  Let $\ell(x)$ denote the piecewise linear function that connects points
  in\\$\left\{\left(\frac{k-1}{k}, \frac{k-1}{k^2}\right) \colon k \in
    \mathbb{N}^{+}\right\}$. That is,
  \begin{align*}
    \ell(x)
    \coloneqq -\frac{k^2-k-1}{k(k+1)}x+\frac{k-1}{k+1}
    \quad\text{for}\quad x \in \left[\frac{k-1}{k}, \frac{k}{k+1}\right]. 
  \end{align*}
  For every real number~$p \in [0, 1)$, let also~$g(p) \coloneqq \max\left\{\ell(x)^{p} \cdot
    x^{1-p} \colon x\in [0,1)\right\}$. Then for every~~$k \in \mathbb{N}_{\ge
    3}$, and~$p \in \left[ \frac{1}{k - 1}, \frac{k - 1}{k(k - 2)} \right)$,
  \begin{equation}\label{eq:gp}
    g(p) 
    = \ell\left(\frac{k-1}{k}\right)^{p} \cdot \left(\frac{k-1}{k}\right)^{1-p}
    = \frac{k-1}{k^{p+1}}.
  \end{equation}
\end{lemma}
\begin{proof}
  For every~$i \in \mathbb{N}^{+}$, let~$\ell_i(x) \coloneqq
  -\frac{i^2-i-1}{i(i+1)}x+\frac{i-1}{i+1}$ for every~$x \in [0,
  1]$. Also, for a real number~$p \in [0, 1)$ and~$i$ as above, let~$t_{i, p}(x) \coloneqq
  \ell_i(x)^p x^{1 - p}$, and let~$g_i(p) \coloneqq \max \left\{ t_{i, p}(x) : x \in
    \left[\frac{i - 1}{i}, \frac{i}{i + 1}\right] \right\}$. Clearly
  \begin{equation}
    \label{eq:ell-i-domain}
    \ell(x) = \ell_i(x) \quad \text{for } x \in \left[\frac{i - 1}{i}, \frac{i}{i + 1}\right].
  \end{equation}
  Thus, note that
  \begin{equation}
    \label{eq:max-over-small-pieces}
    g(p) = \max_{i \in \mathbb{N}^{+}} g_i(p).
  \end{equation}
  Also, for~$i \in \mathbb{N}_{\geq 2}$ and~$p \in [0, 1)$, we have that
  \begin{align}
    \label{eq:piece-intersection}
    t_{i - 1, p}\left(\frac{i - 1}{i}\right)
    &= \ell_{i - 1} \left(\frac{i -
      1}{i}\right)^p \cdot \left(\frac{i - 1}{i}\right)^{1 - p} \nonumber & \\
    &= \ell \left(\frac{i -
      1}{i}\right)^p \cdot \left(\frac{i - 1}{i}\right)^{1 - p} &
                                                              \quad\text{by~\eqref{eq:ell-i-domain}}
                                                              \nonumber \\
    &= \ell_{i} \left(\frac{i -
      1}{i}\right)^p \cdot \left(\frac{i - 1}{i}\right)^{1 - p} &
                                                              \quad\text{by~\eqref{eq:ell-i-domain}}
                                                              \nonumber \\
    &= t_{i, p}\left(\frac{i - 1}{i}\right).
  \end{align}
  For~$p \in (0, 1)$, let
  \[p^{(1)} \coloneqq \frac{\sqrt{4p^2 + 1} + 1}{2p}, \quad \text{ and } \quad
    p^{(2)} \coloneqq \frac{\sqrt{4p + 1} + 1}{2p},\] and note that~$1 < p^{(1)} <
  p^{(2)}$.

  \begin{claim}
    Fix~$p \in (0, 1)$, then for every~$i \in \mathbb{N}^{+}$,
    \begin{enumerate}[label=(\alph*)]
    \item\label{item:k-increasing} if~$i < p^{(1)}$,
      then~$t_{i,p}$ is increasing over~$\left[\frac{i-1}{i},
        \frac{i}{i+1}\right]$, hence,~$g_i(p) = t_{i, p}\left(\frac{i}{i+1}\right)$.
    \item\label{item:i-decreasing} If~$i > p^{(2)}$,
      then~$t_{i,p}$ is decreasing over~$\left[\frac{i-1}{i},
        \frac{i}{i+1}\right]$, hence,~$g_i(p) = t_{i, p}\left(\frac{i-1}{i}\right)$.
    \end{enumerate}
  \end{claim}
  \begin{proof}
    We analyze the first derivative of the function~$t_{i,
      p}$ to show that it is positive on the first case and negative on the
    second. For any~$i \in \mathbb{N}^{+}$, let~$a_i \coloneqq \frac{i^2 - i - 1}{i (i
      + 1)}$ and~$b_i \coloneqq \frac{i - 1}{i + 1}$ so that~$\ell_i (x) = -a_ix +
    b$. Fix~$p \in (0, 1)$ and~$i \in \mathbb{N}^{+}$. We then have that
    \begin{equation}
      \label{eq:derivative_tipx}
      D t_{i, p}(x) = \ell_i(x)^{p - 1}\left(\frac{b_i (1 - p)}{x^p} - a_i x^{1 -
          p}\right).
    \end{equation}
    Also, the following equality will be useful later
    \begin{equation}
      \label{eq:helper_derivative_point}
      \frac{b_i (1 - p)}{\left(\frac{b_i(1 - p)}{a_i}\right)^p} - a_i
      \left(\frac{b_i(1 - p)}{a_i}\right)^{1 - p} = a_i^p b_i^{1 - p}(1 - p)^{1 - p}
      - a_i^p b_i^{1 - p}(1 - p)^{1 - p} = 0.
    \end{equation}
    To prove~\ref{item:k-increasing}, suppose that~$i <
    p^{(1)}$. Assume first that~$i = 1$ and observe that~$a_1 = -1/2$ and~$b_1 =
    0$. Then~$t_{1, p}(x) = (-a_1x + b_1)^p x^{1 - p} = \left(\frac{x}{2}\right)^p
    x^{1 - p} =
    \frac{x}{2^p}$, which is increasing in its entire domain. Assume now that~$2 \leq i
    < p^{(1)}$. Then~$(2pi - 1)^2 < 4p^2 + 1$, which implies that~$-i < -p(i^2 -
    1)$. By adding~$i^2 - 1$ to both sides of the inequality we have that~$i^2 - i -
    1 < i^2 - 1 - p(i^2 - 1) = (1 - p)(i - 1)(i + 1)$. Multiplying by~$\frac{i}{(i^2
      - i + 1)(i + 1)}$ to both sides results in~$\frac{i}{i + 1} < \frac{i - 1}{i +
      1} \cdot \frac{i(i + 1)}{i^2 - i + 1} \cdot (1 - p) = \frac{b_i (1 -
      p)}{a_i}$. Suppose also that~$x \in \left[\frac{i-1}{i},
      \frac{i}{i+1}\right]$. Then~$x \leq \frac{i}{i+1} < \frac{b_i(1 -
      p)}{a_i}$. By using this upper bound
    on~$x$ in~\eqref{eq:derivative_tipx} we get
    \begin{equation*}
      D t_{i, p}(x) > \ell_i(x)^{p - 1}\left(\frac{b_i (1 - p)}{\left(\frac{b_i(1 - p)}{a_i}\right)^p} - a_i
        \left(\frac{b_i(1 - p)}{a_i}\right)^{1 - p}\right)
      \stackrel{\eqref{eq:helper_derivative_point}}{=} 0,
    \end{equation*}
    which proves~\ref{item:k-increasing}.

    To prove~\ref{item:i-decreasing}, suppose that~$i > p^{(2)}$. This implies
    that~$i \geq 2$ and that~$(2pi - i)^2 > 4p + 1$, which implies
    that~$-i - 1 > -pi^2$. By adding~$i^2$ to both sides of the inequality we have
    that~$i^2 - i - 1 > i^2 - pi^2 = (1 - p)i^2$. Multiplying
    by~$\frac{i - 1}{i(i^2 - i - 1)}$ to both sides results
    in~$\frac{i - 1}{i} > \frac{i - 1}{i + 1} \cdot \frac{i (i + 1)}{i^2 - i - 1} \cdot(1 -
    p) = \frac{b_i (1 - p)}{a_i}$. Suppose also
    that~$x \in \left[\frac{i-1}{i}, \frac{i}{i+1}\right]$.
    Then~$x \geq \frac{i - 1}{i} > \frac{b_i(1 - p)}{a_i}$. By using this lower bound
    on~$x$ in~\eqref{eq:derivative_tipx} we get
    \begin{equation*}
      D t_{i, p}(x) < \ell_i(x)^{p - 1}\left(\frac{b_i (1 - p)}{\left(\frac{b_i(1 - p)}{a_i}\right)^p} - a_i
        \left(\frac{b_i(1 - p)}{a_i}\right)^{1 - p}\right)
      \stackrel{\eqref{eq:helper_derivative_point}}{=} 0,
    \end{equation*}
    which proves~\ref{item:i-decreasing} and completes the proof of the claim.
  \end{proof}

  Let~$k \in \mathbb{N}_{\ge 3}$,
  and~$p \in \left[ \frac{1}{k - 1}, \frac{k - 1}{k(k - 2)} \right)$ be fixed. We will
  first show that
  \begin{equation}
    \label{eq:ps-between-ks}
    k - 1 < p^{(1)} < p^{(2)} < k.
  \end{equation}

  For the leftmost inequality, the fact that~$p < \frac{k - 1}{k(k - 2)}$ implies
  that~$pk^2 - 2pk - k + 1 < 0$. Multiplying both sides of the inequality by~$4p$
  results in~$4p^2k^2 - 8p^2k - 4pk + 4p < 0$. We now add~$4p^2 + 1$ to both sides to
  get~$4p^2k^2 + 4p^2 + 1 - 8p^2k - 4pk + 4p = \left(2pk - 2p - 1\right)^2 < 4p^2 +
  1$, which after taking square roots from both sides implies
  that~$2pk - 2p < \sqrt{4p^2 + 1} + 1$. Dividing both sides by~$2p$ we
  obtain~$k - 1 < \frac{\sqrt{4p^2 + 1} + 1}{2p} = p^{(1)}$, as required.

  For the rightmost inequality, observe
  that~$p \geq \frac{1}{k - 1} = \frac{k + 1}{k^2 - 1} > \frac{k + 1}{k^2}$. This
  implies that~$pk^2 - k - 1 > 0$. Multiplying both sides of the inequality by~$4p$
  results in~$4p^2k^2 - 4pk - 4p > 0$. We now add~$4p + 1$ to both sides to
  get~$4p^2k^2 - 4pk + 1 = \left(2pk - 1\right)^2 > 4p + 1$, which after taking
  square roots from both sides implies that~$2pk > \sqrt{4p + 1} + 1$, and after
  dividing both sides by~$2p$ we obtain~$k > \frac{\sqrt{4p + 1} + 1}{2p} = p^{(2)}$,
  which completes the proof of~\eqref{eq:ps-between-ks}.

  We will now determine~$g(p)$ by using \eqref{eq:max-over-small-pieces}, and considering the
  cases~$i \le k - 1$ and~$i \geq k$ separately. In both cases, we analyze the
  sequence~$g_i(p)$, showing that it is non-decreasing in the former case, and
  non-increasing in the latter.

  For~$2 \le i \le k - 1$, we have that
  \begin{equation*}
    g_{i - 1}(p) \stackrel{\eqref{eq:ps-between-ks}\ref{item:k-increasing}}{=} t_{i-1, p}\left(\frac{i - 1}{i}\right)
    \stackrel{\eqref{eq:piece-intersection}}{=} t_{i, p} \left(\frac{i - 1}{i}\right) \leq g_i(p),
  \end{equation*}
  where the last inequality follows from the definition of~$g_i(p)$. Therefore
  \begin{align}
    \label{eq:first-subdomain}
    \max_{i = 1}^{k - 1} g_i(p) = g_{k - 1}(p) \stackrel{\eqref{eq:ps-between-ks}\ref{item:k-increasing}}{=} t_{k - 1,
    p}\left(\frac{k - 1}{k}\right).
  \end{align}
  For~$i \ge k + 1$, by the definition of~$g_{i - 1}(p)$ we have that
  \begin{equation*}
    g_{i - 1}(p) \ge t_{i - 1, p}\left(\frac{i - 1}{i}\right) \stackrel{\eqref{eq:piece-intersection}}{=} t_{i, p}\left(\frac{i - 1}{i}\right) \stackrel{\eqref{eq:ps-between-ks}\ref{item:i-decreasing}}{=} g_i(p),
  \end{equation*}
  and therefore
  \begin{align}
    \label{eq:second-subdomain}
    \max_{i \ge k} g_i(p) = g_{k}(p) \stackrel{\eqref{eq:ps-between-ks}\ref{item:i-decreasing}}{=} t_{k, p}\left(\frac{k -
    1}{k}\right) \stackrel{\eqref{eq:piece-intersection}}{=} t_{k - 1, p}\left(\frac{k - 1}{k}\right).
  \end{align}
  We now combine these results to obtain
  \begin{align}
    g(p)
    &= \max\left\{ \max_{i = 1}^{k - 1} g_i(p), \max_{i \ge k} g_i(p)
      \right\} & \quad\text{by~\eqref{eq:max-over-small-pieces}}\nonumber & \\
    &= t_{k - 1, p}\left(\frac{k - 1}{k}\right) \nonumber &
                                                            \quad\text{by~\eqref{eq:first-subdomain} and~\eqref{eq:second-subdomain}}\\
    &= \ell_{k - 1}\left(\frac{k - 1}{k}\right)^p \cdot \left(\frac{k - 1}{k}\right)^{1 - p}
      \nonumber & \quad\text{by definition}\\
    &= \ell\left(\frac{k - 1}{k}\right)^p \cdot \left(\frac{k - 1}{k}\right)^{1 - p} & \quad\text{by~\eqref{eq:ell-i-domain}} \nonumber.
  \end{align}
  This completes the proof.
\end{proof}

%%%%%%%%%%%%%%%%%%%%%%%%%%
\section{Proof for Corollary~\ref{CORO:Lp-degree-expression-Lagrangian}}\label{APPENDIX:proof:CORO:Lp-degree-expression-Lagrangian}
\begin{corollary}\label{APPENDIX:CORO:Lp-degree-expression-Lagrangian}
    Let $r > t \ge 1$ be integers and $p > 1$ be a real number. 
    Suppose that $V_1 \cup \cdots \cup V_m = [n]$ is a partition and $\mathcal{H} = \mathcal{G}(V_1, \ldots, V_m)$ is a blowup of an $m$-vertex $r$-graph $\mathcal{G}$. Then for every $i \in [m]$ and $v\in V_i$, 
    \begin{align*}
        D_i L_{\mathcal{G},t,p}(x_1, \ldots, x_m) - o(1)
        \le \frac{d_{t,p,\mathcal{H}}(v)}{n^{t-1+p(r-t)}}
        \le D_i L_{\mathcal{G},t,p}(x_1, \ldots, x_m).  
    \end{align*}
    where $x_i \coloneqq |V_i|/n$ for $i\in [m]$. 
\end{corollary}
\begin{proof}[Proof of Corollary~\ref{APPENDIX:CORO:Lp-degree-expression-Lagrangian}]
    For simplicity, let us assume that $V(\mathcal{G}) = [m]$. 
    By relabelling the sets $V_1, \ldots, V_m$ we may assume that $v\in V_1$. 
    First, it follows from the definition that 
    \begin{align*}
        & \quad D_1 L_{\mathcal{G},t,p}(x_1, \ldots, x_m) \\
        & = \sum_{S \in \partial_{r-t}L_{\mathcal{G}}(1)} 
            x_{S} \left(\sum_{I\in L_{\mathcal{G}}(S\cup\{v\})} x_{I} \right)^p 
            + \sum_{T \in \partial_{r-t-1}L_{\mathcal{G}}(1)} 
            x_T \cdot p \left(\sum_{I\in L_{\mathcal{G}}(T)} x_{I}\right)^{p-1} \frac{\sum_{I\in L_{\mathcal{G}}(T), 1 \in I} x_{I}}{x_1} \\
        & = \sum_{S \in \partial_{r-t}L_{\mathcal{G}}(1)} 
            x_{S} \left(\sum_{I\in L_{\mathcal{G}}(S\cup\{v\})} x_{I} \right)^p 
            + \sum_{T \in \partial_{r-t-1}L_{\mathcal{G}}(1)} 
            x_T \cdot p \left(\sum_{I\in L_{\mathcal{G}}(T)} x_{I}\right)^{p-1} \sum_{J\in L_{\mathcal{G}}(T\cup \{1\})} x_{J}.  
    \end{align*}
    Since $\mathcal{H}$ is a blowup of $\mathcal{G}$, 
    \begin{align*}
        \sum_{S \in \partial_{r-t}L_{\mathcal{H}}(v)} 
            d_{\mathcal{H}}^{p}(S\cup\{v\})
        & = \sum_{S \in \partial_{r-t}L_{\mathcal{G}}(1)} 
            x_{S} \cdot n^{t-1} \left(\sum_{I\in L_{\mathcal{G}}(S\cup\{v\})} x_{I}  \cdot  n^{r-t}\right)^p \\
        & = \sum_{S \in \partial_{r-t}L_{\mathcal{G}}(1)} 
            x_{S} \left(\sum_{I\in L_{\mathcal{G}}(S\cup\{v\})} x_{I} \right)^p n^{t-1+p(r-t)}, 
    \end{align*}
    and 
    \begin{align*}
        & \quad \sum_{T \in \partial_{r-t-1}L_{\mathcal{H}}(v)} 
            \left( d_{\mathcal{H}}^{p}(T) - \left(d_{\mathcal{H}}(T) - d_{\mathcal{H}}(T\cup \{v\})\right)^{p} \right) \\
        & = \sum_{T \in \partial_{r-t-1}L_{\mathcal{G}}(1)} 
            x_T  \cdot n^{t}\left( \left(\sum_{I\in L_{\mathcal{G}}(T)} x_{I}  \cdot  n^{r-t}\right)^p - \left(\sum_{I\in L_{\mathcal{G}}(T)} x_{I}  \cdot  n^{r-t} 
            - \sum_{J\in L_{\mathcal{G}}(T\cup \{1\})} x_{J}  \cdot  n^{r-t-1} \right)^{p} \right) \\
        & = \sum_{T \in \partial_{r-t-1}L_{\mathcal{G}}(1)} 
            x_T \left(p \left(\sum_{I\in L_{\mathcal{G}}(T)} x_{I}\right)^{p-1} \sum_{J\in L_{\mathcal{G}}(T\cup \{1\})} x_{J} \right) n^{t-1+p(r-t)} - o(n^{t-1+p(r-t)}),   
    \end{align*}
    where the error term $o(n^{t-1+p(r-t)})$ can be assumed to be non-negative due to~\eqref{equ:FACT:inequality-a-4}. 
    Combining with Lemma~\ref{LEMMA:Lp-degree-expression}, we obtain 
    \begin{align*}
        d_{t,p,\mathcal{H}}(v) 
        & = \sum_{S \in \partial_{r-t}L_{\mathcal{H}}(v)} 
            d_{\mathcal{H}}^{p}(S\cup \{v\})
            + \sum_{T \in \partial_{r-t-1}L_{\mathcal{H}}(v)} 
            \left( d_{\mathcal{H}}^{p}(T) - \left(d_{\mathcal{H}}(T) - d_{\mathcal{H}}(T\cup \{v\})\right)^{p} \right) \\
        & = D_1 L_{\mathcal{G},t,p}(x_1, \ldots, x_m) \cdot n^{t-1+p(r-t)} - o(n^{t-1+p(r-t)}), 
    \end{align*}
    completing the proof of Corollary~\ref{CORO:Lp-degree-expression-Lagrangian}. 
\end{proof}
%%%%%%%%%%%%%%%%%%%%%%%%%%%%%%%%%
\end{appendix}
%%%%%%%%%%%%%%%%%%%%%%%%%%%%%%%%%%%%%%%
%%%%%%%%%%%%%%%%%%%%%%%%%%%%%%%%%%%%%%%
\end{document}